%% file: main.tex
\documentclass[a4paper]{article}
\usepackage{preamble}

\usepackage[style=alphabetic, 
maxbibnames=3, backend=biber, 
doi=false, isbn=false, url=false]{biblatex}
\title{Arity Approximation of $\infty$-Operads}
\author{Shaul Barkan}

\begin{document}
	\maketitle
	\begin{abstract}
		Let $\mathcal{C}$ be an $\infty$-category all of whose mapping spaces are $n$-truncated. We prove that when considering $\mathbb{E}_d$-monoids in $\mathcal{C}$, all coherence diagrams of arity $>n+3$ are redundant. 
		More generally, for an $\infty$-operad $\mathcal{O}$ we bound the arity of the relevant coherence diagrams in terms of the connectivity of certain operadic partition complexes associated to $\mathcal{O}$.
	\end{abstract}
	\setcounter{page}{1}
	\tableofcontents
        \input{sections/introduction.tex}
	\input{sections/algebraic-patterns.tex}
	\input{sections/platonic-patterns.tex}

	\input{sections/operads.tex}
	\input{sections/1-reduced.tex}
	\input{sections/appendix.tex}
	\printbibliography[heading=bibintoc]
\end{document}

%% file: sections/introduction.tex
\section{Introduction}
 
\subsubsection{Commutativity and homotopy coherence}
A \textit{commutative monoid} is a set $M$ equipped with a map of sets $\bullet: M \times M \to M$ satisfying the following conditions:
	\begin{itemize}[label={}]
		\item \textit{Unit:} There exists $e \in M$ such that $e \bullet m = m \bullet e = m$ for every $m \in M$
		\item \textit{Associativity}: $m_1 \bullet (m_2 \bullet m_3) = (m_1 \bullet m_2) \bullet m_3$ for every $m_1,m_2,m_3 \in M$.
		\item \textit{Commutativity}: $m_1 \bullet m_2 = m_2 \bullet m_1$ for every $m_1 , m_2 \in M$. 
\end{itemize}

Note that each axiom in the definition above requires at most \textbf{$3$ free variables}. With this feature in mind we recall the definition of a symmetric monoidal category.
A \textit{symmetric monoidal category} is a category $\calC$ equipped with a bifunctor $\otimes: \calC \times \calC \to \calC$ and the following \textit{data}:
	\begin{itemize}[label={}]
		\item \textit{Unitor:} An object $1 \in \calC$ equipped with isomorphisms $x \otimes 1 \simeq x \simeq 1 \otimes x$ natural in $x \in \calC$.
		\item \textit{Associator:} An isomorphism $x_1\otimes (x_2 \otimes x_3) \simeq (x_1 \otimes x_2) \otimes x_3$ natural in $x_1,x_2,x_3 \in \calC$.
		\item \textit{Braiding:} An isomorphism  $x_1 \otimes x_2 \simeq x_2 \otimes x_1$ natural in $x_1,x_2 \in \calC$.
	\end{itemize} 
	This data is required to satisfy a list of conditions saying that certain diagrams must commute. All of the diagrams can be written with at most $3$ free variables except Maclane's pentagon diagram.\footnote{we omit the other diagrams in the interest of brevity.}
		\[\begin{tikzcd}
		& {(x_1 \otimes x_2) \otimes (x_3\otimes x_4)} \\
		{x_1 \otimes (x_2\otimes (x_3 \otimes x_4))} && {((x_1 \otimes x_2)\otimes x_3)\otimes x_4} \\
		{x_1 \otimes ((x_2\otimes x_3)\otimes x_4)} && {(x_1 \otimes (x_2\otimes x_3))\otimes x_4}
		\arrow[Isom, no head, from=2-1, to=1-2]
		\arrow[Isom, no head, from=1-2, to=2-3]
		\arrow[Isom, no head, from=3-1, to=3-3]
		\arrow[Isom, no head, from=3-3, to=2-3]
		\arrow[Isom, no head, from=3-1, to=2-1]
		\end{tikzcd}\]

One feature that immediately pops out is that notions which were previously introduced as describing conditions - \textit{unit}, \textit{associativity} and \textit{commutativity} - are no longer conditions but rather extra data to be specified. This is a common theme when passing from classical algebra to homotopically coherent algebra. This will not be our focus however. Instead we focus on the fact that each axiom in the definition of a symmetric monoidal category requires at most \textbf{$4$ free variables}.

The diagrammatic definition of a \textit{symmetric monoidal 2-category} is quite delicate, see for example \cite{Day}. Important for us will be the qualitative observation that each axiom requires at most \textbf{$5$ free variables} (for the entire list of axioms spelled out in full detail see \cite[Appendix~C]{Pries}).

From the $\infty$-categorical perspective, all the algebraic structures we mentioned so far are instances of the same homotopy coherent notion. Namely, they are all examples of \textit{homotopy coherent commutative monoids} in an $\infty$-category with finite product. Before recalling the definition let us fix some notation.

\begin{notation}
	Let $\hl{\Fin}$ and $\hl{\Fin_{\ast}}$ denote the categories of \hl{finite sets} and \hl{pointed finite sets} respectively. For an integer $m \ge 1$, we denote $\hl{\bracket{m} \coloneq } \{1,\dots,m\} \coprod \{\ast\} \in \Fin_{\ast}$. For an integer $1 \le i \le m$ we define $\hl{\delta_i:} \bracket{m} \inert \bracket{1}$ as follows
	\[ \hl{\delta_i(j) \coloneq } \begin{cases}
	1 & \text{ if } i=j
	\\
	\ast & \text{ otherwise }
	\end{cases}\]
\end{notation}

For the remainder of this section we fix an $\infty$-category $\calC$ with finite products.

\begin{defn}[{\cite[~2.4.2.1.]{HA}}]\label{defn:comm-monoids}
	A \hl{commutative monoid} in $\calC$ is a functor $F: \Fin_{\ast} \to \calC$ satisfying the \textit{Segal condition}:
	\begin{itemize}
		\item For every $m \ge 0$ the Segal map induces an equivalence
		$\prod^m_{j=1} F(\delta_j):F(m) \iso F(1)^{\times m}$ is an equivalence.
	\end{itemize}
	Denote by $\hl{\CMon(\calC) }\subseteq \Fun(\Fin_{\ast},\calC)$ the full subcategory of commutative monoids.
\end{defn}
In this paper we study a variant of the above notion which depends on a parameter $k \in \bbN$. 

\begin{notation}
	For a finite set $A \in \Fin$ we denote by $\hl{A_+ \coloneq } \, A \coprod \{\ast\} \in \Fin_{\ast}$ the pointed set obtained from $A$ by adding a disjoint base point. Denote by $\hl{|A|} \in \bbN$ the cardinality of $A$ and let $\hl{\Fin^{\le k}_{\ast}} \subseteq \Fin_{\ast}$ denote the full subcategory spanned by pointed sets $A_+$ with $|A|\le k$.
\end{notation}

\begin{defn}\label{defn:restricted-comm-monoids}
	A  \textit{$k$-restricted commutative monoid} in $\calC$ is a functor $F: \Fin^{\le k}_{\ast} \to \calC$ satisfying the following variant of the Segal condition:
	\begin{itemize}
		\item For every $0 \le m \le k$ the Segal map induces an equivalence
		$\prod^m_{j=1} F(\delta_j) : F(m) \iso F(1)^{\times m}$. 
	\end{itemize}
	We denote by $\hl{\CMon^{\le k}(\calC)} \subseteq \Fun(\Fin^{\le k}_{\ast},\calC)$ the full subcategory of $k$-restricted commutative monoids.
\end{defn}

Note that the restriction $\Fun(\Fin_{\ast},\calC) \to \Fun(\Fin^{\le k}_{\ast},\calC)$ preserves the Segal condition and thus gives rise to a functor:
\[ \CMon(\calC) \to \CMon^{\le k}(\calC) \]
Our main result will pertain to the interplay between truncatedness of mapping spaces of an $\infty$-category $\calC$ and the number of variables required to specify commutative monoids in $\calC$. 

\begin{defn}
	An \hl{$(n,1)$-category} is an $\infty$-category $\calC$ all of whose mapping spaces are $(n-1)$-truncated, i.e. $\Map_{\calC}(x,y)$ is $(n-1)$-truncated for every $x,y \in \calC$. A \hl{complete $(n,1)$-category} is an $(n,1)$-category $\calC$ which is complete as an $\infty$-category.
\end{defn}

The following is a special case of our main theorem.

\begin{corA}[\cref{thm:main-theorem-application}]\label{thm-motivational}
	For every complete $(n,1)$-category $\calC$, restriction induces an equivalence of $\infty$-categories
	\[ \CMon(\calC) \simeq \CMon^{\le n+2}(\calC) \]
\end{corA}

\begin{example}
	Sets, categories and $2$-categories constitute complete $(1,1)$-category, $(2,1)$-category, and $(3,1)$-category respectively. We see that our earlier observations, exactly fit the conclusion of \cref{thm-motivational} in the cases $n=1,2,3$. 
\end{example}

\cref{thm-motivational} will be deduced from a general statement about $k$-restricted monoids for an arbitrary $\infty$-operad.

\subsubsection{Arity restricted monoids over $\infty$-operads}

The theory of $\infty$-operads provides a convenient framework for the study of coherent algebraic structures in the setting of $\infty$-categories. Before recalling the definition of an $\infty$-operad, we setp up some terminology and notation.

\begin{defn}	
    The category $\Fin_{\ast}$ admits a factorization system $(\Fin^{\int}_{\ast},\Fin^{\act}_{\ast})$ defined as follows
	\begin{enumerate}
		\item A map $f:A_+ \inert B_+$ is said to be \hl{inert} if $|f^{-1}(b)|=1$ for all $b \in B$. Denote by $\hl{\Fin^{\int}_{\ast}} \subseteq \Fin_{\ast}$ the corresponding wide subcategory.
		\item A map $f:A_+ \actarrow B_+$ is said to be \hl{active} if $f^{-1}(\ast)= \{\ast\}$. Denote by $\hl{\Fin^{\act}_{\ast}} \subseteq \Fin_{\ast}$ the corresponding wide subcategory.
	\end{enumerate}
\end{defn}

\begin{notation}
	We follow the convention in \cite{RH-algebraic} and denote active morphisms with squiggly arrows $\rightsquigarrow$ and inert morphisms with tailed arrows $\rightarrowtail$.  
\end{notation}

\begin{defn}[{\cite[Definition 2.1.1.10.]{HA}}]
	An \textit{$\infty$-operad} is an $\infty$-category $\calO$ equipped with a categorical fibration $\bracket{-}: \calO \to \Fin_{\ast}$ such that the following conditions hold:
	\begin{enumerate}
		\item \textbf{Inert Lifts:} Every inert morphism $\lambda : \bracket{x} \inert A_+$ with $x \in \calO$ and $A_+\in \Fin_{\ast}$ admits a cocartesian lift to a morphism $\widehat{\lambda}: x \to x_{\lambda}$ in $\calO$ such that $\bracket{\widehat{\lambda}} = \lambda$.
		\item \textbf{Segal Condition:} For every $m \ge 0$ the natural functor 
		\[ \prod_{\delta: \bracket{m} \inert \bracket{1}} (\delta_i)_! : \calO_{\bracket{m}} \to \prod_{\delta : \bracket{m} \inert \bracket{1}} \calO_{\bracket{1}} =  \calO_{\bracket{1}}^{\times m}\]
		induced from the cocartesian lifts of $\delta_i : \bracket{m} \inert \bracket{1}$ for all $1\le i\le m$ is an equivalence.
		\begin{itemize}
			\item \textit{Notation:} Given a tuple $(x_1,\dots,x_m) \in \calO_{\bracket{1}}^{\times m}$ we denote the corresponding object in $\calO_{\bracket{m}}$ by $x_1 \oplus \cdots \oplus x_m $ 
			\item More generally the Segal condition gives rise to an equivalence $\calO_{\bracket{m_1}} \times \calO_{\bracket{m_2}} \simeq \calO_{\bracket{m}}$ for every $m_1 + m_2 = m$. On objects we will denote this by $(x_1,x_2) \mapsto x_1 \oplus x_2$.
		\end{itemize}		
		\item \textbf{Mapping Space Condition}: For every tuple $(y_1,\dots.y_m) \in \prod^m_{j=1} \calO_{\bracket{1}}$ and every $x \in \calO$, the natural map 
		\[ \Map_{\calO}(x,y_1 \oplus \cdots \oplus y_m) \to \prod^m_{j=1} \Map_{\calO}(\delta_{j \,!}(x),y_j) \]
		is an equivalence.
	\end{enumerate}
	We call $\bracket{-}:\calO \to \Fin_{\ast}$ the structure map of the $\infty$-operad $\calO$.
\end{defn} 

Intuitively we think of $\infty$-operads as modeling \textit{symmetric multi-$\infty$-categories}. Slightly more precisely, given an $\infty$-operad $\calO$ one can extract the following pieces of data:

\begin{enumerate}
	\item A space of \textbf{objects} $\Obj(\calO) \coloneq  \calO_{\bracket{1}}^{\simeq}$.
	\item For every collection of objects $x_1,\dots,x_k,y \in \Obj(\calO)$, a space of \textbf{multi-morphisms}
	\[\Mul_{\calO}(x_1,\dots,x_k;y) \coloneq  \Map_{\calO^{\act}}(x_1\oplus \cdots \oplus x_m,y) \]
	\item A \textbf{composition} operation on multi-morphisms: for every collection $\alpha = \{\alpha_1,\dots,\alpha_r\}$ of multi-morphisms where $\alpha_j \in \Mul_{\calO}(x_{j,1},\dots,x_{j,k_j};y_j)$ and a multi-morphism $\beta \in \Mul_{\calO}(y_1,\dots,y_r;z)$ there is a composite multi-morphism $\beta \circ \alpha \in \Mul_{\calO}(x_{1,1},\dots,x_{r,k_r};z)$.
\end{enumerate}

The rest of the data in $\calO$ can be thought of as coherently exhibiting all the associativity, and symmetry properties of the composition operation.

\begin{example}
	 We write $\hl{\bbE_{\infty}}$ for the $\infty$-operad corresponding to the identity functor $\Id : \Fin_{\ast} \to \Fin_{\ast}$.
\end{example}

\begin{example}
	We write $\hl{\bbE_1}$ for the $\infty$-operad denoted in \cite[Remark 4.1.1.4]{HA} by $\Assoc^\otimes$. For the benefit of the reader we recall its definition. Objects of $\Assoc^\otimes$ are pointed finite sets and morphisms are pairs $(f,\{\le_t \}) \in \Map_{\Assoc^\otimes}(S_+,T_+)$ where $f:S_+ \to T_+$ is a map of finite sets and $\{\le_t\}$ is a collection of linear ordering on the fibers $f^{-1}(t)$ for every $t \in T$. One then checks that the forgetful functor $\Assoc^\otimes \to \Fin_{\ast}$ exhibits $\Assoc^\otimes$ as an $\infty$-operad. 
\end{example}

\begin{example}
    We write $\bbE_d$ for the \operad{} of little $d$-discs, denoted in \cite[Definition 5.1.0.2]{HA} by $\bbE_d^\otimes$.
    The wide subcategory $\bbE_d^\act \subseteq \bbE_d$ is equivalent to the \category{} whose objects are finite disjoint unions of $\bbR^d$'s and whose mapping spaces are spaces of framed $C^\infty$-embeddings.
\end{example}

\begin{defn}[{\cite[~2.4.2.1.]{HA}}]\label{defn:O-monoids}
	Let $\calO$ be an $\infty$-operad and $\calC$ an $\infty$-category with finite products. An \textit{$\calO$-monoid} in $\calC$ is a functor $F: \calO \to \calC$ satisfying the following Segal condition:
	\begin{itemize}
		\item For every $x \in \calO$ the natural map
		$\prod_{\delta : \bracket{x} \inert \bracket{1} } F(\widehat{\delta}):F(x) \to  \prod_{\delta : \bracket{x} \inert \bracket{1}} F(x_{\delta})$
		is an equivalence.
	\end{itemize}
	We denote by $\hl{\Mon_{\calO}(\calC)} \subseteq \Fun(\calO,\calC)$ the full subcategory of $\calO$-monoids.
\end{defn}

In analogy with the case of commutative monoids, we introduce an arity restricted generalization of $\calO$-monoids.

\begin{notation}
	For an $\infty$-operad $\calO$ we write $\hl{\calO^{\le k} \coloneq  }\, \Fin^{\le k}_{\ast} \times_{\Fin_{\ast}} \calO \subseteq \calO$.
\end{notation}

\begin{defn}\label{defn:restricted-monoid}
	Let $\calO$ be an $\infty$-operad and $\calC$ a category with finite products. A  \textit{$k$-restricted $\calO$-monoid} in $\calC$ is a functor $F: \calO^{\le k}: \to \calC$ satisfying the Segal condition:
	\begin{itemize}
		\item For every $x \in \calO^{\le k}$ the natural map
		$\prod_{\delta : \bracket{x} \inert \bracket{1} } F(\widehat{\delta}):F(x) \to  \prod_{\delta : \bracket{x} \inert \bracket{1}} F(x_{\delta})$ 
		is an equivalence.
	\end{itemize}
	Denote by $\hl{\Mon^{\le k}_{\calO}(\calC)} \subseteq \Fun(\calO^{\le k},\calC)$ the full subcategory of $k$-restricted $\calO$-monoids.
\end{defn}

\cref{defn:restricted-monoid} is a simultaneous generalization of \cref{defn:O-monoids} and \cref{defn:restricted-comm-monoids}. Indeed, we recover the former by letting $k=\infty$, and the latter by letting $\calO=\bbE_{\infty}$. Note that the restriction $\Fun(\Fin_{\ast},\calC) \to \Fun(\Fin^{\le k}_{\ast},\calC)$ preserves the Segal condition and thus gives rise to a functor:
\[ \Mon_{\calO}(\calC) \to \Mon_{\calO}^{\le k}(\calC) \]
In this paper we develop tools to determine, 
given an $\infty$-operad, for which $n$ and $k$, 
does restriction induce an equivalence $\Mon_{\calO}(\calC) \simeq \Mon_{\calO}^{\le k}(\calC)$ for all $(n.1)$-categories $\calC$. In particular we obtain the following result.

\begin{thmA}[\cref{thm:arity-approximation-operads,corA:part-conn-Ed}]\label{thm:main-theorem-application}
	Let $\calC$ be a complete $(n,1)$-category. 
	For all $1 \le d \le \infty$ restriction induces an equivalence of $\infty$-categories:
	\[ \Mon_{\bbE_d}(\calC) \iso \Mon_{\bbE_d}^{\le n+2}(\calC) \]
\end{thmA}

\begin{rem}
	Note that for $\calO = \bbE_{\infty}$ the above is exactly \cref{thm-motivational}. 
\end{rem}

\subsubsection{Partition complexes of $\infty$-operads}

The dependence between the truncation and the arity in \cref{thm:main-theorem-application} is controlled by the connectivity of certain generalized partition complexes attached to any $\infty$-operad. 
In the special case of $\bbE_{\infty}$ these are exactly the partition poset complexes that have found many uses in the theory of $\infty$-operads, in Koszul duality and in Goodwillie calculus (see for example \cite{Part1}, \cite{Fresse} and \cite{Ching}). 
Before stating the main theorem we need to introduce some necessary definitions.

\begin{defn}
	A morphism $\alpha: x \actarrow y$ in $\calO$ is said to be \hl{active} if the underlying morphism of pointed finite sets $\bracket{\alpha}: \bracket{x} \actarrow \bracket{y}$ is active. Denote by $\hl{\calO^{\act}} \subseteq \calO$ the wide subcategory of active morphisms. A \hl{multi-morphism} in $\calO$ is an active morphism $\mu : x \actarrow z$ such that $\bracket{z} \simeq \bracket{1}$. We say that $\mu$ is of \hl{arity $k$} if $\bracket{x} \simeq \bracket{k}$.
\end{defn}

\begin{defn}\label{intro-defn-factorization-category}
	Given a multi-morphism $\mu :x \actarrow z$ in $\calO$, define the \hl{factorization category} of $\mu$ as $\hl{\Fact_{\calO}(\mu) \coloneq } (\calO^{\act}_{/z})_{\mu/}$.
	More informally, $\Fact_{\calO}(\mu)$ is the $\infty$-category whose
	\begin{itemize}
		\item \textbf{Objects} are active factorizations of $\mu$:
		\[\begin{tikzcd}
		& y \\
		x && z
		\arrow[squiggly, from=2-1, to=1-2]
		\arrow[squiggly, from=1-2, to=2-3]
		\arrow["\mu"{description}, squiggly, from=2-1, to=2-3]
		\end{tikzcd}\]
		
		\item \textbf{Morphisms} are commutative diagrams:
		\[\begin{tikzcd}
		& y \\
		& {y'} \\
		x && z
		\arrow[squiggly, from=1-2, to=3-3]
		\arrow["\mu"', squiggly, from=3-1, to=3-3]
		\arrow[squiggly, from=3-1, to=1-2]
		\arrow[squiggly, from=3-1, to=2-2]
		\arrow[squiggly, from=2-2, to=3-3]
		\arrow[squiggly, from=1-2, to=2-2]
		\end{tikzcd}\]
	\end{itemize}
\end{defn}

\begin{example}\label{ex:fact-E1}
	For $\bbE_1$, the multi-morphisms of arity $k$ correspond to bijections $\bracket{k}^{\circ} \simeq [k-1]$ where $\bracket{k}^{\circ} \coloneq  \{1,\dots,k\}$ denotes the (unpointed) finite set obtained from $\bracket{k}$ by removing the base point. Unwinding definitions we see that for any such bijection $\gamma$ we have
	$\Fact_{\bbE_1}(\gamma) \simeq  \mathlarger{\Delta}_{[k-1]/}$.
\end{example}

\begin{example} 
	Since $\mathbb{E}_{\infty}$ is the terminal $\infty$-operad, there exists in every arity $k$ a unique multi-morphism $\mu_k$ of that arity. Unwinding definitions reveals that
	$\Fact_{\mathbb{E}_{\infty}}(\mu_k) \simeq \Fin_{\bracket{k}^{\circ}/}$.
\end{example}

\begin{defn}
	Let $\hl{\Part_k} \subseteq \Fin_{\bracket{k}^{\circ}/}$ denote the full subcategory whose objects are the surjective morphisms $\bracket{k}^{\circ} \to A$ where $A$ is a set of cardinality $1<|A| <k$.
\end{defn}

\begin{defn}\label{cons:intro-definition-partition}
	Let $\calO$ be an $\infty$-operad and let $\mu$ be a multi-morphism in $\calO$ of arity $k$. Then by functoriality there is a canonical map:
	\[\Fact_{\calO}(\mu) \to \Fin_{\bracket{k}^{\circ}/}\]
	 Define the \hl{partition category} of $\mu$ as the following pullback
	\[\begin{tikzcd}
	{\hl{\Part_{\calO}(\mu)}} & {\Fact_{\calO}(\mu)} \\
	{\Part_k} & {\Fin_{\bracket{k}^{\circ}/}}
	\arrow[hook, from=1-1, to=1-2]
	\arrow[from=1-1, to=2-1]
	\arrow[""{name=0, anchor=center, inner sep=0}, hook, from=2-1, to=2-2]
	\arrow[from=1-2, to=2-2]
	\arrow["\lrcorner"{anchor=center, pos=0.125}, draw=none, from=1-1, to=0]
	\end{tikzcd}\]
	Define the \hl{partition complex} of $\mu$ as $\hl{\mathlarger{\Pi}^{\calO}_{\mu} \coloneq } \big| \Part_{\calO}(\mu) \big|$, where $|-|$ denotes the left adjoint to the inclusion $\calS \mono \Cat_{\infty}$. 
	Finally, we define
	\[ \hl{\sig{\calO}(k)\coloneq } \inf \bigg\{ \conn \big(\mathlarger{\Pi}^{\calO}_{\mu}\big) \bigg| \, \mu  \text{ multi-morphism in $\calO$ of arity} > k \bigg\} \] 
	where $\conn \big(\mathlarger{\Pi}^{\calO}_{\mu}\big)$ denotes the connectivity of $\mathlarger{\Pi}^{\calO}_{\mu}$.
\end{defn}

\begin{example}\label{ex:partition-Einfinity}
	Note that $\spart{\bbE_{\infty}}{k}$ is by definition the realization of the poset of partitions of $k$ excluding $(1,\dots,1)$ and $(k)$. It is well known (see \cite[Proposition 4.109]{Hyperplanes}) that:
	\[ \spart{\bbE_{\infty}}{k}\simeq \underset{(k-1)!}{\bigvee} S^{k-3}.\]
\end{example}

\begin{lem}\label{lem:partition-of-E1}
	Let $k \ge 2$ be an integer and $\gamma$ a multi-morphism in $\bbE_1$ of arity $k$. Then there is an equivalence:
	\[ \spart{\bbE_1}{\gamma} \simeq S^{k-3}.\] 
\end{lem}
\begin{proof}
	By \cref{ex:fact-E1} we may identify $\Part_{\bbE_1}(\gamma)$ with the full subcategory of $\Delta_{[k-1]/}$ consisting of surjective morphisms $[k-1] \to [j]$ with $1 \le j \le k-2$. Recall that there is an equivalence $(\Delta^{\surj})^{\op} \simeq \Delta^{\inj}_+$ defined on objects by $I \mapsto I \setminus \{i_{\max}\}$. We get an induced equivalence $(\mathlarger{\Delta}^{\surj}_{[k-1]/})^{\op} \simeq (\mathlarger{\Delta}^{\inj}_{+})_{/[k-2]}$.	Under this equivalence $\Part_{\bbE_1}(\gamma)^{\op}$ is identified with the full subcategory of $ (\mathlarger{\Delta}^{\inj}_{+})_{/[k-2]}$ consisting of injective morphisms $[i] \to [k-2]$ with $0 \le i \le k-3$. The latter category is precisely the poset of non-empty proper subsets of $\{0,\dots,k-2\}$ whose realization is a $(k-3)$-dimensional sphere.
\end{proof}

We are ready to state the main result of the paper.

\begin{thmA}[Arity Approximation of $\infty$-Operads - \cref{thm:arity-approximation-operads}]\label{thm:main-theorem}
	Let $\calO$ be an $\infty$-operad and let $\calC$ be a complete $(\sig{\calO}(k)+1,1)$-category. Then restriction induces an equivalence of $\infty$-categories
	\[\Mon_{\calO}(\calC) \simeq \Mon^{\le k}_{\calO}(\calC)\]
\end{thmA}

\cref{thm:main-theorem} reduces \cref{thm:main-theorem-application} to a computation.
For $d \in \{1,\infty\}$ we can deduce $\sig{\bbE_d}(k)=k-3$ from \cref{ex:partition-Einfinity} and \cref{lem:partition-of-E1}.
For $1<d<\infty$ the partition complexes of $\bbE_d$ are much more complicated. 
To describe them we let $\FM_d(k)$ denote the compactification of the configuration space $\mrm{Conf}_k(\bbR^d)/\bbR^d \rtimes \bbR_{>0}$ constructed by Getzler-Jones in \cite{getzler}, where it is denoted by $\mathsf{F}_d(k)$.

\begin{thmA}[Partition Complexes of $\bbE_d$]\label{thm:partition-Ed-intro}
    For $1 < d< \infty$ and $k \ge 3$ we have an equivalence:
    \[\spart{\bbE_d}{k} \simeq \fib\left( \partial \FM_d(k) \hookrightarrow \FM_d(k)\right).\]
\end{thmA}

We shall deduce \cref{thm:partition-Ed-intro} from a general result about partition complexes in the dendroidal setting (\cref{thm:dendroidal-partition-complex}), using the operad structure on these compactified configuration spaces.
This operad structure has a long and complicated history.
It was first observed by Getzler-Jones in \cite{getzler-jones} and then developed in detail by Markl in \cite{markl}.
Markl's construction was shown to be equivalent to the little disc operad  by Salvatore in \cite{salvatore}.

Despite the complicated appearance of $\spart{\bbE_d}{k}$, a simple and somewhat miraculous cancellation leads to the connectivity being independent of $d$.

\begin{corA}\label{corA:part-conn-Ed}
     For all $1 \le d\le \infty $ and all $k \ge 3$ we have $\sig{\bbE_d}(k) = k-3$.
\end{corA}
\vspace{-1.5em}
\begin{proof}
    It suffices to show that $\conn(\spart{\bbE_d}{k})=k-4$.
    The cases $d \in \{1,\infty\}$ follow from \cref{ex:partition-Einfinity} and \cref{lem:partition-of-E1}.
    In the remaining cases it suffices by \cref{thm:partition-Ed-intro} to show that the inclusion 
	$\partial \FM_d(k) \hookrightarrow \FM_d(k)$ 
	is $(k-4)$-connected.
	Since $\FM_d(k)$ is a $\left(d(k-1)-1\right)$-dimensional manifold with corners of homotopical dimension $(d-1)(k-1)$ it remains to prove the following claim:
	\begin{itemize}
	    \item 
	    Let $m >1$ and let $M$ be a compact oriented smooth $m$-dimensional manifold with corners, of homotopical dimension $q$.
	    Then the inclusion
    	$\partial M \coloneq  M \setminus \mrm{int}(M) \hookrightarrow M$ is $(m-q-2)$-connected. 
	\end{itemize}
	Smoothing the corners we may assume that $M$ is a compact smooth manifold with $\partial M$ as its boundary.
	By assumption $M$ has homotopical dimension $q$, so we may pick the handle decomposition whose top dimensional handles are in dimension $q$.
	The corresponding $\mrm{CW}$-structure induces a bijection between the $r$-dimensional cells of $M$ and the $(m-r)$-dimensional cells of the relative $\mrm{CW}$ complex $(M,\partial M)$.
	In particular the bottom most cells of $(M,\partial M)$ sit in dimension $m-q$.
	The inclusion $\partial M \hookrightarrow M$ is therefore $(m-q-2)$-connected.
\end{proof}

\subsubsection{Structure of the paper}

In \cite{RH-algebraic,RH-cartesian}, Haugseng and Chu develop a generalization of $\infty$-operads called \textit{algebraic patterns}. We will make extensive use of the theory of algebraic patterns. In \cref{sect:algebraic-patterns} we review the basics of this theory and develop some necessary tools for our intended application. In \cref{sect:analytic-Patterns} we shall restrict our attention to a narrower class of algebraic patterns which we term \textit{analytic patterns} of which $\infty$-operads are a special case. Analytic patterns (inspired by the \textit{cartesian patterns} of \cite{RH-cartesian}), share many properties with $\infty$-operads, and yet, being less rigid objects allow for more constructions. This makes them suitable for our purposes. 
The main result of \cref{sect:analytic-Patterns} is \cref{thm:artiy-restriction-theorem}, which is a version of the arity approximation theorem (\cref{thm:main-theorem}) in the general setting of analytic patterns. 
In \cref{sect:partition-complexes} we restrict attention to $\infty$-operads and simplify the complexes appearing in \cref{thm:artiy-restriction-theorem} until we arrive at \cref{thm:main-theorem}.
In \cref{sect:5} we give an alternative description of operadic partition complexes in terms of dendroidal homotopy theory and use it to compute the partition complexes of $\bbE_d$.
We have also included an appendix recording some technical results about factorization systems and slice $\infty$-categories which we make use of throughout the paper.

\subsubsection{Acknowledgements}

First and foremost, I wish to express my deep gratitude to my advisor, Tomer Schlank, for his guidance and support.
His inspiring talent for dodging mathematical obstacles were incredibly helpful at every stage of this project.
I wish to thank Shay Ben Moshe, my academic brother, whose detailed comments on a first draft were invaluable, and to all members of the Seminarak group for many useful conversations.
Among them I'm especially grateful to Lior Yanovski, for useful insights regarding early aspects of this work.
I wish to also thank Rune Haugseng for his valuable suggestions and Jan Steinebrunner, for many interesting and fun conversations which were crucial in shaping my understanding of this subject.

\subsubsection{Conventions}

We work in the framework of $\infty$-categories (a.k.a. quasicategories), introduced by Joyal \cite{Joyal} and extensively developed by Lurie in \cite{HTT} and \cite{HA}. We shall also use the following notation and terminology:

\begin{enumerate}
	\item Given $-2 \le n \le \infty$ we denote by $\calS^{\le n}$ the $\infty$-category of $n$-truncated spaces.
	\item Given an $\infty$-category $\calC$ and a functor $f: X \to Y$ we let $f^{\ast}:\Fun(Y,\calC) \to \Fun(X,\calC)$ denote the functor defined by precomposing with $f$. 
	
	\item In the situation of $(2)$, if $\calC$ admits all small limits (resp. colimits) then $f$ admits a right (resp. left) adjoint which we denote by $f_{\ast}$ and $f_!$ respectively.
	
	\item We say that a functor $f: X \to Y$ is initial if $f^{\op}$ is cofinal in the sense of \cite[Definition 4.1.1.1]{HTT}. We say that $f$ is final if $f^{\op}$ is initial.
\end{enumerate}

%% file: sections/algebraic-patterns.tex
\section{Algebraic patterns}\label{sect:algebraic-patterns}

In this section we make extensive use of the theory of algebraic patterns developed by Haugseng and Hongyi in \cite{RH-algebraic}. 
The main purpose will be to establish a recognition theorem for Morita equivalences of algebraic patterns (\cref{thm:Morita-equivalence-from-weakly-initial-slices}) which will serve as the main technical tool in \cref{sect:analytic-Patterns}.   
We begin by reviewing some basic definitions and facts on algebraic patterns. 

\subsection{Background on algebraic patterns}

We review the basics of algebraic patterns as developed in \cite{RH-algebraic}. These are $\infty$-categorical gadgets which provide a general framework for the study of homotopy coherent algebraic structures  of "Segal type".
We also introduce the notion of \textit{algebraic subpattern} (\cref{defn:algebraic-subpattern}), which is a subcategory $\calP \subseteq \calQ$ on which there's a canonical structure of an algebraic pattern.

\begin{defn}[{\cite[Definition 2.1]{RH-algebraic}}]
	An \textit{\hl{algebraic pattern}} is an $\infty$-category $\calO$ equipped with the following structure
	\begin{enumerate}
		\item An (inert,active) factorization system $(\hl{\calO^{\int}},\hl{\calO^{\act}})$ on $\calO$.
		
		\item A full subcategory $\hl{\calO^{\el}} \subseteq \calO^{\int}$ of \hl{elementary objects}.
	\end{enumerate}
	A morphism $f:\calO \to \calP$ of algebraic patterns is a functor  preserving all of the above. That is, it sends inert (resp. active) morphisms to inert (resp. active) morphisms and elementary objects to elementary objects.
\end{defn}

\begin{notation}
	Let $\calO$ be a algebraic pattern and let $x \overset{\alpha}{\to} y$ be a morphism in $\calO$. We denote the (inert,active) factorization of $\alpha$ as follows
	\[x \hl{\overset{\alpha^{\int}}{\rightarrowtail} \Lambda(\alpha) \overset{\alpha^{\act}}{\rightsquigarrow}} y\]
	More generally we follow the convention introduced in \cite{RH-algebraic} and denote active morphisms with squiggly arrows $\rightsquigarrow$ and inert morphisms.with tailed arrows $\rightarrowtail$.  
\end{notation}

%

\begin{example}\label{ex:inert-active-on-Fin}
	We abuse notation and denote by $\hl{\Fin_{\ast}}$ also the algebraic pattern whose underlying category is pointed finite sets, whose factorization system is $(\Fin^{\int}_{\ast},\Fin^{\act}_{\ast})$ and whose elementary objects are $A_+ \in \Fin_{\ast}$ with $|A|=1$.
\end{example}

\begin{example}[{\cite[Definition 3.7]{RH-cartesian}}]\label{ex:operads-are-algebraic-patterns}
	Let $\bracket{-}: \calO \to \Fin_{\ast}$ be an $\infty$-operad. Then $\calO$ admits a factorization system defined as follows
	\begin{enumerate}
		\item A map $\alpha :x \to y$ is \hl{inert} if $\bracket{\alpha}: \bracket{x} \to \bracket{y}$ is inert and $\alpha$ is a cocartesian lift of $\bracket{\alpha}$.
		\item A map $\alpha: x \to y$ is \hl{active} if $\bracket{\alpha}: \bracket{x} \to \bracket{y}$ is active.
	\end{enumerate} 
	We regard $\calO$ as an algebraic pattern with the above factorization system and with $\calO^{\el} \coloneq  \calO^{\int}_{\bracket{1}} \subseteq \calO^{\int}$.
\end{example}

\begin{example}
	Let $\calO$ be an algebraic pattern. The subcategories $\calO^{\int}$ and $\calO^{\el}$ of $\calO$ are naturally algebraic patterns by taking elementary objects, inert/active morphisms the same as those in $\calO$. The fully faithful inclusion $\hl{i_{\calO} :} \calO^{\el} \to \calO^{\int}$ is then a morphism of algebraic patterns.
\end{example}

\begin{defn}[{\cite[Definition 2.6]{RH-algebraic}}]\label{defn:el-diagrams}
	For an algebraic pattern $\calO$ we denote by $\hl{\calK_{\calO}}$ the collection of all diagrams of the form $\calO^{\rm el}_{x/}$ for some $x \in \calO$.
\end{defn}

The following definition relates algebraic patterns to the study of algebraic structures.

\begin{defn}[{\cite[Definition 2.7.]{RH-algebraic}}]
	Let $\calO$ be an algebraic pattern and let $\calC$ be an $\infty$-category which admits all $\calK_{\calO}$-limits. 
	An \textit{$\calO$-Segal object of $\calC$} is a functor $F: \calO \to \calC$ satisfying the following:
	\begin{itemize}
		\item \textbf{Segal Condition:} For every $x \in \calO$ the natural comparison map
	    $F(x) \to \underset{e \in \calO^{\el}_{x/}}{\lim} F(e)$
		is an equivalence. 
	\end{itemize}
	Denote by $\hl{\Seg_{\calO}(\calC)} \subseteq \Fun(\calO,\calC)$ the full subcategory of $\calO$-Segal objects in $\calC$.
\end{defn}

\begin{example}
	Commutative monoids are by definition $\Fin_{\ast}$-Segal objects, i.e.~$\CMon(\calC) \coloneq  \Seg_{\Fin_{\ast}}(\calC)$. More generally, an $\infty$-operad $\calO$ may be regarded as an algebraic pattern as in \cref{ex:operads-are-algebraic-patterns}. 
	In this case the $\calO$-Segal objects are precisely the $\calO$-monoids.
\end{example}

\begin{lem}[{\cite[Lemma 2.9.]{RH-algebraic}}]\label{lem: Segal iff right Kan extended}
	$F : \calO \to \calC$ is Segal if and only if $F|_{\calO^{\int}}$ is the right Kan extension of $F|_{\calO^{\el}}$.
\end{lem}

\begin{rem}\label{rem:alternative-definition-of-Seg}
	As a consequence of \cref{lem: Segal iff right Kan extended}, we have a pullback square
	\[\begin{tikzcd}
	{\Seg_{\calO}(\calC)} && {\Fun(\calO,\calC)} \\
	{\Fun(\calO^{\el},\calC)} && {\Fun(\calO^{\int},\calC).}
	\arrow[from=1-1, to=1-3]
	\arrow[from=1-3, to=2-3]
	\arrow[""{name=0, anchor=center, inner sep=0}, "{i_{\calO, \ast}}"', from=2-1, to=2-3]
	\arrow[from=1-1, to=2-1]
	\arrow["\lrcorner"{anchor=center, pos=0.125}, draw=none, from=1-1, to=0]
	\end{tikzcd}\]
	where the bottom functor is right Kan extension along the inclusion $i_{\calO} : \calO^{\el} \mono \calO^{\int}$.
\end{rem}

\begin{cor}\label{lem: restriction to elementary is conservative}
	Let $\calO$ be a algebraic pattern and $\calC$ an $\infty$-category admitting all $\calK_{\calO}$-limits. Then the restriction functor $\Seg_{\calO}(\calC) \to \Fun(\calO^{\el},\calC)$ is conservative.
\end{cor}
\begin{proof}
	Let $\alpha: F \to G$ be a morphism of Segal objects such that $\alpha^{\el}: F|_{\calO^{\el}} \to F|_{\calO^{\el}}$ is an equivalence. By \cref{lem: Segal iff right Kan extended}, the restriction $\alpha^{\int} :F|_{\calO^{\int}} \to G|_{\calO^{\int}}$ is also an equivalence. But $\calO^{\int} \to \calO$ is a wide subcategory (i.e.  contains all objects) so we're done.
\end{proof}

\begin{notation}
	We denote by $\AlgPatt$ the $\infty$-category of algebraic patterns \cite[Definition 5.4.]{RH-algebraic}.
\end{notation}

\begin{lem}[{\cite[Corollary 5.5]{RH-algebraic}}]\label{lem: limits of algebraic patterns}
	The $\infty$-category $\AlgPatt$ of algebraic patterns admits all limits and filtered colimits and these are preserved by the forgetful functor $\AlgPatt \longrightarrow \Cat_{\infty}$.
\end{lem}

The following corollary shows that the assignment $\calO \longmapsto \Seg_{\calO}(\calC)$ is compatible with filtered colimits of algebraic patterns.

\begin{cor}\label{cor:Segal-objects-for-filtered-colimit}
	Let $\calO_{(-)}:  I \to \AlgPatt$ be a filtered diagram of algebraic patterns and let $\calC$ be a complete $\infty$-category. Then the natural functor
	\[ \Seg_{\underset{i\in I}{\colim} \calO_i}(\calC) \to \underset{i \in I}{\lim} \Seg_{\calO_i}(\calC), \]
	is an equivalence.
\end{cor}
\begin{proof}
	Let $\calO\coloneq \underset{i\in I}{\colim} \calO_i$ and consider the following natural cube
	\[\begin{tikzcd}[row sep={40,between origins}, column sep={40,between origins}]
	& \Seg_{\calO}(\calC) \ar{rr}\ar{dd}\ar{dl} & & \Fun(\calO,\calC)  \ar{dd}\ar{dl} \\
	\underset{i \in I}{\lim}\Seg_{\calO_i}(\calC)  \ar[crossing over]{rr} \ar{dd} & &	\underset{i \in I}{\lim} \Fun(\calO_i,\calC) \\
	& \Fun(\calO^{\el},\calC) \ar[pos=0.7]{rr} \ar{dl} & & \Fun(\calO^{\int},\calC) \ar{dl} \\
	\underset{i \in I}{\lim}\Fun(\calO^{\el}_i,\calC) \ar{rr} && \underset{i \in I}{\lim}\Fun(\calO^{\int}_i,\calC) \ar[from=uu,crossing over]
	\end{tikzcd}\]
	The back face is cartesian by \cref{rem:alternative-definition-of-Seg}. The front face is a limit of cartesian squares by \cref{rem:alternative-definition-of-Seg} and is therefore also cartesian. On the other hand the bottom left, bottom right and top right edge are all equivalences by \cref{lem: limits of algebraic patterns}. Consequently the top left edge is an equivalence.
\end{proof}

\begin{war}
	If $f: \calO \to \calP$ is a morphism of algebraic patterns, then $f^{\ast}: \Fun(\calP,\calC) \to \Fun(\calO,\calC)$ need not preserve Segal objects. For a counterexample see \cite[Remark 4.6]{RH-algebraic}.
\end{war}

\begin{defn}[{\cite[Definition 4.2 \& Lemma 4.5]{RH-algebraic}}]\label{defn: definition of Segal morphism}
	A morphism of algebraic patterns $f: \calO \to \calP$ is called a \textit{Segal morphism} if the following equivalent conditions are satisfied
	\begin{enumerate}
		\item For every $\infty$-category $\calC$ admitting all $\calK_{\calO}$ and $\calK_{\calP}$-limits the functor $f^{\ast}: \Fun(\calP,\calC) \to \Fun(\calO,\calC)$ restricts to Segal objects:
		\[f^{\ast} : \Seg_{\calP}(\calC) \to \Seg_{\calO}(\calC)\]
		\item The functor $f^{\ast}: \Fun(\calP,\calS) \to \Fun(\calO,\calS)$ restricts to  Segal objects:
		$$f^{\ast} : \Seg_{\calP}(\calS) \to \Seg_{\calO}(\calS)$$
		\item\label{part (2) of definition of Segal morphism} For every $x \in \calO$ and $F \in \Seg_{\calP}(\calS)$ the natural map,
		\[\underset{\calP^{\el}_{f(x)/}}{\lim} F \to \underset{\calO^{\el}_{x/}}{\lim} F \circ f^{\el},\] 
		is an equivalence.
	\end{enumerate} 
\end{defn}

\begin{defn}
	Let $f: \calO \to \calP$ be a morphism of algebraic patterns. An $\infty$-category $\calC$ is said to be \textit{$f$-complete} if the following conditions hold:
	\begin{enumerate}
		\item $\calC$ admits all $\calK_{\calO}$ and $\calK_{\calP}$-limits (see \ref{defn:el-diagrams}).
		\item For every Segal $\calO$-object $F: \calO \to \calC$ and every $y \in \calP$ the following limit exists
		\[\!\!\!\!\!\! \underset{(x,y \to f(x)) \in \calO \times_{\calP}\calP_{y/}}{\lim}\!\!\!\!\! F(x) \in \calC.\]
	\end{enumerate}
\end{defn}
 
\begin{example}
	Let $\calO$ be an algebraic pattern and let $i_{\calO}: \calO^{\el} \mono \calO^{\int}$ denote the inclusion of the elementary objects. An $\infty$-category $\calC$ is $i_{\calO}$-complete if and only if it admits all $\calK_{\calO}$-limits.
\end{example}

Different algebraic patterns can give rise to equivalent algebraic structures. To make this notion of equivalence precise we introduce a variation on \cite[Definition 10.1]{RH-cartesian}.

\begin{prop}\label{prop:morita-equivalence}
	Let $f: \calO \to \calP$ be a morphism of algebraic pattern.
	The following are equivalent:
	\begin{enumerate}
		\item The adjunction $f^{\ast}: \Fun(\calP,\calS^{\le n}) \adj \Fun(\calO,\calS^{\le n}):f_{\ast}$ restricts to an equivalence of $\infty$-categories:
		$$f^{\ast} : \Seg_{\calP}(\calS^{\le n}) \simeq \Seg_{\calO}(\calS^{\le n}) : f_{\ast}$$
		\item For every $\infty$-category $\calA$ the adjunction $f^{\ast}: \Fun(\calP,\Fun(\calA,\calS^{\le n})) \adj \Fun(\calO,\Fun(\calA,\calS^{\le n})):f_{\ast}$ restricts to an equivalence of $\infty$-categories:
		\[ f^{\ast} : \Seg_{\calP}(\Fun(\calA,\calS^{\le n})) \simeq \Seg_{\calO}(\Fun(\calA,\calS^{\le n})) : f_{\ast}, \]
		\item For every $f$-complete $(n+1,1)$-category $\calC$ restriction $f^{\ast} : \Fun(\calP,\calC) \to \Fun(\calO,\calC)$ preserves Segal objects and induces an equivalence,
		\[ f^{\ast} : \Seg_{\calP}(\calC) \simeq \Seg_{\calO}(\calC) : f_{\ast}, \]
		with the right adjoint given by (pointwise)
		right kan extension (which exists for Segal $\calO$-objects by assumption that $\calC$ is $f$-complete).
	\end{enumerate}
\end{prop}
\begin{proof}
    $(1 \Rightarrow 2)$ Limits in presheaves are pointwise and thus $F: \calO \to \Fun(\calA,\calS^{\le n})$ satisfies the Segal condition if and only if $F(a): \calO \to \calS^{\le n}$ satisfies the Segal condition for every $a \in \calA$. Clearly though, the same holds for $\calP$ in place of $\calO$. Using (1) we then conclude that both $f^{\ast}$ as well as its right adjoint preserve the Segal condition and thus restricts to give an adjunction: 
	\[f^{\ast}: \Seg_{\calP}(\Fun(\calA,\calS^{\le n})) \adj \Seg_{\calO}(\Fun(\calA,\calS^{\le n})):f_{\ast},\]
	which is necessarily an equivalence since the unit and counit evaluate at every $a \in \calA$ to the unit and counit of the adjunction from $(1)$.
    $(2 \Rightarrow 3)$ The mapping spaces in $\calC$ are $n$-truncated and thus Yoneda restricts to fully faithful inclusion $\calC \mono \Fun(\calC^{\op},\calS^{\le n})$. From it we get a limit preserving fully faithful inclusion $\Fun(\calO,\calC) \mono \Fun(\calO,\Fun(\calC^{\op},\calS^{\le n}))$ and similarly for $\calP$. Let $\widehat{f}^{\ast} \dashv \widehat{f}_{\ast}$ denote the adjunction of $(2)$ in the case $\calA \coloneq  \calC^{\op}$. Observe that $\widehat{f}^{\ast}$ restricts to $f^{\ast}$ and since the latter preserves Segal objects so must the former. Let $F: \calO \to \calC$ be a Segal $\calO$-object. By $(2)$ its right kan extension $\widehat{f}_{\ast}F$ is a Segal $\calP$-object of $\Fun(\calC^{\op},\calS^{\le n})$. On the other hand $\calC$ being $f$-complete, implies that that for every $y \in \calP$ we have $(\widehat{f}_{\ast} F)(y) \in \calC$. We have thus shown that $\widehat{f}^{\ast} \dashv \widehat{f}_{\ast}$ restricts to an adjunction: \[f^{\ast}:\Seg_{\calP}(\calS) \adj \Seg_{\calO}(\calC) :f_{\ast}\]
	It remains to observe that by construction the unit and counit of $f^{\ast} \dashv f_{\ast}$ coincide with the unit and counit of $\widehat{f}^{\ast} \dashv \widehat{f}_{\ast}$ which induce equivalences on Segal objects by $(2)$.
    $(3 \Rightarrow 1)$  Immediate from the fact that $\calS^{\le n}$ is complete.
\end{proof}

\begin{defn}\label{weak Morita equivalence}
	A morphism of algebraic patterns $f: \calO \to \calP$ is called a \textit{Morita $n$-equivalence} if it satisfies the equivalent conditions of \cref{prop:morita-equivalence}.
\end{defn}

\begin{example}
    It is is straightforward to verify by hand that the fully faithful inclusion $\Fin^{\le 3}_{\ast} \mono \Fin_{\ast}$ is a Morita $0$-equivalence of algebraic patterns. Indeed this corresponds to the fact that the textbook definition of commutative monoids in $\mathrm{Set}$ matches with that of a $\Fin_{\ast}^{\le 3}$-Segal object. In fact this is a special case of our main theorem \cref{thm-motivational}.
\end{example}

\begin{lem}\label{lem: essentially surjective on elementary objects implies restriction is conservative}
	Let $f: \calO \to \calP$ be a morphism of algebraic patterns and let $\calC$ be an $\infty$-category admitting all $\calK_{O}$ and $\calK_{\calP}$-limits. Suppose that:
	\begin{enumerate}
		\item $f^{\el} : \calO^{\el} \to \calP^{\el}$ is essentially surjective.
		\item $f^{\ast} : \Fun(\calP,\calC) \to  \Fun(\calO,\calC) $ preserves Segal objects.
	\end{enumerate}
	Then the restriction of $f^{\ast}$ to Segal objects, $f^{\ast}:\Seg_{\calP}(\calC) \to \Seg_{\calO}(\calC)$,
	is conservative.
\end{lem}
\begin{proof}
	Consider the following diagram of algebraic patterns
	\[\begin{tikzcd}
	{\calO^{\el}} & \calO \\
	{\calP^{\el}} & \calP
	\arrow[from=1-1, to=1-2]
	\arrow[from=1-1, to=2-1]
	\arrow[from=1-2, to=2-2]
	\arrow[from=2-1, to=2-2]
	\end{tikzcd}\]
	Passing to Segal objects yields
	\[\begin{tikzcd}
	\Seg_{\calP}(\calC) & \Fun(\calP^{\el},\calC) \\
	\Seg_{\calO}(\calC) & \Fun(\calO^{\el},\calC)
	\arrow[from=1-1, to=1-2]
	\arrow[from=1-1, to=2-1]
	\arrow[from=1-2, to=2-2]
	\arrow[from=2-1, to=2-2]
	\end{tikzcd}\]
	By \cref{lem: restriction to elementary is conservative}, both horizontal morphisms are conservative. By assumption $\calO^{\el} \to \calP^{\el}$ is essentially surjective and thus the right vertical map is conservative. By composition it follows that $\Seg_{\calP}(\calC) \to \Fun(\calO^{\el},\calC)$ is conservative which by cancellation implies that $\Seg_{\calP}(\calC) \to \Seg_{\calO}(\calC)$ is conservative.
\end{proof}

\begin{rem}\label{rem: composition of Segal}
	Segal morphisms are closed under composition. That is, if $f: \calO \to \calP$ and $g: \calP \to \calQ$ are Segal morphisms of algebraic patterns then their composition $g \circ f : \calO \to \calQ$ is also a Segal morphism of algebraic patterns.
\end{rem}

\begin{defn}\label{defn: definition of iso-Segal morphism}
	Let $f: \calO \to \calP$ be a morphism of algebraic patterns. We say that $f$ is an \textit{iso-Segal morphism} if for every $x \in \calO$ the natural map, $\calO^{\el}_{x/} \to \calP^{\el}_{f(x)/}$, is an equivalence.
\end{defn}

\begin{example}\label{ex:operads-are-strict}
	The structure map $\bracket{-}:\calO \to \Fin_{\ast}$ of an $\infty$-operad is an iso-Segal morphism (see \cite[Example 3.7]{RH-cartesian}).
\end{example}

A useful fact about iso-Segal morphisms is their left left cancellation property which we establish in the following lemma.

\begin{lem}\label{lem: cancellation for iso-Segal}
	Let $f: \calO \to \calP$ and $g: \calP \to \calQ$ be morphisms of algebraic patterns. Suppose that $g$ is iso-Segal. Then $f$ is iso-Segal if and only if $g \circ f$ is iso-Segal.
\end{lem}
\begin{proof}
	Let $x \in \calO$ and consider the composite $\calO^{\el}_{x/} \to \calP^{\el}_{f(x)/} \to \calQ^{\el}_{(g\circ f)(x)/}$.
	Since $g$ is iso-Segal the second map is an equivalence. Consequently the first map is an equivalence if and only if the composite is.
\end{proof}

Finally we recall yet another type of morphism which sits between Segal morphisms and iso-Segal morphisms.

\begin{defn}[{\cite[Remark 4.4]{RH-algebraic}}]\label{defn: definition of strong Segal morphism}
	Let $f: \calO \to \calP$ be a morphism of algebraic patterns. We say that $f$ is \textit{strong Segal} if for every $x \in \calO$ the natural map $\calO^{\el}_{x/} \to \calP^{\el}_{f(x)/}$ is initial.
\end{defn}

\begin{rem}\label{rem:different-Segal-notions}
	For the reader's convenience we record here the logical implications between the 3 different types of Segal morphisms
	\[ \hyperref[defn: definition of iso-Segal morphism]{\text{iso-Segal}} \implies \hyperref[defn: definition of strong Segal morphism]{\text{strong Segal}} \implies \hyperref[defn: definition of Segal morphism]{\text{Segal}}\]
\end{rem}

\subsubsection{Algebraic subpatterns}

\begin{prop}\label{prop: faithful subpattern}
	Let $\calQ$ be an algebraic pattern and let $\calP \subseteq \calQ$ be a replete subcategory. Suppose the following conditions hold:
	\begin{enumerate}
		\item $\calP$ is closed under (inert,active) factorizations: if $x \overset{\alpha}{\to} y$ is a morphism in $\calP$ then $x \overset{\alpha^{\act}}{\rightarrowtail} \Lambda(\alpha)$ and $\Lambda(\alpha) \overset{\alpha^{\int}}{\rightsquigarrow} y$ are morphisms in $\calP$.
		
		\item $\calP$ contains all inert morphisms to elementary objects: If $x \in \calP$, and $\alpha : x \inert e \in \calQ^{\el}_{x/}$ then $\alpha$ is a morphisn in $\calP$.
		
	\end{enumerate} 
	Then there exists a unique structure of an algebraic pattern on $\calP$ for which the inclusion $\calP \mono \calQ$ is a iso-Segal morphism of algebraic patterns.
\end{prop}
\begin{proof}
	Define the inert and active morphisms in $\calP$ respectively as
	\[ \calP^{\int}\coloneq \calP \cap \calQ^{\int}, \quad \text{ and } \quad \calP^{\act} \coloneq  \calP \cap \calQ^{\act}.\]
	We claim that these form a factorization system on $\calP$. 
    Indeed, $f: \calP \to \calQ$ preserves inert and active morphisms by construction and is faithful by assumption. Since (inert,active) factorizations are unique in $\calQ$ they must be unique in $\calP$ as well. Existence follows from (1). 
	
	Define the elementary objects of $\calP$ by $\calP^{\el}\coloneq  \calP^{\int} \cap \calQ^{\el}$.
    It remains to show that the inclusion $\calP \mono \calQ$ is iso-Segal morphism of algebraic patterns. The inclusion $\calP^{\int} \mono \calQ^{\int}$ is faithful by construction and thus the natural map $\calP^{\el}_{x/} \to \calQ^{\el}_{x/}$ is fully faithful (see \cref{lem:fully-faithful-on-clice-from-faithful}). In addition it is essentially surjective by (2) and thus an equivalence.
\end{proof}

\begin{defn}\label{defn:algebraic-subpattern}
	In the situation of \cref{prop: faithful subpattern} we say that $\calP \subseteq \calQ$ is the inclusion of a (replete) \textit{algebraic subpattern}. 
\end{defn}

\begin{rem}\label{cor:subpattern-properties}
	We collect some useful properties of algebraic subpatterns proven below.
	\begin{itemize}
		\item \textbf{Strictness:} If $\calP \subseteq \calQ$ is an algebraic subpattern. Then the inclusion $\calP \mono \calQ$ is a iso-Segal morphism of algebraic patterns.
		\item \textbf{Transitivity} (\ref{cor:transitivity}): If $\calQ$ is an algebraic pattern and $\calO \subseteq \calP \subseteq \calQ$ are replete subcategories such that $\calP \subseteq \calQ$ and $\calO \subseteq \calP$ are algebraic subpatterns. Then $\calO \subseteq \calP$ is a an algebraic subpattern.
		\item \textbf{Base change} (\ref{lem:basechange-subpattern}): If $\calP' \subseteq \calP$ is an algebraic subpattern and $f: \calO \to \calP$ is a morphism of algebraic patterns. Then $\calO \times_{\calP} \calP' \subseteq \calO$ is an algebraic subpattern. 
	\end{itemize}
\end{rem}

\begin{cor}\label{cor:transitivity}
	If $\calQ$ is an algebraic pattern and $\calO \subseteq \calP \subseteq \calQ$ are faithful subcategories such that $\calP \subseteq \calQ$ and $\calO \subseteq \calP$ are algebraic subpatterns. Then $\calO \subseteq \calP$ is a an algebraic subpattern.
\end{cor}

\begin{example}
	Let $\calO$ be an algebraic pattern. Then it's straightforward to check that $\calO^{\el} \subseteq \calO^{\int} \subseteq \calO$ is a nested inclusion of algebraic subpatterns.
\end{example}

\begin{example}\label{ex:Fin<=k-subpattern}
	Recall $\Fin^{\le k}_{\ast} \subseteq \Fin_{\ast}$ was defined as the full subcategory of pointed finite sets $A_+$ with $|A|\le k$. It's straightforrward to check that $\Fin^{\le k}_{\ast}$ satisfies the conditions of \cref{prop: faithful subpattern} and thus defines an algebraic subpattern. 
\end{example}

\begin{lem}\label{lem:basechange-subpattern}
	Let $\calP' \subseteq \calP$ be an algebraic subpattern and let $f: \calO \to \calP$ be a morphism of algebraic patterns. Then the pullback $\calO' \coloneq  \calO \times_{\calP} \calP' \subseteq \calO$ is an algebraic subpattern.
\end{lem}
\begin{proof}
	By the universal property of the pullback a morphism $\alpha: x \to y$ in $\calO$ lies in the subcategory $\calO' \subseteq \calO$ if and only if its image $f(\alpha): f(x) \to f(y)$ lies in the subcategory $\calP' \subseteq \calP$. Using this it is straightforward to check the conditions of \cref{prop: faithful subpattern} are satisfied for $\calO' \subseteq \calO$. For the benefit of the reader we provide the necessary details.
 
    $(1)$ Let $\alpha: x \to y$ be a morphism in $\calO'$. Since $\calP' \subseteq \calP$ is a subpattern and $f(\alpha)$ is a morphism in $\calP'$ we see that $f(\alpha^{\int}) = f(\alpha)^{\int}$ and $f(\alpha^{\act}) = f(\alpha)^{\act}$ are morphisms in $\calP'$. It follows that $\alpha^{\int}$ and $\alpha^{\act}$ are morphisms in $\calO'$ as desired.
    
    $(2)$ Let $\lambda: x \inert e$ be an inert morphism in $\calO$ with $x \in \calP'$ and $e \in \calO^{\el}$. Its image $f(\lambda) : f(x) \inert f(e)$ is an inert morphism in $\calP$ with $f(x) \in \calP'$ and $f(e) \in \calP^{\el}$. Since $\calP' \subseteq \calP$ is a subpattern it follows that $f(\lambda)$ is a morphism in $\calP'$. Consequently $\lambda$ is a morphism in $\calP'$ as desired.
\end{proof}

The following simplified version of \cref{prop: faithful subpattern} for the case of a full subcategory is straightforward.

\begin{lem}\label{lem:fully-faithful-subpattern}
	Let $\calQ$ be an algebraic pattern and let $\calP \subseteq \calQ$ be a full subcategory. Suppose
	\begin{enumerate}
		\item If $x \inert e$ is an inert morphism such that $x \in \calP$ and $e \in \calQ^{\el}$ then $e \in \calP$
		
		\item If $\alpha : x \to y$ is a morphism such that $x,y \in \calP$ then $\Lambda(\alpha) \in \calP$.
	\end{enumerate} 
	Then $\calP \subseteq \calQ$ is an algebraic subpattern.
\end{lem}

\subsection{Slice patterns and weakly initial morphisms}

We begin this subsection with a brief discussion of the problem of recognizing Morita equivalences. 
Motivated by this we proceed to study \textit{slices} of algebraic patterns and their canonical algebraic pattern structure (\cref{rem:algebraic-pattern-structure-on-slice}). 
We then introduce \textit{weakly initial} objects (\cref{defn:weakly-initial-object}) and morphisms (\cref{prop: Quillen A for Segal}) as tools for computing limits of Segal objects. 
These notions will be used in \cref{subsect:subsection-II.3} where we prove the \textit{recognition theorem} for Morita equivalences of algebraic patterns (\cref{thm:Morita-equivalence-from-weakly-initial-slices}).

\subsubsection{Recognizing Morita equivalences}\label{Motivational-Discussion-Recognition}

Let $f: \calO \to \calP$ be a Segal morphism of algebraic patterns (see \cref{defn: definition of Segal morphism}) so that $f^{\ast}$ preserves Segal objects. Recall that our end goal is to develop tools to identify when the restriction functor $f^{\ast} : \Seg_{\calP}(\calS^{\le n}) \to \Seg_{\calO}(\calS^{\le n})$ is an equivalence. Before attacking this problem, we add a specializing assumption. Instead of asking for the existence of \textit{some} inverse to $f^{\ast}$ we ask whether the right Kan extension $f_{\ast}$ is its inverse. Equivalently we ask whether $f: \calO \to \calP$ is a Morita $n$-equivalence in the sense of \cref{weak Morita equivalence}. For this to happen, $f_{\ast}$ must in particular preserve $n$-truncated Segal objects. We are thus led to the question of finding sufficient condtions for $f_{\ast} : \Fun(\calO,\calS^{\le n}) \to \Fun(\calP,\calS^{\le n})$ to preserves Segal objects. Let $F \in \Seg_{\calO}(\calS)$ and $y \in \calP$ and consider the formula for right Kan extension:
\[(f_{\ast}F)(y) \simeq \underset{(x,y \to f(x)) \in \calO_{y/}}{\lim}F(x).\]
Generally there is not much more to be said about the right hand side. In the setting of algebraic patterns, however, the right hand side is more structured than might seem at first glance. Specifically we shall see that $\calO_{y/}$ admits a canonical structure of an algebraic pattern (\cref{rem:algebraic-pattern-structure-on-slice}) with respect to which $F|_{\calO_{y/}}$ is a Segal object (see \cref{lem: projection from slice is iso-Segal}). 

\subsubsection{Slice patterns}

\begin{prop}\label{prop:-slice-under-is-a-algebraic-pattern}
	Let $\calO$ be an algebraic pattern. For every $x \in \calO$ there's a natural algebraic pattern structure on $\calO_{x /}$ having the following properties
	\begin{enumerate}
		\item An object $x \to e \in \calO_{x/}$ is elementary if and only if $e$ is elementary.
		\item Evaluation at the source $s: \calO_{x/} \to \calO$ is a morphism of algebraic patterns.
		\item There's a canonical equivalence $(\calO_{x/})^{\el} = \calO_{x/} \times_{\calO} \calO^{\el}$.
		\item For every morphism of algebraic patterns $f: \calO \to \calP$ the natural functor
		$\calO_{x/} \to \calP_{f(x)/}$
		is also a morphism of algebraic patterns.
	\end{enumerate}
\end{prop}
\begin{proof}
	Let the (inert,active) factorization system on $\calO_{x/}$ be that of \cref{prop:factorization-system-on-slice}. Define the elementary objects of $\calO_{x/}$ to be those $\alpha: x \to e$ such that $e$ is elementary.
	Property (1) and (2) hold by construction. Properties (3) and (4) follow from \cref{lem: limits of algebraic patterns} together with the fact that $f$ preserves (inert,active) factorizations and elementary objects and therefore respects all the constructions.
\end{proof}

\begin{rem}\label{rem:algebraic-pattern-structure-on-slice}
	Let $f: \calO \to \calP$ is a morphism of algebraic patterns and $y \in \calP$. By \cref{prop:-slice-under-is-a-algebraic-pattern} and \cref{lem: limits of algebraic patterns}, the slice category $\calO_{y/} \coloneq  \calO \times_{\calP} \calP_{y/}$ admits a canonical structure of an algebraic pattern.
\end{rem}

\begin{lem}\label{lem: projection from slice is iso-Segal}
	Let $f: \calO \to \calP$ be a morphism of algebraic patterns and let $y \in \calP$. Then the natural forgetful functor
    $\calO \times_{\calP} \calP_{y/} \to \calO $
	is an iso-Segal morphism of algebraic patterns:
	\begin{itemize}
		\item For every $(x,\alpha: y \to f(x))\in \calO \times_{\calP} \calP_{y/}$ the natural map
		$( \calO \times_{\calP} \calP_{y/})^{\el}_{\alpha/} \to \calO^{\el}_{x/}$
		is an equivalence.
	\end{itemize} 
\end{lem}
\begin{proof}
	Consider the following cube
	\[\begin{tikzcd}
	& {(\calP_{y/})^{\el}} && {(\calP_{y/})^{\int}} \\
	& {\calP^{\el}} && {\calP^{\int}} \\
	{(\calO \times_{\calP} \calP_{y/})^{\el}} && {(\calO \times_{\calP} \calP_{y/})^{\int}} \\
	{\calO^{\el}} && {\calO^{\int}}
	\arrow[from=3-1, to=4-1]
	\arrow[from=3-1, to=3-3]
	\arrow[from=3-3, to=4-3]
	\arrow[from=4-1, to=4-3]
	\arrow[from=3-1, to=1-2]
	\arrow[from=3-3, to=1-4]
	\arrow[from=4-3, to=2-4]
	\arrow[from=4-1, to=2-2]
	\arrow[from=1-4, to=2-4]
	\arrow[from=1-2, to=1-4]
	\arrow[from=2-2, to=2-4]
	\arrow[from=1-2, to=2-2]
	\end{tikzcd}\]
	The left and right faces are pullback squares by \cref{lem: limits of algebraic patterns}. The back face is a pullback square by definition (see \cref{prop:-slice-under-is-a-algebraic-pattern}). It follows that the front face is also a pullback square.
	
	Let $(x,\alpha : y \to f(x)) \in \calO \times_{\calP} \calP_{y/}$ and consider the following natural diagram
	\[\begin{tikzcd}
	{(\calO \times_{\calP} \calP_{y/})^{\el}_{\alpha/}} & {(\calO \times_{\calP} \calP_{y/})^{\el}} & {\calO^{\el}} \\
	{(\calO \times_{\calP} \calP_{y/})^{\int}_{\alpha/}} & {(\calO \times_{\calP} \calP_{y/})^{\int}} & {\calO^{\int}}
	\arrow[from=1-2, to=2-2]
	\arrow[from=1-2, to=1-3]
	\arrow[from=1-3, to=2-3]
	\arrow[from=2-2, to=2-3]
	\arrow[from=2-1, to=2-2]
	\arrow[from=1-1, to=2-1]
	\arrow[from=1-1, to=1-2]
	\end{tikzcd}\]
	The left square is a pullback by definition and right square is a pullback by the previous paragraph. By pasting the large rectangle is also a pullback. By \cref{lem: double slice} there's a canonical equivalence $(\calO \times_{\calP} \calP_{y/})_{\alpha/} \simeq \calO_{x/}$. 
	This equivalence restricts to an equivalence $(\calO \times_{\calP} \calP_{y/})^{\int}_{\alpha/} \simeq \calO_{x/}^{\int}$. Substituting these in the outer rectangle above we get the following pullback square
	\[\begin{tikzcd}
	{(\calO \times_{\calP} \calP_{y/})^{\el}_{\alpha/}} & {\calO^{\el}} \\
	{\calO^{\int}_{x/}} & {\calO^{\int}}
	\arrow[from=1-1, to=2-1]
	\arrow[from=1-1, to=1-2]
	\arrow[from=1-2, to=2-2]
	\arrow[from=2-1, to=2-2]
	\end{tikzcd}\]
	In other words the natural map
	\[(\calO \times_{\calP} \calP_{y/})^{\el}_{\alpha/} \to \calO^{\el}_{x/}\]
	is an equivalence. 
\end{proof}

\begin{cor}\label{cor: map on slices is iso-Segal}
	Let $f: \calO \to \calP$ be a morphism of algebraic patterns. Then for every $x \in \calO$ the natural functor $\calO_{x/} \to \calO \times_{\calP} \calP_{f(x)/}$
	is an iso-Segal morphism of algebraic patterns.
\end{cor}
\begin{proof}
	Let $x \in \calO$ and consider the composite $\calO_{x/} \to \calO \times_{\calP} \calP_{f(x)/} \to \calO$.
    The second map and the composite are iso-Segal by \cref{lem: projection from slice is iso-Segal}. By  \cref{lem: cancellation for iso-Segal} it follows that the first map is also iso-Segal.
\end{proof}

\begin{cor}\label{cor: map from inert to non-inert slice is iso-Segal}
	Let $f: \calO \to \calP$ be a morphism of algebraic patterns. Then for every $y \in \calP$ the natural map $\calO^{\int} \times_{\calP^{\int}} \calP^{\int}_{y/} \to\calO \times_{\calP} \calP_{y/}$
	is a iso-Segal morphism of algebraic patterns.
\end{cor}
\begin{proof}
	Consider the following commutative square
	\[\begin{tikzcd}
	{\calO^{\int} \times_{\calP^{\int}} \calP^{\int}_{y/}} & {\calO \times_{\calP} \calP_{y/}} \\
	{\calO^{\int}} & \calO.
	\arrow["{f^{\int}}"', from=1-1, to=2-1]
	\arrow["f", from=1-2, to=2-2]
	\arrow[from=2-1, to=2-2]
	\arrow[from=1-1, to=1-2]
	\end{tikzcd}\]
	The left and right vertical morphisms are iso-Segal by \cref{lem: projection from slice is iso-Segal}. It's straightforward to check the bottom horizontal morphism is iso-Segal. Consequently by \cref{lem: cancellation for iso-Segal} the top horizontal map is iso-Segal.	
\end{proof}

\subsubsection{Weakly initial morphisms}

The following generalizes the notion of initial objects to the setting of algebraic patterns.

\begin{defn}\label{defn:weakly-initial-object}
	Let $-2 \le n \le \infty$ be an integer and let $x_0 \in \calO$. We say that $x_0$ is \textit{weakly $n$-initial} if for every $n$-truncated Segal object $F \in \Seg_{\calO}(\calS^{\le n})$ the natural map,
	$\lim_{x\in \calO} F(x) \to F(x_0),$
	is an equivalence. In the special case $n =\infty$ we omit it and simply say that $x$ is \textit{weakly initial}.
\end{defn}

\begin{example}
	If $x \in \calO$ is initial then $x$ is in particular weakly initial.
\end{example}

\begin{lem}\label{lem: weakly initial objects are Morita invariant} 
	Let $f: \calO \to \calP$ be a Morita $n$-equivalence of algebraic patterns. Then $x_0 \in \calO$ is weakly $n$-initial if and only if its image $f(x_0) \in \calP$ is weakly $n$-initial. 
\end{lem}
\begin{proof}
	Since $f^{\ast}: \Seg_{\calP}(\calS^{\le n}) \to \Seg_{\calO}(\calS^{\le n})$ is essentially surjective it suffices to show that for every $F \in \Seg_{\calP}(\calS^{\le n})$ and every $x_0 \in \calO$ the composition 
	\begin{align*}
	\underset{\calP}{\lim} F \to \underset{\calO}{\lim} F \circ f \to F(f(x_0))
	\end{align*} 
	is an equivalence if and only if the second map is an equivalence. Indeed we claim that the first map is an equivalence. To see this note that it factors as
	\begin{align*}
	\underset{\calP}{\lim} F  \simeq \Map_{\Seg_{\calP}(\calS^{\le n})}(\pt,F) \iso  \Map_{\Seg_{\calO}(\calS^{\le n})}(f^{\ast}\pt,f^{\ast}F) \simeq \underset{\calO}{\lim} F \circ f
	\end{align*}
	where the middle map is an equivalence since $f^{\ast}: \Seg_{\calP}(\calS^{\le n}) \to \Seg_{\calO}(\calS^{\le n})$ is fully faithful.
\end{proof}

Quillen's Theorem A \cite{Quillen-A}, gives necessary and sufficient conditions for a functor $f:\calC \to \calD$ (of $1$-categories) to induce an equivalence on colimits $\varinjlim(p\circ f) \simeq \varinjlim(p)$ for all diagrams $p:\calD \to \calE$ (where $\calE$ is a $1$-category). In \cite[\S 4.1.3.]{HTT}, Lurie formulated and proved a variant of Quillen's Theorem A, attributed to Joyal, where $\calC$,  $\calD$ and $\calE$ are taken to be $\infty$-categories. 
For future use we record a variant of Quillen A interpolating between $1$-categorical and $\infty$-categorical.

\begin{prop}[Quillen A]\label{prop: Quillen A}
	Let $-2 \le n \le \infty$ be an integer and $q: \calI \to \calJ$ a functor of $\infty$-categories. The following are equivalent:
	\begin{enumerate}
		\item For every $(n+1,1)$- category $\calC$ and every functor $F: \calJ \to \calC$ the natural map $\underset{\calJ}{\lim} F \to \underset{\calI}{\lim}  F \circ q$
		is an equivalence whenever either side exist.
		
		\item For every functor $F: \calJ \to \calS^{\le n}$ the limit of the unit map induces an equivalence $\underset{\calJ}{\lim} F \simeq \underset{\calJ}{\lim}  q_{\ast} q^{\ast}F$.
		
		\item The left Kan extension of the terminal functor $q_!(\pt)$ is $n$-connected.
		
		\item For every $j \in \calJ$ the space $\big|\calI \times_\calJ \calJ_{/j}\big|$ is $n$-connected.
	\end{enumerate}
	When these conditions are satisfied we say that $q: \calI \to \calJ$ is \hl{$n$-initial}.
\end{prop}
\begin{proof}
	$(1 \Rightarrow 2)$ Apply (1) in the special case where $\calE = \calS^{\le n}$.
	$(2 \Rightarrow 1)$ Let $\calE$ be an $n$-truncated $\infty$-category and $F: \calJ \to \calE$ a functor. The yoneda embedding $y: \calE \to \Fun(\calE^{\op},\calS)$ is fully faithful and preserves all limits. It thus suffices to show the claim after applying yoneda. Furthermore since $\calE$ is $n$-truncated by assumption the yoneda embedding lands in presheaves taking values in $n$-truncated spaces. Since equivalences of presheaves can be checked levelwise and evaluation preserves limits we may assume without loss of generality that $\calE =\calS^{\le n}$. Finally (1) for the case $\calE = \calS^{\le n}$ is simply (2) so we're done.
	$(2 \Leftrightarrow 3)$ Note that since $q^{\ast} \pt  \simeq \pt$ we have $q_! \pt \simeq q_!q^{\ast} \pt$. We claim that the counit map $q_! q^{\ast}\pt \to \pt$ corepresents the comparison map of (2). Indeed for every functor $F: \calJ \to \calS$ we have
	\begin{align*}
	\underset{\calJ}{\lim} F \simeq \Map_{\Fun(\calJ,\calS)}(\pt,F) & \to \Map_{\Fun(\calJ,\calS)}(q_! q^{\ast}\pt,F) \\
	& \simeq \Map_{\Fun(\calJ,\calS)}(\pt,q_{\ast} q^{\ast}F) \\
	& \simeq \underset{\calJ}{\lim} q_{\ast} q^{\ast} F
	\end{align*}
	Consequently $q_! q^{\ast} \pt \to \pt$ is $n$-connected if and only if  the comparison map $\underset{\calJ}{\lim} F \to \underset{\calJ}{\lim}  q_{\ast} q^{\ast}F$ is an equivalence for every $n$-truncated functor $F : \calJ \to \calS^{\le n}$.
	$(3 \Leftrightarrow 4)$ Follows from the formula for left Kan extension. Indeed for every $j \in \calJ$ we have
	$(q_!(\pt))(j)  \simeq \underset{i \in \calI_{/j}}{\colim} \pt \simeq \big|\calI \times_\calJ \calJ_{/j}\big|$.
\end{proof}

\begin{prop}\label{prop: Quillen A for Segal}
	Let $-2 \le n \le \infty$ be a natural number and $f: \calO \to \calP$ a morphism of algebraic patterns. The following are equivalent:
	\begin{enumerate}
		\item 
        For every complete $(n+1,1)$-category $\calC$ and every Segal object $F: \calP \to \calC$ the natural map $\underset{\calP}{\lim} F \to \underset{\calO}{\lim}  F \circ f$
        is an equivalence whenever either side exists.
		
		\item For every Segal object $F: \calP \to \calS^{\le n}$ the natural map $\underset{\calP}{\lim} F \to \underset{\calO}{\lim} f^{\ast} F$ is an equivalence.
	\end{enumerate}
	When these conditions are satisfied we say that $f$ is \hl{weakly $n$-initial}. When $n =\infty$ we omit it and instead say that $f$ is \hl{weakly initial}.
\end{prop}
\begin{proof}
	Carry along the Segal condition in the proof of $(1 \Leftrightarrow 2)$ in \cref{prop: Quillen A}.
\end{proof}

\begin{example}\label{lem: inclusion of elementary is weakly initial}
	For every algebraic pattern $\calO$ the inclusion $i_{\calO}: \calO^{\el} \mono \calO^{\int}$ induces an equivalence \[i_{\calO}^{\ast}: \Seg_{\calO^{\int}}(\calS) \simeq  \Fun(\calO^{\el},\calS) 
	: i_{\calO,\ast}.\]
	In particular $i_{\calO}$ is always weakly initial.
\end{example}

\begin{rem}
	Note that if $f : \calO \to \calP$ is a morphism of algebraic patterns whose underlying functor is $n$-initial then $f$ is weakly $n$-initial.
\end{rem}

\begin{lem}\label{lem:fully-faithful-restriction-implies-weakly-initial}
	Let $f: \calO \to \calP$ be a Segal morphism of algebraic patterns and suppose that the functor
	\[f^{\ast} :  \Seg_{\calP}(\calS^{\le n}) \to \Seg_{\calO}(\calS^{\le n}),\] 
	is fully faithful. Then $f$ is weakly initial. 
\end{lem}
\begin{proof}
	For every $F \in \Seg_{\calP}(\calS)$ we have
	\[ \Map_{\Seg_{\calP}(\calS)}(\pt,F)  \simeq \underset{\calP}{\lim} F \to \underset{\calO}{\lim} F \circ f \simeq  \Map_{\Seg_{\calO}(\calS)}(f^{\ast}\pt,f^{\ast}F).\]
	Since $f^{\ast}$ is fully faithful the composite is an equivalence.
\end{proof}

Since initial objects are preserved by initial functors it is natural to expect that $n$-initial objects are preserved by weakly $n$-initial morphisms. The following lemma shows that this is indeed the case provided the morphism in question is a Segal morphism.

\begin{lem}\label{lem: weakly initial morphisms preserve weakly initial objects}
	Let $f: \calO \to \calP$ be a Segal morphism of algebraic patterns and let $x_0\in \calO$ be weakly $n$-initial. If $f$ is weakly $n$-initial then $f(x_0) \in \calP$ is weakly $n$-initial.
\end{lem}
\begin{proof}
	Let $F \in \Seg_{\calP}(\calS^{\le n})$ and consider the composite $\underset{\calP}{\lim} F \to \underset{\calO}{\lim} F \circ f  \to F(f(x_0))$.
    The second map is an equivalence since $f$ is a Segal morphism. Consequently the first map is an equivalence if and only if the composite is.
\end{proof}

\begin{war}
	Weakly $n$-initial morphisms are not closed under composition in general. Instead we have the following behavior.
\end{war}

\begin{lem}\label{lem: cancellation for weakly initial}
	Let $f: \calO \to \calP$ and $g: \calP \to \calQ$ be morphisms of algebraic patterns. Suppose that
	\begin{itemize}
		\item $f$ is weakly $n$-initial.
		\item $g$ is Segal.
	\end{itemize} 
	Then $g$ is weakly $n$-initial if and only if $g \circ f$ is weakly $n$-initial.
\end{lem}
\begin{proof}
	Let $F \in \Seg_{\calQ}(\calS^{\le n})$ and consider the composite ${\lim}_{\calQ} F \to {\lim}_{\calP} F \circ g \to {\lim}_{\calO} F \circ g \circ f$.
    Since $g$ is Segal, $F \circ g \in \Seg_{\calP}(\calS^{\le n})$ and since $f$ is weakly $n$-initial the second map is an equivalence. 
    By $2$ out of $3$ it follows the first map is an equivalence if and only if the composite is.
\end{proof}

\begin{lem}\label{lem: identity is weakly initial in slice when exists}
	Let $f: \calO \to \calP$ be a morphism of algebraic patterns inducing an equivalence on elementary object $f^{\el} : \calO^{\el} \simeq \calP^{\el}$ on elementary objects. Then for every $x \in \calO$ the following conditions are equivalent
	\begin{enumerate}
		\item The natural map $\calO^{\el}_{x/} \to \calP^{\el}_{f(x)/}$ is initial.
		\item $(x,\id_{f(x)}) \in \calO^{\int}_{f(x)/}$ is weakly initial.
	\end{enumerate}
\end{lem}
\begin{proof}
	Consider the following diagram
	\[ (\calO^{\int}_{f(x)/})^{\el}_{\id_{f(x)}/} \to \calO^{\el}_{x/} \to (\calO^{\int}_{f(x)/})^{\el} \to \calP^{\el}_{f(x)/} \to \calO^{\int}_{f(x)/}.\]
	The first map is an equivalence by \cref{lem: projection from slice is iso-Segal}. We claim that the third map is also an equivalence. To see this note that since $\calO^{\el} \simeq \calP^{\el}$ we have
	\[ (\calO^{\int}_{f(x)/})^{\el} \simeq (\calO^{\int} \times_{\calP^{\int}} \calP^{\int}_{f(x)/})^{\el} \simeq \calO^{\el} \times_{\calP^{\el}} \calP^{\el}_{f(x)/} \simeq \calP^{\el} \times_{\calP^{\el}} \calP^{\el}_{f(x)/} \simeq \calP^{\el}_{f(x)/}.\]
	Let $F \in \Seg_{\calO^{\int}_{f(O)/}}(\calS)$ and consider the comparison maps on limits
	\begin{equation*}\label{diag:long-composition-weakly-initial-lemma}
	    \underset{\calO^{\int}_{f(x)/}}{\lim} \! F \to \underset{\calP^{\el}_{f(x)/}}{\lim} \! F \iso \!\!\! \underset{(\calO^{\int}_{f(x)/})^{\el}}{\lim} \!\!\!\! F \to  \underset{ \calO^{\el}_{x/}}{\lim} \, F \iso \!\!\!\!\!\!\!\!\! \underset{(\calO^{\int}_{f(x)/})^{\el}_{\id_{f(x)}/}}{\lim} \!\!\!\!\!\!\!\!\! F \simeq F(x,\id_{f(x)}), \tag{$\star$}
	\end{equation*} 
	where the last equivalence follows from the Segal condition. The composition of the first two maps can be identified with the comparison map
	\[\underset{\calO^{\int}_{f(x)/}}{\lim} \! F \to \underset{(\calO^{\int}_{f(x)/})^{\el}}{\lim} \!\!\! F.\] 
	We claim that this map is an equivalence. Indeed, since $F$ is Segal $F|_{\calO^{\int}_{f(x)/}}$ is the right Kan extension of $F|_{(\calO^{\int}_{f(x)/})^{\el}}$ (see \cref{lem: Segal iff right Kan extended}) and thus the comparison map on limits is an equivalence. 
	
	Looking back at \eqref{diag:long-composition-weakly-initial-lemma} we see that the total composite is an equivalence if and only if
	\[\psi_x : \underset{\calP^{\el}_{f(x)/}}{\lim} F \to  \underset{\calO^{\el}_{x/}}{\lim} F,\]
	is an equivalence.
	Hence, $(x,\id_{f(x)}) \in \calO^{\int}_{f(x)/}$ is weakly initial if and only if $\psi_x$ is an equivalence for every $F \in \Seg_{\calO^{\int}_{f(x)/}}(\calS)$. Note however that we have equivalences
	\[ \Seg_{\calO^{\int}_{f(x)/}}(\calS) \simeq \Fun((\calO^{\int}_{f(x)/})^{\el},\calS) \simeq \Fun(\calP^{\el}_{f(x)/},\calS).\]
	The first by \cref{lem: Segal iff right Kan extended} and the second by what we already showed. We conclude that instead of quantifying over $F \in \Seg_{\calO^{\int}_{f(x)/}}(\calS)$ we may quantify over $F \in \Fun(\calP^{\el}_{f(x)/},\calS)$. We conclude that $(x,\id_{f(x)}) \in \calO^{\int}_{f(x)/}$ is a weakly initial object if and only if the map $\calO^{\el}_{x/} \to \calP^{\el}_{f(x)/}$ is initial.
\end{proof}

As an immediate consequence of \cref{lem: identity is weakly initial in slice when exists} we have the following corollary.

\begin{cor}
	Let $f: \calO \to \calP$ be a morphism of algebraic patterns such that $f^{\el}:\calO^{\el} \to \calP^{\el}$ is an equivalence. Then the following are equivalent:
	\begin{enumerate}
		\item $f$ is a strong Segal morphism.
		\item $(x,\id_{f(x)}) \in \calO^{\int}_{f(x)/}$ is weakly initial for every $x \in \calO$, 
	\end{enumerate}
\end{cor}

\begin{prop}\label{prop: weakly initial slice from weakly initial object}
	Let $f: \calO \to \calP$ be a strong Segal morphism of algebraic patterns inducing an equivalence on elementary objects $f^{\el} : \calO^{\el} \simeq \calP^{\el}$. Then for every $x \in \calO$ the following are equivalent:
	\begin{enumerate}
		\item $\calO^{\int}_{f(x)/} \to \calO_{f(x)/}$ is weakly $n$-initial.
		
		\item $(x,\id_{f(x)}) \in  \calO_{f(x)/}$ weakly $n$-initial.
	\end{enumerate}
\end{prop}
\begin{proof}
	Let $F \in \Seg_{\calO_{f(x)/}}(\calS^{\le n})$ and consider the natural maps
    ${\lim}_{\calO_{f(x)/}} F \to {\lim}_{\calO^{\int}_{f(x)/}} F \to F(x,\id_{f(x)})$.
	By \cref{cor: map from inert to non-inert slice is iso-Segal}, $\calO^{\int}_{f(x)/} \to \calO_{f(x)/}$ is a iso-Segal morphism and thus $F|_{\calO^{\int}_{f(x)/}} \in \Seg_{\calO^{\int}_{f(x)/}}(\calS^{\le n})$. Since $f$ is strong Segal and satisfies the hypotheses of \cref{lem: identity is weakly initial in slice when exists} we conclude that \linebreak $(x,\id_{f(x)}) \in \calO^{\int} \times_{\calP^{\int}} \calP^{\int}_{f(x)/}$ is weakly initial. Consequently the second map above is an equivalence, and thus the composition is an equivalence if and only if the first map is.
\end{proof}

\subsection{A recognition theorem for Morita equivalences}\label{subsect:subsection-II.3}

In this subsection we prove a \textit{recognition theorem for Morita equivalences} (\cref{thm:Morita-equivalence-from-weakly-initial-slices}) using the tools we developed so far. We also record a convenient formulation of a weaker version of it for later use (\cref{thm: Morita equivalence theorem for faithful subcategories}). As we observed in the previous subsection (see \cref{Motivational-Discussion-Recognition}), the question of whether a morphism of algebraic patterns $f: \calO \to \calP$ is a Morita equivalence is closely related to the question of whether $f_{\ast} : \Fun(\calO,\calS) \to \Fun(\calP,\calS)$ preserves Segal objects. We begin with an observation about the latter.

\begin{lem}\label{lem: Segal from Beck-Chevalley}
	Let $f: \calO \to \calP$ be a morphism of algebraic patterns, $\calC$ a complete $\infty$-categry and $F \in  \Seg_{\calO}(\calC)$ a Segal object. Suppose the Beck-Chevalley map 
	$(f_{\ast} F)|_{\calP^{\int}} \to f^{\int}_{\ast} (F|_{\calO^{\int}})$
	is an equivalence. Then $f_{\ast} F$ is a Segal object.
\end{lem}
\begin{proof}
    The proof follows an argument from \cite[proposition~6.3.]{RH-algebraic}. We repeat it here for the benefit of the reader. Let $i_{\calO} :\calO^{\el} \mono \calO^{\int}$ and $i_{\calP}: \calP^{\el} \mono \calP^{\int}$ denote the inclusions of the elementary objects. By \cref{lem: Segal iff right Kan extended} it suffices to show that the natural map
	$(f_{\ast}F)|_{\calP^{\int}} \to i_{\calP,\ast}(f_{\ast} F)|_{\calP^{\el}}$
	is an equivalence. 
    We have
	\begin{align*}
	(f_{\ast}F)|_{\calP^{\int}} \simeq f^{\int}_{\ast}( F|_{\calO^{\int}})  \simeq f^{\int}_{\ast} i_{\calO,\ast} F|_{\calO^{\el}} \simeq (f \circ i_{\calO})_{\ast} F|_{\calO^{\el}} \simeq (i_{\calP} \circ f^{\el})_{\ast} F|_{\calO^{\el}} \simeq i_{\calP ,\ast} (f^{\el}_{\ast} (F|_{\calO^{\el}})),
	\end{align*}
	where the second equivalence follows from \cref{lem: Segal iff right Kan extended} since $F$ is Segal by assumption. 
    On the other hand we have:
	\begin{align*}
	f^{\el}_{\ast} (F|_{\calO^{\el}}) & \simeq i^{\ast}_{\calP} i_{\calP,\ast} f^{\el}_{\ast} (F|_{\calO^{\el}}) && \text{($i_{\calP}$ is fully faithful)} \\
	& \simeq i^{\ast}_{\calP} f^{\int}_{\ast} i_{\calO,\ast} F|_{\calO^{\el}} \\
	& \simeq i^{\ast}_{\calP} f^{\int}_{\ast} F|_{\calO^{\int}} && (\text{$F$ is Segal}) \\
	& \simeq i^{\ast}_{\calP} F|_{\calP^{\int}} && (\text{Beck Chevalley}) \\
	& \simeq F|_{\calP^{\el}}
	\end{align*}
	Combining the above equivalences we conclude $(f_{\ast}F)|_{\calP^{\int}} \simeq i_{\calP ,\ast} (f^{\el}_{\ast} (F|_{\calO^{\el}})) \simeq  i_{\calP,\ast} (F|_{\calP^{\el}})$
    as desired. 
\end{proof}

\begin{lem}\label{lem: weakly n-initial criterion for Beck Chevalley}
	Let $f: \calO \to \calP$ be a morphism of algebraic patterns. Let $y \in \calP$ and suppose the natural morphism of algebraic patterns
	\[\calO^{\int} \times_{\calP^{\int}} \calP^{\int}_{y/} \longrightarrow \calO \times_{\calP} \calP_{y/},\]
	is weakly $n$-initial. Then for every $F \in \Seg_{\calO}(\calS^{\le n})$ the Beck-Chevalley map
	\[(f_{\ast} F)|_{\calP^{\int}} \to f^{\int}_{\ast} (F|_{\calO^{\int}}),\]
	is an equivalence at $y \in \calP$.
\end{lem}
\begin{proof}
	Let $F \in \Seg_{\calO}(\calS^{\le n})$ and $y \in \calP$. Evaluating the Beck-Chevalley map of $F$ at $y$ and using the formula for right Kan extensions we get:
	\begin{align*}
	\underset{\calO \times_{\calP} \calP_{y/}}{\lim}F \simeq (f_{\ast} F)(y) \to (f^{\int}_{\ast} (F|_{\calO^{\int}}))(y)  \simeq \underset{\calO^{\int} \times_{\calP^{\int}} \calP^{\int}_{y/}}{\lim}F
	\end{align*}
	It thus suffices to show that the comparison map
	\[  \underset{\calO \times_{\calP} \calP_{y/}}{\lim}F  \to  \underset{\calO^{\int} \times_{\calP^{\int}} \calP^{\int}_{y/}}{\lim}F,\] 
	is an equivalence. 
    Since $F \in \Seg_{\calO}(\calS^{\le n})$ and $\calO \times_{\calP} \calP_{y/} \to \calO$ is a (strict) Segal morphism (see \cref{lem: projection from slice is iso-Segal}) it follows that $F|_{\calO \times_{\calP} \calP_{y/}} \in \Seg_{\calO \times_{\calP} \calP_{y/}}(\calS^{\le n})$. But $\calO^{\int} \times_{\calP^{\int}} \calP^{\int}_{x/} \to \calO \times_{\calP} \calP_{y/}$ is weakly $n$-initial by assumption and thus the above map is an equivalence.
\end{proof}

\begin{defn}
	We say that a morphism of algebraic patterns $f: \calO \to \calP$ is \hl{$\mrm{BC}_n$}\footnote{$\mrm{BC}$ stands for Beck-Chevalley} if for every $y \in \calP$ the induced map $\calO^{\int} \times_{\calP^{\int}} \calP^{\int}_{y/} \to \calO \times_{\calP} \calP_{y/}$ is weakly $n$-initial. When $n=\infty$ we drop the subscript and simply say that $f$ is $\hl{\mrm{BC}}$.
\end{defn}

\begin{obs}
	Let $f: \calO \to \calP$ be a $\mrm{BC}_n$-morphism and let $\calC$ be an $(n+1,1)$-category. Then $\calC$ is $f$-complete if and only if it is $f^{\int}$-complete. 
\end{obs}

The following lemma is a direct consequence of \cite[Lemma~6.2]{RH-algebraic}.

\begin{lem}[{\cite[Lemma~6.2]{RH-algebraic}}]\label{lem:BC-criterion-RH}
	Let $f: \calO \to \calP$ be a morphism of algebraic patterns such that for every $x \in \calO$ the natural map $(\calO_{/x})^{\simeq} \to (\calP_{/f(x)})^{\simeq}$ is an equivalence. Then $f$ is $\mrm{BC}$.
\end{lem}

\begin{cor}\label{lem:when-strong-Segal-is-BCn}
	Let $f: \calO \to \calP$ be a strong Segal morphism of algebraic patterns. Then $f$ is $\mrm{BC}_n$ if and only if for every $y \in \calP$ at least one of the following conditions hold:
	\begin{enumerate}
		\item\label{thm: Morita equivalence from weakly initial slices - condition 3.a.} The fully faithful inclusion $\calO^{\int} \times_{\calP^{\int}} \calP^{\int}_{y/} \mono \calO \times_{\calP} \calP_{y/}$
		is weakly $n$-initial.
		\item\label{thm: Morita equivalence from weakly initial slices - condition 3.b.}  There exists $x \in \calO$ with $ f(x) \simeq y$ and such that $(x,\id_{f(x)}) \in \calO_{f(x)/}$ is weakly $n$-initial.
	\end{enumerate}
\end{cor}
\begin{proof}
	\cref{prop: weakly initial slice from weakly initial object} implies that for any $y \in \calP$, condition $(2)$ implies condition $(1)$ but $\mrm{BC}_n$ means by definition satisfying $(1)$ for all $y \in \calP$.
\end{proof}

\begin{cor}\label{lem:right-kan-ext-BCn-morphism}
	Let $f: \calO \to \calP$ be a $\mrm{BC}_n$-morphism of algebraic patterns. Then for every $f^{\int}$-complete $(n+1,1)$-category $\calC$ and every Segal $\calO$-object $F: \calO \to \calC$ the right kan extension $f_{\ast} F$ exists and is moreover a Segal $\calP$-object.
\end{cor}
\begin{proof}
	Combine \cref{lem: weakly n-initial criterion for Beck Chevalley} with \cref{lem: Segal from Beck-Chevalley}.
\end{proof}

\begin{thm}\label{thm:Morita-equivalence-from-weakly-initial-slices}
	Let $f: \calO \to \calP$ be a strong Segal morphism of algebraic patterns. Suppose:
	\begin{enumerate}
		\item $f^{\el} : \calO^{\el} \to \calP^{\el}$ is an equivalence of $\infty$-categories.
		\item $f: \calO \to \calP$ is a $\mrm{BC}_n$-morphism. 
		Equivalently (see \cref{lem:when-strong-Segal-is-BCn}), for every $y \in \calP$ at least one of the following conditions hold:
		\begin{enumerate}
			\item The fully faithful inclusion $\calO^{\int} \times_{\calP^{\int}} \calP^{\int}_{y/} \mono \calO \times_{\calP} \calP_{y/}$
			is weakly $n$-initial.
			\item  There exists $x \in \calO$ with $ f(x) \simeq y$ such that $(x,\id_{f(x)}) \in \calO_{f(x)/}$ is weakly $n$-initial.
		\end{enumerate}
	\end{enumerate}
	Then $f$ is a Morita $n$-equivalence.
\end{thm}
\begin{proof}
	Restriction $f^{\ast}: \Fun(\calP,\calS^{\le n}) \to \Fun(\calO,\calS^{\le n})$ preserves Segal objects by assumption that $f$ is strong Segal. Since $f$ is $\mrm{BC}_n$ right kan extension $f_{\ast}$ preserves Segal objects by \cref{lem:right-kan-ext-BCn-morphism}. We conclude that the adjunction $f^{\ast} \dashv f_{\ast}$ restricts to an adjunction on Segal objects:
	\[ f^{\ast}: \Seg_{\calP}(\calS^{\le n}) \adj \Seg_{\calO}(\calS^{\le n}): f_{\ast} \]
	By \cref{lem: essentially surjective on elementary objects implies restriction is conservative}, $f^{\ast}$ is conservative so it suffices to show that $f_{\ast}$ is fully faithful. Equivalently we must show that for every $F \in \Seg_{\calO}(\calS^{\le n})$ the counit map $f^{\ast} f_{\ast} F \to F$ is an equivalence. By \cref{lem: restriction to elementary is conservative} it suffices to check this after restricting to elementary objects.	Denote by $i_{\calO}: \calO^{\el} \mono \calO^{\int}$ and $i_{\calP}:\calP^{\el} \mono \calP^{\int}$ the natural inclusions. We have,
	\begin{align*}
	(f^{\int} \circ i_{\calO})^{\ast} (f^{\int} \circ i_{\calO})_{\ast} (F|_{\calO^{\el}}) & \simeq (f^{\int} \circ i_{\calO})^{\ast} f^{\int}_{\ast} (i_{\calO,\ast} (F|_{\calO^{\el}}))  \\
	& \simeq (f^{\int} \circ i_{\calO})^{\ast} f^{\int}_{\ast} (F|_{\calO^{\int}}) &&& \text{($F$ is a Segal object)} \\
	& \simeq (f^{\int} \circ i_{\calO})^{\ast} ((f_{\ast} F)|_{\calP^{\int}}) &&& \text{(Beck-Chevalley)} \\
	& \simeq (f^{\ast}f_{\ast} F)|_{\calO^{\el}} \to F|_{\calO^{\el}},
	\end{align*}
	Consequently it suffices to show that the counit map
    $(f^{\int} \circ i_{\calO})^{\ast} (f^{\int} \circ i_{\calO})_{\ast} (F|_{\calO^{\el}}) \to F|_{\calO^{\el}}$
    is an equivalence. 
	This would follow once we show $f^{\int} \circ i_{\calO}$ is fully faithful. 
    To see this observe that $f^{\int} \circ i_{\calO} \simeq i_{\calP} \circ f^{\el}$, but $f^{\el}$ is an equivalence due to (1), and $i_{\calP}$ is fully faithful so we're done.
\end{proof}

The conditions of \cref{thm:Morita-equivalence-from-weakly-initial-slices} might be tricky to check in general. 
For convenience we record here a weaker, but easier to use, variant of \cref{thm:Morita-equivalence-from-weakly-initial-slices} which can be easily deduced using \cref{lem:BC-criterion-RH}.

\begin{cor}\label{thm: Morita equivalence theorem for faithful subcategories}
	Let $f: \calO \to \calP$ be a strong Segal morphism of algebraic patterns satisfying
	\begin{enumerate}[(1)]
		\item\label{item:eleq} $f^{\el} \colon \calO^{\el} \to \calP^{\el}$ is an
		equivalence of $\infty$-categories.
		\item\label{item:acteq1} for every $x \in \calO$, the functor
		$(\calO_{/x}^\act)^\simeq \to (\calP_{/f(x)}^\act)^\simeq$ 
		is an equivalence of $\infty$-groupoids.
	\end{enumerate}
	Then $f$ is a Morita equivalence.
\end{cor}

%% file: sections/platonic-patterns.tex
\section{Arity approximation of analytic patterns}\label{sect:analytic-Patterns}

In this section we prove a general form of the \textit{arity approximation theorem} (\cref{thm:main-theorem}) in the setting of \textit{analytic patterns}. 
Analytic patterns form a wide class of algebraic patterns with operad-like features. 
They include all $\infty$-operads and are closed under passage to subpatterns, 
in particular under \textit{arity restriction} (see \cref{defn:general-arity-restriction}). 
As such they provide a convenient setting in which to analyze arity related phenomenon. 

\begin{defn}
	An \textit{analytic pattern} is an algebraic pattern $\calO$ equipped with a map $\bracket{-}: \calO \to \Fin_{\ast}$ satisfying the following properties:
	\begin{enumerate}
		\item $\bracket{-}$ is a strict Segal morphism.
		\item $\bracket{-}^{\int} : \calO^{\int} \to \Fin^{\int}_{\ast}$ is conservative.
	\end{enumerate}
\end{defn}

Throughout this section we fix an analytic pattern $\calO$ and denote its structure map by $\bracket{-}:\calO \to \Fin_{\ast}$. 

\begin{example}\label{ex:operads-are-analytic-patterns}
	Let $\calO$ be an $\infty$-operad. We claim that $\calO$ admits a canonical structure of an analytic pattern. Indeed we saw in \cref{ex:operads-are-algebraic-patterns} that $\calO$ admits a canonical structure of an algebraic pattern and in \cref{ex:operads-are-strict} that $\bracket{-}: \calO \to \Fin_{\ast}$ is a strict Segal morphism. It remains to check that $\bracket{-}^{\int}: \calO^{\int} \to \Fin^{\int}_{\ast}$ is conservative. This follows from the fact that all inert morphisms in $\calO$ are $\bracket{-}$-cocartesian.
\end{example}

\begin{defn}\label{defn:general-arity-restriction}
	For an integer $k \ge 1$ we define the \textit{arity $k$-restriction} of $\calO$ as 
	$\hl{\calO^{\le k} \coloneq  } \, \Fin^{\le k}_{\ast}  \times_{\Fin_{\ast}} \calO$.
	By virtue of \cref{lem:basechange-subpattern}, $\calO^{\le k} \subseteq \calO$ is a subpattern.	
\end{defn}

\begin{example}
	We think of $\calO^{\le k}$ as obtained from $\calO$ by throwing all the information in arities $> k$. Indeed if $\calO$ is an $\infty$-operad and $x_1,\cdots,x_m,y \in \calO^{\el}\coloneq \calO_{\bracket{1}}$ is a collection of objects then as long as $m \le k$ we have
	\[\Mul_{\calO}(x_1,\dots,x_m;y) \coloneq  \Map_{\calO^{\act}}(x_1 \oplus \dots \oplus x_m,y) \simeq \Map_{\calO^{\act,\le k}}(x_1 \oplus \dots \oplus x_m,y)\]
	In other words $\calO^{\le k}$ knows about all multi-mapping spaces of arity $\le k$ in $\calO$.
\end{example}

\begin{defn}\label{defn:maximmaly-active-toy}
	A map $f: A_+ \to B_+$ of pointed finite sets is called \textit{maximally active} if it's active and $|f(A)|=1$. A morphism $\alpha : x \to y$ in $\calO$ is called \textit{maximally active} if the underlying map of pointed sets $\bracket{\alpha}: \bracket{x} \to \bracket{y}$ is maximally active.
\end{defn}

\begin{rem}
	Let $M: \Fin_{\ast} \to \Set$ be a commutative monoid. The maximally active morphisms give rise to operations $M^{\times k} \to M^{\times l}$ which may be factored as follows
	\[\begin{tikzcd}
	{(m_1,\dots,m_k)} & {m_1 \cdots m_k} & {(e,\dots,e,m_1\cdots m_k,e\dots,e)}
	\arrow[maps to, from=1-2, to=1-3]
	\arrow[maps to, from=1-1, to=1-2]
	\end{tikzcd}\]
	Intuitively, these are the operations where all inputs are multiplied in one fell swoop.
\end{rem}

\begin{defn}
	Let $\mu: x \actarrow y$ be an active morphism in $\calO$. Define the \textit{factorization category} of $\mu$ as follows
	\[\hl{\Fact_{\calO}(\mu) \coloneq }\, \calO^{\act}_{x/ /\mu}\]
\end{defn}

\begin{defn}\label{defn:precise-definition-of-quasi-partition-category}
	Let $\mu : x \rightsquigarrow z$ be an active morphism. A \textit{quasi-partition of $\mu$} is a factorization of $\mu$
	\[\begin{tikzcd}
	& y \\
	x && z
	\arrow["\alpha"{description}, squiggly, from=2-1, to=1-2]
	\arrow[squiggly, from=1-2, to=2-3]
	\arrow["\mu"{description}, squiggly, from=2-1, to=2-3]
	\end{tikzcd}\]
	such that 
	\begin{enumerate}
		\item $\alpha: x \rightsquigarrow y$ is \textbf{not} maximally active.
		\item $|y|<|x|$. 
	\end{enumerate}
	Define the \textit{quasi-partition category} of $\mu$ to be the full subcategory $\hl{\QPart_{\calO}(\mu) }\subseteq \Fact_{\calO}(\mu)$ whose objects are the quasi-partitions of $\mu$. 
\end{defn}
	
\begin{defn}
	Let $\mu :x \actarrow y$ be an active morphism in $\calO$. Define the \textit{quasi-partition complex} of $\mu$ as follows
	\[ \hl{\qpart{\calO}{\mu} \coloneq } \, \big| \QPart_{\calO}(\mu) \big|\]
\end{defn}

Intuitively, we may think of $\qpart{\calO}{\mu}$ as the space of all ways to express $\mu$ as a non-trivial composite. Indeed $\qpart{\calO}{\mu}$ is non-empty if and only if $\mu$ has a non-trivial such decomposition, and roughly speaking, the more connected $\qpart{\calO}{\mu}$ is, the more unique that decomposition of $\mu$ is.

\begin{defn}
	For a pair of integers $1 \le k_0 \le k_1 \le \infty$ define
	\[\hl{\qsig{\calO}(k_0,k_1)\coloneq } \inf \bigg\{ \conn \bigg(\qpart{\calO}{\mu} \bigg) \bigg| \, \mu : x \rightsquigarrow y \text{ maximally active morphism with } k_0 < |x| \le k_1 \bigg\} \]
	where $\conn(X)$ denotes the connectivity of the space $X$.
\end{defn}

\begin{war}
	Note that the definition of $\qsig{\calO}$ is not the same as that of $\sig{\calO}$ given in  \cref{cons:intro-definition-partition}. They are in fact different functions in general. However, we shall see in \cref{sect:arity-approximation-operads}, that when $\calO$ is an $\infty$-operad they do happen to agree. This should not be obvious at this point.
\end{war}

We are ready to state the main theorem of this section.

\begin{thm}[Arity Approximation for Analytic Patterns]\label{thm:artiy-restriction-theorem}
	Let $\calO$ be an analytic pattern. For every $1 \le k \le k+l \le \infty$ the fully faithful inclusion
	$\calO^{\le k} \mono \calO^{\le k+l}$
	is a Morita $\qsig{\calO}(k,k+l)$-equivalence.
\end{thm}

The theorem will be proved in \cref{subsect:proof-of-arity-restriction} after we develop the necessary preliminaries. For the benefit of the reader we give below a schematic structure of the proof along  with references to key technical steps.

\subsubsection{Structure of the proof}
\begin{enumerate}
	\item By induction the theorem reduces to the case $\calO^{\le k-1} \mono \calO^{\le k}$ (purely for notational convenience we have chosen to index over $2 \le k \le \infty$ in the proof).
	
	\item In \cref{defn:k-1/2-definition} we introduce an intermediate subpattern $\calO^{\le k-1} \subseteq \hl{\calO^{\le k- \frac 12}} \subseteq \calO^{\le k}$.
	
	\item In \cref{prop: inclusion of k-1 in k-1/2 is Morita equivalence} we show that the inclusion $\calO^{\le k-1} \subseteq \calO^{\le k- \frac 12}$ satisfies \cref{thm: Morita equivalence theorem for faithful subcategories} and is thus a Morita equivalence.
	
	\item The inclusion $\calO^{\le k-\frac 12} \subseteq \calO^{\le k}$ satisfies condition \ref{thm: Morita equivalence from weakly initial slices - condition 3.b.} of \cref{thm:Morita-equivalence-from-weakly-initial-slices} for every $x \in \calO^{\le k-\frac 12}$ with $|x|<k$. This leaves us with the case of $|x|=k$.
	
	\item In \cref{defn: definition of Js} we introduce for every $x \in \calO^{\le k-\frac 12}$ with $|x|=k$, an intermediate subpattern 
	\[\calO^{\le k-\frac 12}_{x/} \subseteq \hl{\calJ^{\nm}_{\calO,x}} \subseteq \calO^{\le k}_{x/} \times_{\calO^{\le k}} \calO^{\le k-\frac 12}\]
	
	\item In \cref{prop:J^<=k-1/2-morita-equivalent-to-J^<k} we show the inclusion $\calO^{\le k-\frac 12}_{x/} \mono \calJ^{\nm}_{\calO,x}$ satisfies \cref{thm: Morita equivalence theorem for faithful subcategories} and is thus a Morita equivalence.
	
	\item In \cref{subsect:slices-to-partitions} we study the slices of the fully faithful inclusion $\calJ^{\nm}_{\calO,x} \mono \calO^{\le k}_{x/} \times_{\calO^{\le k}} \calO^{\le k-\frac 12}$ and relate them to quasi-partition categories. We show that the inclusion is $n$-initial if and only if some collection of quasi-partition complexes are $n$-connected.

	\item By (6) and (7) the inclusion $\calO^{\le k-\frac12} \mono \calO^{\le k}$ satisfies condition \ref{thm: Morita equivalence from weakly initial slices - condition 3.a.} of \cref{thm:Morita-equivalence-from-weakly-initial-slices} with $n = \qsig{\calO}(k-1,k)$ for all $x \in \calO^{\le k-\frac 12}$ such that $|x|=k$.
	
	\item By (4) and (8) it follows that $\calO^{\le k-\frac 12} \mono \calO^{\le k}$ satisfies \cref{thm:Morita-equivalence-from-weakly-initial-slices} with $n= \qsig{\calO}(k-1,k)$ and is therefore a Morita $\qsig{\calO}(k-1,k)$-equivalence.
	
	\item Finally by (3) and (9) it follows that the composition $\calO^{\le k-1} \mono \calO^{\le k}$ is a Morita $\qsig{\calO}(k-1,k)$-equivalence.
\end{enumerate}

\subsection{Analytic patterns}

We begin this subsection by establishing some basic properties of analytic patterns. We define \textit{strongly active morphisms} (\cref{defn:strongly-active}) and the notion of \textit{$(k- \frac{1}{2})$-restriction} of an analytic pattern as an intermediate subpattern $\calO^{\le k-1} \subseteq \calO^{\le k-\frac 12} \subseteq \calO^{\le k}$ (\cref{defn:k-1/2-definition}). The main result of this subsection is \cref{prop: inclusion of k-1 in k-1/2 is Morita equivalence} which shows that the inclusion $\calO^{\le k-1} \mono \calO^{\le k-\frac 12}$ is a Morita equivalence.

\begin{lem}\label{lem:inert=iso-criterion}
	Let $\alpha: x \inert y$ be an inert morphism in $\calO$. The following are equivalent:
	\begin{enumerate}
		\item $\alpha$ is an isomorphism.
		\item $\alpha$ admits a retraction.
		\item $|x| \le |y|$.
	\end{enumerate}
\end{lem}
\begin{proof}
	We prove $(3) \implies (1)$. Both $(1) \implies (2)$ and $(2) \implies (3)$ are straightforward. Since $\calO$ is analytic, $\calO^{\int} \to \Fin^{\int}_{\ast}$ is conservative and thus it suffices to show that $\bracket{\alpha}$ is an isomorphism. Since $\bracket{\alpha} : \bracket{x} \to \bracket{y}$ is inert it's surjective. But $|x| \le |y|$ by assumption, so $\bracket{\alpha}$ is an isomorphism.
\end{proof}

\begin{cor}\label{lem:detect-active}
	Let $\alpha: x \to y$ be a morphism in $\calO$. Then $\alpha$ is active if and only if $\bracket{\alpha}$ is active. 
\end{cor}
\begin{proof}
	Apply \cref{lem:inert=iso-criterion} to $\alpha^{\int}$.
\end{proof}

\begin{lem}\label{lem: active cancelation}
	Let $\alpha : x \to y$ and $\beta : y \to z$ be morphisms in $\calO$. If $\beta \circ \alpha$ is active then so is $\alpha$.
\end{lem}
\begin{proof}
	By \cref{lem:detect-active} we may assume that $\calO = \Fin_{\ast}$. Now let $f: A_+ \to B_+$ and $g : B_+ \to C_+$ be maps of pointed finite sets such that $g \circ f$ is active. Then we have $f^{-1}(\ast) \subseteq g^{-1} (f^{-1}(\ast)) = (g \circ f)^{-1}(\ast) = \{\ast\}$
	so that $f$ is active as required.
\end{proof}

One useful consequence of \cref{lem:detect-active} is the following characterization of elementary objects.

\begin{cor}\label{cor:elementar-means-cardinality-1}
	An object $x \in \calO$ is elementary if and only if $|x| =1$.
\end{cor}
\begin{proof}
	Apply \cref{lem:inert=iso-criterion} to the unique inert morphism in $\calO^{\el}_{x/}\simeq \bracket{x}^{\circ}$.
\end{proof}

\begin{rem}
	As a consequence of \cref{lem:detect-active} and \cref{cor:elementar-means-cardinality-1}, we see that the algebraic pattern structure on an analytic pattern $(\calO,\bracket{-})$ is uniquely determined from the map of underlying $\infty$-categories $\bracket{-}:\calO \to \Fin_{\ast}$.
\end{rem}

\begin{cor}\label{rem: elementary objects of arity restriction}
	For any integer $k \ge 1$ the subpattern inclusion $\calO^{\le k} \subseteq \calO$ induces an equivalence on elementary objects $\calO^{\le k,\el} \simeq \calO^{\el}$.
\end{cor}

\begin{defn}\label{defn:strongly-active}
	For an active morphism $\alpha : x \rightsquigarrow y$ in $\calO$ we say that $\alpha$ is \textit{$k$-strongly active} if it is not an equivalence and $|x| =k$.
\end{defn}

\begin{war}
	The property of being strongly active \textit{can not} be detected at the level of underlying pointed finite sets. This should be contrasted with the definition of maximally active morphisms (see \ref{defn:maximmaly-active-toy}).
\end{war}

\begin{lem}\label{lem:strongly-active-depends-on-active-part}
	For a morphism $\alpha: x \to y$ the following are equivalent:
	\begin{enumerate}
		\item $\alpha$ is $|x|$-strongly active.
		\item $\alpha^{\act}$ is $|x|$-strongly active.
	\end{enumerate}
\end{lem}
\begin{proof}   
    $(1 \Rightarrow 2)$ Suppose $\alpha: x \to y$ is $|x|$-strongly active. Since $\alpha$ is active we have $\alpha = \alpha^{\act}$. In particular $\alpha^{\act}$ is $|x|$-strongly active.
    $(2 \Rightarrow 1)$ Suppose $\alpha^{\act} : \Lambda(x) \to y$ is $|x|$-strongly active so that in particular $|\Lambda(\alpha)| = |x|$. It follows that $\alpha^{\int}$ is an isomorphism so that $\alpha = \alpha^{\act}$. But $\alpha^{\act}$ is $|x|$-strongly active by assumption.
\end{proof}
\begin{lem}\label{lem: strongly active cancelation}
	Let $\alpha : x \rightsquigarrow y$ and $\beta : y \rightsquigarrow z$ be active morphisms in $\calO$ whose composition $\beta \circ \alpha$ is $k$-strongly active for some $k$. Then at least one of $\alpha$ or $\beta$ is $k$-strongly active.
\end{lem}
\begin{proof}
	If $\alpha$ is $k$-strongly active we're done. Assume otherwise. Since $\beta \circ \alpha$ is $k$-strongly active we must have $|x|=k$. On the other hand by \cref{lem: active cancelation}, $\alpha$ is active. Since we assumed $\alpha$ is not $k$-strongly active it must be an equivalence. We conclude that $\beta \simeq \beta \circ \alpha$. In particular $\beta$ is $k$-strongly active as required.
\end{proof}

\begin{defn}\label{defn:k-1/2-definition}
	Given an integer $k \ge 1$ we define the \textit{$(k-\frac 12)$-restriction} of $\calO$ to be the wide subcategory $\hl{\calO^{\le k- \frac 12}} \subseteq \calO^{\le k}$ whose morphisms are \textit{not} $k$-strongly active (these are closed under composition by \cref{lem: strongly active cancelation}).
\end{defn}

\begin{lem}\label{lem: O<=k-1/2 subpattern of O^<=k}
	The following statements hold:
	\begin{enumerate}
		\item The wide and replete subcategory $\calO^{\le k- \frac 12} \subseteq \calO^{\le k}$ is an algebraic subpattern and the inclusion restricts to an equivalence on elementary objects $\calO^{\le k- \frac 12,\el} \simeq \calO^{\le k,\el}$.
		\item Tbe full subcategory $\calO^{\le k- 1} \subseteq \calO^{\le k - \frac 12}$ is an algebraic subpattern and the inclusion restricts to an equivalence on elementary objects $\calO^{\le k- 1,\el} \simeq \calO^{\le k -\frac 12 ,\el}$.
	\end{enumerate}
\end{lem}
\begin{proof}
	We omit the proof of $(2)$ as it is an an immediate consequence $(1)$. To prove $(1)$ we verify the conditions of \cref{prop: faithful subpattern}. Condition (2) is an immediate since inert morphism are never strongly active. Let $\alpha : x \to y$ be a morphism in $\calO^{\le k - \frac 12}$. We must show that $\alpha^{\int}$ and $\alpha^{\act}$ are not strongly active. For the former this holds by definition. For the latter this follows from \cref{lem:strongly-active-depends-on-active-part}. The claim about elementary objects is a consequence of \cref{cor:elementar-means-cardinality-1}.
\end{proof}

\begin{prop}\label{prop: inclusion of k-1 in k-1/2 is Morita equivalence}
	For every $k \ge 2$ the fully faithful inclusion $\calO^{\le k -1} \mono \calO^{\le k- \frac 12}$ is a Morita equivalence.
\end{prop}
\begin{proof}
	We verify the conditions of \cref{thm: Morita equivalence theorem for faithful subcategories}. Condition (1) is immediate from \cref{rem: elementary objects of arity restriction}. To verify condition (2) we must show that if $\alpha : y \to x$ is a morphism in $\calO$ which is not $k$-strongly active with $|y|\le k$ and $|x| < k$ then $\Lambda(\alpha) < k$. 
	We prove the contrapositive. Suppose  $\Lambda(\alpha) = k$. Since $|y|\le k$ it follows that $\alpha^{\int}$ is an isomorphism so that $\alpha$ is active. But by assumption $|x| < k = |\Lambda(\alpha)| \le |y|$ so $\alpha$ can not be an equivalence and is therefore $k$-strongly active.
\end{proof}

\subsection{From slices to partition complexes}\label{subsect:slices-to-partitions}

In this subsection we study the slices that appear naturally when \cref{thm:Morita-equivalence-from-weakly-initial-slices} is applied to the inclusion $\calO^{\le k-\frac 12} \mono \calO^{\le k}$. We introduce the \textit{quasi-partition category} (\cref{defn:precise-definition-of-quasi-partition-category}) of an active morphism and relate it to these slice categories (\cref{lem:relation-between-J's-and-the-partition-category}). These results  will be used in \cref{subsect:proof-of-arity-restriction} where we prove the arity approximation theorem for analytic patterns (\cref{thm:artiy-restriction-theorem}).

\begin{defn}\label{defn: definition of Js}
	Let $x \in \calO$ with $|x| \ge 1$. We introduce some notation.
	\begin{itemize}
		\item Denote $\hl{\calJ_{\calO,x} \coloneq }\, \calO^{\le |x|-\frac 12} \times_{\calO^{\le |x|}} \calO^{\le |x|}_{x/}$.
		\item Denote by $\hl{\calJ^{\ns}_{\calO,x}} \subseteq \calJ_{\calO,x}$ the full subcategory on objects of the form $(y,\alpha :x \to y)$ such that $\alpha$ is \textbf{not} $|x|$-strongly active.
		\item Denote by $\hl{\calJ^{\nm}_{\calO,x}} \subseteq \calJ_{\calO,x}$ the full subcategory on objects of the form $(y,\alpha :x \to y)$ such that $\alpha$ is \textbf{not} maximally active.
	\end{itemize}
\end{defn}

\begin{rem}\label{rem:identification-of-J-with-slice}
	Observe that the inclusion $\calO^{\le k -\frac 12} \to \calO^{\le k}$ induces a fully faithful functor
	\[ \calO^{\le k - \frac 12}_{x/} \mono \calO^{\le k}_{x/} \times_{\calO^{\le k}} \calO^{\le k-\frac 12} = \calJ_{\calO,x},\]
	which identifies $\calO^{\le |x|-\frac 12}_{x/}$ with the full subcategory $\calJ^{\ns}_{\calO,x} \subseteq \calJ_{\calO,x}$.
\end{rem}

\begin{rem}
	We shall soon show (see \cref{lem:J<k-is-algebraic-subpattern-of-J<=k} and \cref{lem: inclusion of J<=k-1/2 in J<k is Segal and equivalence on elementary}) that 
    $\calJ^{\ns}_{\calO,x} \subseteq \calJ^{\nm}_{\calO,x} \subseteq \calJ_{\calO,x}$
    is a nested inclusion of subpatterns.
\end{rem}

\begin{lem}\label{lem:inert-cancellation-for-maximal}
	Let $\mu: x \to y$ and $\lambda: y \inert z$ be morphisms in $\calO$ with $\lambda$ inert. Suppose that $\lambda \circ \mu$ is active. Then $\lambda \circ \mu$ is maximally active if and only if $\mu$ is maximally active.
\end{lem}
\begin{proof}
	Without loss of generality we may assume that $\calO = \Fin_{\ast}$. We write $f = A_+ \to B_+$ in place of $\mu: x \actarrow y$  and $g: B_+ \inert C_+$ in place of $\lambda: y \inert z$. Since $g$ is inert we have a bijection 
    $g: f(A) \cap g^{-1}(C) \simeq (g \circ f)(A) \cap C$
	Since $g \circ f$ is active we have $(g \circ f)(A) \subseteq C$ and therefore $f(A) \subseteq g^{-1}(C)$. We thus have a bijection
    $g: f(A) \simeq (g \circ f)(A)$
    and in particular $|(g \circ f) (A) |=|f(A)|$.
\end{proof}

\begin{lem}\label{lem:J<k-is-algebraic-subpattern-of-J<=k}
	For every $x \in \calO$ the full subcategory $\calJ^{\nm}_{\calO,x} \subseteq \calJ_{\calO,x}$ is an algebraic subpattern.
\end{lem}
\begin{proof}
	We verify the conditions of \cref{lem:fully-faithful-subpattern}. Unpacking definitions and rephrasing in terms of diagrams in $\calO$ reduces to showing that for a commutative triangle in $\calO^{\le k}$
	\[\begin{tikzcd}
		& x \\
		y && e,
		\arrow["\beta"{description}, from=1-2, to=2-3]
		\arrow["\phi"{description}, from=2-1, to=2-3]
		\arrow["\alpha"{description}, from=1-2, to=2-1]
		\end{tikzcd}\]
    the following claims hold
    \begin{enumerate}
        \item if $\phi$ is inert, $\alpha$ is not maximally active and $e$ is elementary, then $\beta$ is not maximally active. 
        \item 
        if $\alpha$ and $\beta$ are not maximally active, then $\phi^{\int} \circ \alpha$ is not maximally active.
    \end{enumerate}
    Claim $(1)$ holds since if $\phi$ is inert and $\alpha$ is not maximally active then by \cref{lem:inert-cancellation-for-maximal} neither is $ \beta= \phi \circ \alpha$.
    For claim $(2)$ note that if neither $\alpha$ nor $\beta$ are maximally active then by \cref{lem:inert-cancellation-for-maximal} $\phi^{\int} \circ \alpha$ is not active.
\end{proof}

\begin{lem}\label{lem: inclusion of J<=k-1/2 in J<k is Segal and equivalence on elementary}
	For every $x \in \calO$ the full subcategory $\calJ^{\ns}_{\calO,x} \subseteq \calJ^{\nm}_{\calO,x}$ is an algebraic subpattern and the inclusion restricts to an equivalence on elementary objects $\calJ^{\ns,\el}_{\calO,x} \simeq \calJ^{\nm,\el}_{\calO,x}$.
\end{lem}
\begin{proof}
	We verify the conditions of \cref{lem:fully-faithful-subpattern}. Unpacking definitions and rephrasing in terms of diagrams in $\calO$ reduces to showing that for ay commutative triangle in $\calO^{\le k}$
    \[\begin{tikzcd}
		& x \\
		y && z
		\arrow["\beta"{description}, from=1-2, to=2-3]
		\arrow["\phi"{description}, from=2-1, to=2-3]
		\arrow["\alpha"{description}, from=1-2, to=2-1]
		\end{tikzcd}\]
    the following claims hold
    \begin{enumerate}
        \item if $\beta$ is not maximally active and $e$ is elementary, then $\beta$ is not strongly active.
        \item 
        if $\alpha$ and $\beta$ are not strongly active, then $\phi^{\int} \circ \alpha$ is not strongly active.
    \end{enumerate}
    Claim $(1)$ hols since if $\beta$ is not maximally active and $e$ is elementary, $\beta$ can not be active. 
    Claim $(2)$ follows from \cref{lem: strongly active cancelation}.
\end{proof}

\begin{prop}\label{prop:J^<=k-1/2-morita-equivalent-to-J^<k}
	For every $x \in \calO$ the fully faithful inclusion $\calJ^{\ns}_{\calO,x} \mono \calJ^{\nm}_{\calO,x}$ is a Morita equivalence of algebraic patterns.
\end{prop}
\begin{proof}
	We verify the conditions of \cref{thm: Morita equivalence theorem for faithful subcategories}. Condition (1) follows from \cref{lem: inclusion of J<=k-1/2 in J<k is Segal and equivalence on elementary}. It remains to check condition (2). Rephrasing in terms of diagrams in $\calO$ we arrive at the the following condition
	\begin{itemize}
		\item Whenever we have a commutative triangle in $\calO^{\le |x|}$
		\[\begin{tikzcd}
		& x \\
		y && z
		\arrow["\beta", from=1-2, to=2-3]
		\arrow["\alpha"', from=1-2, to=2-1]
		\arrow["\phi"', from=2-1, to=2-3]
		\end{tikzcd}\]
		in which $\beta$ and $\phi$ are not $|x|$-strongly active. Then $\phi^{\int} \circ \alpha$ is not $|x|$-strongly active.
	\end{itemize}
	Equivalently we must show that if $\phi^{\int}\circ \alpha$ is strongly active then at least one of $\beta$ or $\phi$ are strongly active. Suppose then that $\phi^{\int} \circ \alpha$ is strongly active. In particular since $\phi^{\int}\circ \alpha$ is active and $\beta \simeq \phi \circ \alpha \simeq \phi^{\act} \circ (\phi^{\int} \circ \alpha)$ we deduce that $\beta$ is active. 
    If $\beta$ strongly active we are done. We may therefore assume without loss of generality that $x = z$ and $\beta = \id_x$ in which case we have $\phi \circ \alpha \simeq \id_x$ and we must show that $\phi$ is strongly active. By \cref{lem:strongly-active-depends-on-active-part} it suffices to show this for $\phi^{\act}$. 
    Since $\phi^{\act} \circ (\phi^{\int} \circ \alpha) \simeq \phi \circ \alpha \simeq \id_x$ it follows that if $\phi^{\act}$ is an isomorphism, $\phi^{\int} \circ \alpha$ must be its inverse. But $\phi^{\int} \circ \alpha : x \actarrow \Lambda(\phi)$ is $|x|$-strongly active by assumption. In particular it is not an isomorphism and so neither is $\phi^{\act} : \Lambda(\phi) \actarrow x$. 
    It remains to show that $|\Lambda(\phi)|=|x|$. 		
    Note that $y \in \calO^{\le k}$ so that $|\Lambda(\phi)| \le |y| \le |x|$. But we just showed that $x$ is a retract of $\Lambda(\phi)$ which implies that $|\Lambda(\phi)| \ge |x|$. We conclude that $|\Lambda(\phi)| = |x|$ as required.\qedhere	
\end{proof}

\begin{prop}\label{lem: final double slice lemma}
	For every morphism $\mu : x \to y$ in $\calO$ the fully faithful inclusion 
	\[ (\calJ^{\act}_{\calO,x})_{/\mu} \times_{\calJ^{\act}_{\calO,x}} \calJ^{\nm,\act}_{\calO,x} \mono (\calJ_{\calO,x})_{/\mu} \times_{\calJ_{\calO,x}} \calJ^{\nm}_{\calO,x} \]
	admits a left adjoint.
\end{prop}
\begin{proof}
	By \cref{lem:J<k-is-algebraic-subpattern-of-J<=k}, 
	$\calJ^{\nm}_{\calO,x} \subseteq \calJ_{\calO,x}$ is an algebraic subpattern and is thus in particular closed under inert-active factorizations. We apply the dual version of \cref{prop:factorization-adjoint-on-comma-category} to the inclusion $\calJ^{\nm}_{\calO,x} \mono \calJ_{\calO,x}$ at $(x,\mu:x \to y) \in \calJ_{\calO,x}$.
    After unpacking definitions we are left to show that whenever we have a commutative triangle in $\calO^{\le |x|}$
    \[\begin{tikzcd}
		& x \\
		z && y
		\arrow["\mu", from=1-2, to=2-3]
		\arrow["\alpha"', from=1-2, to=2-1]
		\arrow["\varphi"', from=2-1, to=2-3]
		\end{tikzcd}\]
  in which $\alpha$ is not maximally active and $\varphi$ is not $|x|$-strongly active. Then: $(a)$ $\varphi^{\act}$ is not strongly active, $(b)$ $\varphi^{\int} \circ \alpha$ is not maximally active.
  Claim $(a)$ follows from \cref{lem:strongly-active-depends-on-active-part}. Claim $(b)$ follows from \cref{lem:inert-cancellation-for-maximal}.
\end{proof}

\begin{lem}\label{lem:relation-between-J's-and-the-partition-category}
	Let $\mu : x \rightsquigarrow z$ be a maximally active morphism in $\calO$. Then there's a canonical equivalence
	\[(\calJ^{\act}_{\calO,x})_{/\mu} \times_{\calJ^{\act}_{\calO,x}} \calJ^{\nm,\act}_{\calO,x} \simeq \QPart_{\calO}(\mu) \]
\end{lem}
\begin{proof}
	Unpacking the definitions we see that there's a natural faithful inclusion,
	\[(\calJ^{\act}_{\calO,x})_{/\mu} \times_{\calJ^{\act}_{\calO,x}} \calJ^{\nm,\act}_{\calO,x} \mono \calO^{\act}_{x//\mu} =: \Fact_{\calO}(\mu),\]
	which identifies the category on the left with a faithful subcategory of $\Fact_{\calO}(\mu)$ whose objects and morphisms are given as follows:
	\begin{itemize}
		\item \textit{Objects} are diagrams:
				\[\begin{tikzcd}
				& y \\
				x && z
				\arrow["\alpha"{description}, from=2-1,squiggly, to=1-2]
				\arrow["\beta"{description},squiggly, from=1-2, to=2-3]
				\arrow["\mu"{description}, squiggly, from=2-1, to=2-3]
				\end{tikzcd}\]
			in which $\alpha$ is not maximally active and $\beta$ is not $|x|$-strongly active.
		\item \textit{Morphisms} are diagrams:
		\[\begin{tikzcd}
		& x \\
		y && {y^{\prime}} \\
		& z
		\arrow[squiggly, from=1-2, to=2-1]
		\arrow[squiggly, from=2-1, to=3-2]
		\arrow[squiggly, from=1-2, to=2-3]
		\arrow[squiggly, from=2-3, to=3-2]
		\arrow["\phi"{description}, squiggly, from=2-1, to=2-3]
		\end{tikzcd}\]
		in which $\phi$ is not $|x|$-strongly active. 
	\end{itemize} 
	
	We claim that the condition on objects implies the a priori stronger condition that $|y|<|x|$. Indeed by assumption $\beta : y \rightsquigarrow z$ is not strongly active and therefore $|y| = |x|$ can only happen if $\beta$ is an equivalence. This can never happen though as it would imply that $\mu \simeq \beta \circ \alpha \simeq \alpha$ which is impossible since by assumption $\mu$ is maximally active whereas $\alpha$ is not.
	Finally we observe that $|y|<|x|$ always holds for objects in the essential image, and since the condition on morphisms holds automatically the inclusion is fully faithful. 
\end{proof}

\subsection{Proof of the airty approximation theorem}\label{subsect:proof-of-arity-restriction}

In this subsection we prove the Arity approximation theorem for analytic patterns (\cref{thm:artiy-restriction-theorem}). We begin with the special case $l=1$.

\begin{prop}\label{prop:arity-restriction-1-step}
	Let $\calO$ be an analytic pattern
	For every $k \ge 2$ the fully faithful inclusion 
	$\calO^{\le k-1} \mono \calO^{\le k}$
	is a Morita $\qsig{\calO}(k-1,k)$-equivalence.
\end{prop}
\begin{proof}
	The inclusion $\calO^{\le k-1}  \to \calO^{\le k- \frac 12}$ is a Morita equivalence by \cref{prop: inclusion of k-1 in k-1/2 is Morita equivalence}. It thus suffices to show that the inclusion $\calO^{\le k- \frac 12} \to \calO^{\le k}$ is a Morita $\qsig{\calO}(k-1,k)$-equivalence. To show this we verify the conditions of \cref{thm:Morita-equivalence-from-weakly-initial-slices}. Conditions (1) and (2) follow from \cref{lem: O<=k-1/2 subpattern of O^<=k} and \cref{cor:subpattern-properties}. It remains to verify condition (3) for every $x \in \calO^{\le k}$. Specifically we show that when $|x| <k$, \ref{thm: Morita equivalence from weakly initial slices - condition 3.a.} is satisfied and when $|x|=k$, \ref{thm: Morita equivalence from weakly initial slices - condition 3.b.} is satisfied.
	\begin{itemize}
		\item \textit{Case $|x|<k$:} We claim that the fully faithful inclusion 
		\[ \calO^{\le k-\frac 12,\int} \times_{\calO^{\le k,\int}} \calO^{\le k,\int}_{x/} \mono \calO^{\le k-\frac 12} \times_{\calO^{\le k}} \calO^{\le k}_{x/} \]
		admits a right adjoint. To show this it suffices to verify the criterion of \cref{prop:factorization-adjoint-on-comma-category}. Since $\calO^{\le k -\frac 12} \subseteq \calO^{\le k}$ is an algebraic subpattern it is closed under inert-active factorizations. It thus suffices to show the second condition. Specifically we must show that if $\alpha: x \to y$ is any morphism in $\calO^{\le k}$ then $\alpha^{\act} : \Lambda(\alpha) \to y$ is not strongly active. But by assumption we have $|x| < k$ so that $|\Lambda(\alpha)|\le |x| <k$ and therefore any morphism with source $\Lambda(\alpha)$ is automatically not strongly active. 
		
		\item \textit{Case $|x|=k$:} In order to verify \ref{thm: Morita equivalence from weakly initial slices - condition 3.b.} we must show that $(x, \id_x) \in \calJ_{\calO,x} \coloneq   \calO^{\le k-\frac 12} \times_{\calO^{\le k}} \calO^{\le k}_{x/}$ is weakly $\qsig{\calO}(k-1,k)$-initial. Note that $(x,\id_x)$ is initial when considered in the full subcategory $\calJ^{\ns}_{\calO,x} \subseteq \calJ_{\calO,x}$, indeed by \cref{rem:identification-of-J-with-slice} there's a canonical equivalence $\calJ^{\ns}_{\calO,x} \simeq \calO^{\le k -\frac 12}_{x/}$ where the claim becomes obvious. It's therefore sufficient to show that the inclusion $\calJ^{\ns}_{\calO,x} \mono \calJ_{\calO,x}$ is weakly $\qsig{\calO}(k-1,k)$-initial. Recall that this inclusion factors as a composition of fully faithful inclusions
		\[\calJ^{\ns}_{\calO,x} \mono \calJ^{\nm}_{\calO,x} \mono \calJ_{\calO,x} \]
		The first map is a Morita equivalence by \cref{prop:J^<=k-1/2-morita-equivalent-to-J^<k} and thus in particular weakly initial (see \cref{lem:fully-faithful-restriction-implies-weakly-initial}). By \cref{lem:J<k-is-algebraic-subpattern-of-J<=k} the second map is an inclusion of an algebraic subpattern and is therefore a Segal morphism (see \cref{cor:subpattern-properties}). It thus suffices by \cref{lem: cancellation for weakly initial} to show that the second map is weakly $\qsig{\calO}(k-1,k)$-initial. 
		
		We claim that the fully faithful inclusion $\calJ^{\nm}_{\calO,x} \mono \calJ_{\calO,x}$ is in fact $\qsig{\calO}(k-1,k)$-initial. To prove this it suffices by \cref{prop: Quillen A} to show that for every $\mu \in \calJ_{\calO,x}$ the space $\big| (\calJ_{\calO,x})_{/\mu} \times_{\calJ_{\calO,x}} \calJ^{\nm}_{\calO,x}\big|$ is $\qsig{\calO}(k-1,k)$-connected.

		When $\mu \in \calJ^{\nm}_{\calO,x}$ this is clear. Indeed the inclusion $\calJ^{\nm}_{\calO,x} \mono \calJ_{\calO,x}$ is fully faithful so that we have a canonical equivalence
		\[(\calJ^{\nm}_{\calO,x})_{/\mu} \simeq (\calJ_{\calO,x})_{/\mu} \times_{\calJ_{\calO,x}} \calJ^{\nm}_{\calO,x}\]
		But since $\mu \overset{=}{\to} \mu \in (\calJ^{\nm}_{\calO,x})_{/\mu}$ is terminal the space $\big|(\calJ^{\nm}_{\calO,x})_{/\mu} \big|$ is contractible.
		
		Suppose $\mu \notin \calJ^{\ns}_{\calO,x}$, i.e. that $\mu$ is maximally active. By \cref{lem: final double slice lemma} the fully faithful inclusion,
		\[(\calJ^{\act}_{\calO,x})_{/\mu} \times_{\calJ^{\act}_{\calO,x}} \calJ^{\nm,\act}_{\calO,x} \mono  (\calJ_{\calO,x})_{/\mu} \times_{\calJ_{\calO,x}} \calJ^{\nm}_{\calO,x},\]
		admits a left adjoint and is thus in particular initial. By \cref{lem:relation-between-J's-and-the-partition-category} there's an equivalence
		\[(\calJ^{\act}_{\calO,x})_{/\mu} \times_{\calJ^{\act}_{\calO,x}} \calJ^{\nm,\act}_{\calO,x} \simeq \QPart_{\calO}(\mu) \]
		Passing to realizations we deduce an equivalence:
		\[\qpart{\calO}{\mu} \coloneq  \big|\QPart_{\calO}(\mu) \big| \simeq \big|(\calJ^{\act}_{\calO,x})_{/\mu} \times_{\calJ^{\act}_{\calO,x}} \calJ^{\nm,\act}_{\calO,x}\big|\]
		By definition $\qpart{\calO}{\mu}$ is at least $\qsig{\calO}(k-1,k)$-connected so we're done.
	\end{itemize}
\end{proof}

We are finally in a position to prove \cref{thm:artiy-restriction-theorem}.

\begin{proof}[Proof of \cref{thm:artiy-restriction-theorem}]
	First we prove the statement for $l < \infty$ by induction on $l$. The base case $l=0$ is trivial. For the induction step consider the nested inclusions
    $\calO^{\le k} \mono \calO^{\le k+l} \mono \calO^{\le k+l+1}$.
    The first map is a Morita $\qsig{\calO}(k,k+l)$-equivalence by the induction hypothesis. The second map is a Morita $ \qsig{\calO}(k+l,k+l+1)$-equivalence by \cref{prop:arity-restriction-1-step}. The composite is therefore a Morita $\qsig{\calO}(k,k+l+1)$-equivalence since by definition $\qsig{\calO}(k,k+l+1)=\min\big\{\qsig{\calO}(k,k+l),\qsig{\calO}(k+l,k+l+1) \big\}$.
	We now prove the remaining case $l=\infty$. The inclusion $\calO^{\le k} \mono \calO$ is the filtered colimit over $l \in \mathbb{N}$ of the inclusions $\calO^{\le k} \mono \calO^{\le k+l}$. Consider the nested inclusions
    $\calO^{\le k} \mono \calO^{\le k+l} \mono \calO$.
    Passing to $\qsig{\calO}(k,l)$-truncated Segal objects and taking the limit as $l \to \infty$ gives
	\[\Seg_{\calO}(\calS^{\le \qsig{\calO}(k,\infty)}) \longrightarrow \underset{l \to \infty}{\lim} \Seg_{\calO^{\le k+l}}(\calS^{\le \qsig{\calO}(k,\infty)}) \longrightarrow \Seg_{\calO^{\le k}}(\calS^{\le \qsig{\calO}(k,\infty)}).\] 
	Observe that $\qsig{\calO}(k,\infty)= \underset{l\ge 0 }{\inf}\{\qsig{\calO}(k,k+l)\}$ and thus the statement for $l < \infty$ (which we already proved) implies that the second map is an equivalence. The first map is an equivalence by \cref{cor:Segal-objects-for-filtered-colimit}. We conclude that $\calO^{\le k} \mono \calO$ is a Morita $\qsig{\calO}(k,\infty)$-Morita equivalence.
\end{proof}

%% file: sections/operads.tex
\section{Arity approximation of \texorpdfstring{$\infty$-}{infinity }operads}\label{sect:arity-approximation-operads}

In this section we prove the arity approximation theorem for $\infty$-operads.

\begin{thm}[Arity Approximation for $\infty$-Operads]\label{thm:arity-approximation-operads}
	Let $\calO$ be an $\infty$-operad and let $1 \le k \le k+l \le \infty$. 
	For every complete $(\sig{\calO}(k,k+l)+1,1)$-category $\calC$, restriction induces an equivalence of $\infty$-categories
	\[ \Mon^{\le k +l}_{\calO}(\calC) \simeq \Mon^{\le k}_{\calO}(\calC) \]
\end{thm}

The theorem will follow from the arity approximation theorem for analytic patterns (\cref{thm:artiy-restriction-theorem}) once we show the following statement.

\begin{prop}\label{prop:quasi-partition-function-equal}
	Let $\calO$ be an $\infty$-operad. Then $\qsig{\calO}(k,k+l) = \sig{\calO}(k,k+l)$ for every $1 \le k \le k+l \le \infty$.
\end{prop}

\begin{proof}[Proof of \cref{thm:arity-approximation-operads} given \cref{prop:quasi-partition-function-equal}]
	As we saw in \cref{ex:operads-are-analytic-patterns}, $\infty$-operads are a special case of analytic patterns. Consequently, \cref{thm:artiy-restriction-theorem} implies that for every complete $(\qsig{\calO}(k,k+l),1)$-category $\calC$, restriction induces an equivalence of $\infty$-categories $\Mon^{\le k+l}_{\calO}(\calC) \simeq \Mon^{\le k}_{\calO}(\calC)$. On the other hand by \cref{prop:quasi-partition-function-equal} we have $\qsig{\calO}(k,k+l) = \sig{\calO}(k,k+l)$ so we're done.
\end{proof}

We will deduce \cref{prop:quasi-partition-function-equal} from a pair of sharper statements about (quasi-)partition complexes (\cref{prop:reduction-to-essential-partitions} and \cref{prop:reduction-to-quasi-partition-of-nonunital-morphism}), whose proofs will appear in \cref{sect:partition-complexes}. The purpose of \cref{sect:non-unital-unitary} is to establish a certain canonical factorization system on any $\infty$-operad which we make use of in \cref{sect:partition-complexes}.

\subsubsection{Reducing \cref{prop:quasi-partition-function-equal} to claims about (quasi-)partition complexes} 

For the remainder of the section we fix an $\infty$-operad $\calO$ with structure map $\bracket{-}:\calO \to \Fin_{\ast}$.

\begin{defn}\label{defn:unitary-nonunital}
	Let $\alpha: x \actarrow y$ be a morphism in $\calO^{\act}$.	
	\begin{itemize}
		\item We say that $\alpha$ is \textit{non-unital} if the underlying map of sets $\bracket{\alpha}^{\circ}: \bracket{x}^{\circ} \to \bracket{y}^{\circ}$ is surjective. We denote by $\hl{\calO^{\act,\nonu} }\subseteq \calO^{\act}$ the wide subcategory of non-unital morphisms.
		\item We say that $\alpha$ is \textit{unitary} if it has unique left lifting property with respect to non-unital morphisms. We denote by $\hl{\calO^{\mrm{u}}} \subseteq \calO^{\act}$ the wide subcategory of unitary morphisms.
	\end{itemize}
\end{defn}

\begin{war}
	The property of being unitary \textit{can not} be detected at the level of underlying pointed sets. This should be contrasted with non-unital morphisms which are defined by a property of the underlying map of pointed sets.
\end{war}

\begin{rem}
	The astute reader may have noticed that the notions in  \cref{defn:unitary-nonunital} make sense for an arbitrary analytic pattern. Since we only make use of these notions for $\infty$-operads we have chosen to phrase their definition in that generality.
\end{rem}

In \cref{sect:non-unital-unitary} we prove the following.

\begin{prop}\label{prop:nonunital-unitary-factorization-system}
	Let $\calO$ be an $\infty$-operad. Then $(\calO^{\act,\nonu},\calO^{\mrm{u}})$ defines a factorization system on $\calO^{\act}$.
\end{prop}

\begin{notation}\label{rem:nonunital-factorization-system}
	Given an active morphism $\alpha : x \actarrow y$ in $\calO$ we denote its (nonunital,unitary) factorization as follows:
	\[ \alpha : x \overset{\hl{\alpha^{\nonu}}}{\actarrow} \hl{\im(\alpha)} \overset{\hl{\alpha^{\rm u}}}{\actarrow} y\]
	Note that $\bracket{\alpha}^{\nonu}=\bracket{\alpha^{\nonu}}$, $\bracket{\alpha}^{\rm u} = \bracket{\alpha^{\rm u}}$ and $\bracket{\im(\alpha)}=\im(\bracket{\alpha})$.
\end{notation}

\begin{defn}\label{defn:precise-definition-of-partition-category}
	Let $\mu : x \rightsquigarrow z$ be an active morphism in $\calO$. A \hl{partition of $\mu$} is a factorization
	\[\begin{tikzcd}
	& y \\
	x && z,
	\arrow["\alpha"{description}, squiggly, from=2-1, to=1-2]
	\arrow[squiggly, from=1-2, to=2-3]
	\arrow["\mu"{description}, squiggly, from=2-1, to=2-3]
	\end{tikzcd}\]
	such that:
	\begin{enumerate}
		\item $\alpha: x \rightsquigarrow y$ is \textit{not} maximally active.
		\item $|y|<|x|$.
		\item  $\alpha: x \actarrow y$ is non-unital.
	\end{enumerate}
	We define the \hl{partition category} of $\mu$ as the full subcategory $\hl{\Part_{\calO}(\mu) }\subseteq \Fact_{\calO}(\mu)$ spanned by partitions of $\mu$. 
\end{defn}

\begin{rem}
	Note that conditions (1) and (2) say that $\mu$ is a quasi-partition. So we may have also defined a partition to be a quasi-partition satisfying condition (3).
\end{rem}

\begin{rem}\label{rem:partition-depends-only-nonunital}
	Denote $\calO^\nonu \coloneq  \Fin^\surj \times_\Fin \calO \subseteq \calO$.
	An active morphism 
	$\mu \colon x \actarrow y$ 
	in $\calO$ is non-unital if and only if it lies in the subcategory $\calO^{\nonu}$.
	Examaning \cref{defn:precise-definition-of-partition-category} we see that when $\mu$ is non-unital we have $\Part_{\calO^\nonu}(\mu) = \Part_\calO(\mu)$ and thus $\spart{\calO^\nonu}{\mu} \simeq \spart{\calO}{\mu}$.
\end{rem}

\begin{defn}
	Let $\mu :x \actarrow y$ be an active morphism in $\calO$. Define the \textit{partition complex of $\mu$} to be the realization of the partition category: 
	\[ \hl{\spart{\calO}{\mu} \coloneq } \, \big| \Part_{\calO}(\mu) \big|\]
\end{defn}
\begin{defn}
	For a any $1 \le k_0 \le k_1 \le \infty$ define:
	\[\hl{\sig{\calO}(k_0,k_1)\coloneq } \inf \bigg\{ \conn \bigg(\spart{\calO}{\mu} \bigg) \bigg| \, \mu : x \rightsquigarrow y \text{ non-unital, maximally active morphism with } k_0 < |x| \le k_1 \bigg\} \]
\end{defn}

\begin{obs}\label{partition-connectivity-depends-on-nonunital}
	\cref{rem:partition-depends-only-nonunital} implies that for any \operad{} $\calO$ we have $\sig{\calO}=\sig{\calO^{\nonu}}$.
\end{obs}

Below we state the key technical results of this section after which we deduce \cref{prop:quasi-partition-function-equal} as an immediate consequence.

\begin{prop}[Proved in \cref{proof:reduction-to-essential}]\label{prop:reduction-to-essential-partitions}
	Let $\mu: x \actarrow z$ be a non-unital maximally active morphism in $\calO$. Then the natural fully faithful inclusion,
	\[ \Part_{\calO}(\mu) \mono \QPart_{\calO}(\mu),\]
	admits a right adjoint. In particular it induces an equivalence $\spart{\calO}{\mu} \simeq \qpart{\calO}{\mu}$.
\end{prop}
\begin{prop}[Proved in \cref{proof:eduction-to-quasi}]\label{prop:reduction-to-quasi-partition-of-nonunital-morphism}
	Let $\mu: x \actarrow z$ be a a maximally active morphism in $\calO$. There's a natural functor,
	\[ \QPart_{\calO}(\mu) \to \QPart_{\calO}(\mu^{\nonu}),\]
	inducing an equivalence on realizations $\qpart{\calO}{\mu} \simeq \qpart{\calO}{\mu^{\nonu}}$.
\end{prop}

\begin{proof}[Proof of \cref{prop:quasi-partition-function-equal} given Propositions \ref{prop:reduction-to-quasi-partition-of-nonunital-morphism} and \ref{prop:reduction-to-quasi-partition-of-nonunital-morphism}]
	For every maximally active morphism $\mu$ in $\calO$ we have:
	\[ \qpart{\calO}{\mu} \overset{(\ref{prop:reduction-to-quasi-partition-of-nonunital-morphism})}{\simeq} \qpart{\calO}{\mu^{\nonu}} \overset{(\ref{prop:reduction-to-essential-partitions})}{\simeq} \spart{\calO}{\mu^{\nonu}}\]
	It follows that for every $m \ge 1$ the following collections of spaces are equivalent:
	\[ \bigg\{ \qpart{\calO}{\mu} \, \big| \, \mu \text{ maximally active in $\calO$ of arity } m \bigg\} = \bigg\{ \spart{\calO}{\mu} \, \big| \, \mu \text{ non-unital maximally active in $\calO$ of arity } m\bigg\}\]
	In particular we have $\sig{\calO}=\qsig{\calO}$.
\end{proof}

\subsection{Unitary and nonunital morphisms}\label{sect:non-unital-unitary}

The purpose of this subsection is to prove \cref{prop:nonunital-unitary-factorization-system}. We begin by recalling a fundamental property of $\infty$-operads which distinguishes them from arbitrary analytic patterns. 

\begin{prop}[{\cite[Proposition 9.15]{RH-algebraic}}]\label{prop:decomposition-of-active-morphisms}
	For every $y \in \calO$ there's a canonical equivalence of $\infty$-categories:
	\begin{equation*}\label{diag:decomposition-of-active-morphisms}
	    \calO^{\act}_{/y} \simeq \underset{j \in \bracket{y}^{\circ}}{\prod} \calO^{\act}_{/y_j}\tag{$\star$} 
	\end{equation*}
	
\end{prop}

\begin{notation}
	Given an active morphism $\alpha : x \actarrow y \in \calO^{\act}_{/y}$ we denote the tuple corresponding to $\alpha$ under \eqref{diag:decomposition-of-active-morphisms} by
	\[\big( \hl{\alpha_j: x_j \actarrow y_j} | \, j\in \bracket{y}^{\circ} \big) \in \underset{j\in \bracket{y}^{\circ}}{\prod} \calO^{\act}_{/y_j}\]
	In the other direction given a tuple of active morphisms $\big( \alpha_j: x_j \actarrow y_j | \, j\in \bracket{y}^{\circ} \big) $ we denote the active morphism corresponding to this tuple under \eqref{diag:decomposition-of-active-morphisms} by 
	\[ \hl{\underset{j \in \bracket{y}^{\circ}}{\oplus} \alpha_j : \underset{j \in \bracket{y}^{\circ}}{\oplus} x_j \actarrow \underset{j \in \bracket{y}^{\circ}}{\oplus} y_j}\]
	These constructions are of course inverses of each other in the sense that for every active morphism $\alpha : x \actarrow y$ there's a canonical isomorphism 
	\[ \alpha \simeq \underset{j \in \bracket{y}^{\circ}}{\oplus} \alpha_j\]
\end{notation}

\begin{rem}\label{rem:0-object-of-operad}
	Let $\calO$ be an $\infty$-operad. The terminal object $\bracket{0} \in \Fin_{\ast}$ admits a (necessarily unique) lift to a terminal object $0 \in \calO$. Note that even though $\bracket{0}$ is the initial object of $\Fin^{\act}_{\ast} \simeq \Fin$ this will \textit{not} be the case for a general $\infty$-operad. In fact $0 \in \calO$ is an initial object of $ \calO^{\act}$ if and only if $\calO$ is a \textit{unital $\infty$-operad} in the sense of \cite[Definition 2.3.1.1]{HA}. Despite this fact $0 \in \calO^{\act}$ still behaves like the empty set in the following respects: 
	\begin{itemize}
		\item Any morphism $\alpha: x \actarrow 0$ is an isomorphism. To see this note that since $\alpha$ is active it gives a map $\bracket{\alpha}^{\circ} : \bracket{x}^{\circ} \to \bracket{y}^{\circ} = \emptyset$ and thus $\bracket{x}^{\circ}=\emptyset$ which by uniqueness gives $x=0$. 
		\item For every object $x$ there's a canonical isomorphism $x \oplus 0 \simeq x$.
	\end{itemize} 
\end{rem}

\begin{lem}\label{lem:morphism-from-emptyset-is-injective}
	Every morphism $\alpha : 0 \actarrow z$ in $\calO^{\act}$ is unitary.
\end{lem}
\begin{proof}
	We must show that in every square diagram:
	\[\begin{tikzcd}
	x & {0} \\
	{y} & z
	\arrow[squiggly,from=1-2, to=2-2]
	\arrow["{\alpha}"', two heads, from=1-1, to=2-1]
	\arrow[squiggly, from=1-1, to=1-2]
	\arrow[squiggly, from=2-1, to=2-2]
	\arrow["{\exists!}"{description}, dashed, from=2-1, to=1-2]
	\end{tikzcd}\]
	where $\alpha: x \epi y$ is active non-unital there exists a unique dashed lift making the diagram commute. 
	
	By \cref{rem:0-object-of-operad} the top horizontal map is an isomorphism. Since $\alpha$ is non-unital, $|y|\le|x|=0$ and therefore $|y| =0$. Again by \cref{rem:0-object-of-operad} it follows that $\alpha$ is an isomorphism. We conclude that the square above  may be rewritten as follows:
	\[\begin{tikzcd}
	0 & 0 \\
	0 & z
	\arrow["\sim", from=1-1, to=1-2]
	\arrow["\wr"', from=1-1, to=2-1]
	\arrow[squiggly,from=2-1, to=2-2]
	\arrow[squiggly, from=1-2, to=2-2]
	\end{tikzcd}\]
	Clearly any such square admits a unique lift so we're done.
\end{proof}

\begin{lem}\label{lem:box-of-strongly-injective}
	For an active morphism $\alpha: x \actarrow y$ the following are equivalent:
	\begin{enumerate}
		\item $\alpha = \underset{j \in \bracket{y}^{\circ}}{\oplus} \alpha_j :\underset{j \in \bracket{y}^{\circ}}{\oplus}  x_j \actarrow \underset{j \in \bracket{y}^{\circ}}{\oplus} y_j $ is unitary.
		\item $\alpha_j$ is unitary for every $j \in \bracket{y}^{\circ}$.
	\end{enumerate}
\end{lem}
\begin{proof}
	Let $S$ denote the collection of all non-unital morphisms in $\calO^{\act}_{/ y}$ and let $S_j$ denote the collection of all non-unital morphisms in $\calO^{\act}_{/y_j}$. Using that non-unital morphisms are exactly those whose underlying map of sets is surjective we see that the canonical equivalence,
	\[ \calO^{\act}_{/y} \simeq \underset{j \in \bracket{y}^{\circ}}{\prod} \calO^{\act}_{/y_j},\] 
	of \cref{prop:decomposition-of-active-morphisms} identifies the collection $S$ on the left with the collection $\underset{j \in \bracket{y}^{\circ}}{\mathlarger{\prod}} S_j$ on the right. It follows that $\alpha : x \actarrow y\in \calO^{\act}_{/y}$ is $S$-local if and only if $\alpha_j:x_j \actarrow y_j \in \calO^{\act}_{/y_j}$ is $S_j$-local for every $j\in \bracket{y}^{\circ}$. Consequently, by \cref{lem:lifting-property-equivalent-to-local-in-slice}, $\alpha$ is unitary if and only if $\alpha_j$ is unitary for every $j \in \bracket{y}^{\circ}$.
\end{proof}

\begin{prop}\label{prop:every-morphism-is-nonunital-composed-with-unitary}
	Every morphism $\alpha: x \actarrow y$ in $\calO^{\act}$ can be factored as a composition of a non-unital morphism followed by a unitary morphism.
\end{prop}
\begin{proof}
	To simplify notation we define:
	\begin{align*}
	J_+ \coloneq  \im(\bracket{\alpha}^{\circ}) \subseteq \bracket{y}^{\circ} && J_0 \coloneq  \bracket{y}^{\circ} \setminus J_+
	\end{align*}
	Observe that $\underset{j \in J_0}{\oplus} x_j =  \underset{j \in J_0}{\oplus} 0 = 0$ and therefore $\big(\underset{j \in J_+}{\oplus} y_j  \big)\oplus \big(\underset{j \in J_0}{\oplus} x_j \big) \simeq \underset{j \in J_+}{\oplus} y_j $. On the other hand we also have $\big(\underset{j \in J_+}{\oplus} y_j\big) \oplus \big(\underset{j \in J_0}{\oplus} y_j\big) \simeq \underset{j \in \bracket{y}^{\circ}}{\oplus} y_j \simeq y$. Using these isomorphisms we may define a morphism $\alpha^{\mrm{u}}: \underset{j \in J_+}{\oplus} y_j  \actarrow y$ so that the following diagram commutes
	\[\begin{tikzcd}
	{\underset{j \in J_+}{\oplus} y_j } &&& y \\
	{\big(\underset{j \in J_+}{\oplus} y_j  \big)\oplus \big(\underset{j \in J_0}{\oplus} x_j \big)} &&& {\big(\underset{j \in J_+}{\oplus} y_j\big) \oplus \big(\underset{j \in J_0}{\oplus} y_j\big)}
	\arrow["{\big(\underset{j \in J_+}{\oplus} \id_{y_j} \big) \oplus \big( \underset{j \in J_0}{\oplus} \alpha_j \big) }"', from=2-1, to=2-4]
	\arrow["{\alpha^{\mrm{u}}}", from=1-1, to=1-4]
	\arrow["{\mathlarger{\wr}}", from=2-1, to=1-1]
	\arrow["{\mathlarger{\wr}}", from=1-4, to=2-4]
	\end{tikzcd}\]
	Note that $\alpha^u$ is unital by \cref{lem:box-of-strongly-injective}. Similarly define $\alpha^{\nonu} : x \actarrow \underset{j \in \bracket{y}^{\circ}} \oplus y_j$ so that the following diagram commutes
	\[\begin{tikzcd}
	x &&& {\underset{j \in J_+}{\oplus} y_j } \\
	{\big(\underset{j \in J_+}{\oplus} x_j  \big)\oplus \big(\underset{j \in J_0}{\oplus} x_j \big)} &&& {\big(\underset{j \in J_+}{\oplus} y_j\big) \oplus \big(\underset{j \in J_0}{\oplus}x_j\big)}
	\arrow["{\alpha^{\nonu}}", from=1-1, to=1-4]
	\arrow["{\mathlarger{\wr}}"', from=1-1, to=2-1]
	\arrow["{\mathlarger{\wr}}"', from=2-4 , to=1-4]
	\arrow["{\big(\underset{j \in J_+}{\oplus} \alpha_j \big) \oplus \big(\underset{j \in J_0}{\oplus} \id_{x_j} \big)}"', from=2-1, to=2-4]
	\end{tikzcd}\]
	Note that $\alpha^{\nonu}$ is non-unital since by construction $\bracket{\alpha^{\nonu}}^{\circ}: \bracket{x}^{\circ} \to \im \bracket{\alpha}^{\circ}$ is surjective. Finally we show that $\alpha = \alpha^{\mrm{u}} \circ  \alpha^{\nonu}$:
	\begin{align*}
	\alpha & \simeq \underset{j \in \bracket{y}^{\circ}}{\oplus} \alpha_j \\
	& \simeq \big(\underset{j \in J_+}{\oplus} \alpha_j \big) \oplus \big( \underset{j \in J_0}{\oplus} \alpha_j \big) \\
	& \simeq \bigg( \big(\underset{j \in J_+}{\oplus} \id_{y_j} \big) \oplus \big( \underset{j \in J_0}{\oplus} \alpha_j \big) \bigg)  \circ \bigg( \big(\underset{j \in J_+}{\oplus} \alpha_j \big) \oplus \big(\underset{j \in J_0}{\oplus} \id_0 \big)\bigg) \\
	& \simeq \alpha^{\mrm{u}} \circ \alpha^{\nonu}
	\end{align*}
\end{proof}

We are now in a position to prove \cref{prop:nonunital-unitary-factorization-system}.

\begin{proof}[Proof of \cref{prop:nonunital-unitary-factorization-system}]
	By definition we have $\calO^{\mrm{u}}=(\calO^{\act,\nonu})^{\perp}\subseteq \calO^{\nonu}$. It thus suffices to show that every active morphism can be factored as a composition of a non-unital active morphism followed by a unitary morphism. This is the content of \cref{prop:every-morphism-is-nonunital-composed-with-unitary}.
\end{proof}

\subsection{Partition complexes of \texorpdfstring{$\infty$-operads}{Infinite Operads}}\label{sect:partition-complexes}

In this subsection we prove \cref{prop:reduction-to-essential-partitions} and \cref{prop:reduction-to-quasi-partition-of-nonunital-morphism}.

\begin{defn}
	For a morphism $\mu:x \actarrow z$ in $\calO^{\act}$ we let $\hl{\Fact^{\nonu}_{\calO}(\mu) }\subseteq \Fact_{\calO}(\mu)$ denote the full subcategory consisting of factorizations,
	\[\begin{tikzcd}
	& y \\
	x && z
	\arrow[squiggly, from=1-2, to=2-3]
	\arrow["\mu"', squiggly, from=2-1, to=2-3]
	\arrow["\alpha", squiggly, from=2-1, to=1-2]
	\end{tikzcd}\]
	such that $\alpha: x\actarrow y$ is non-unital. 
\end{defn}

\begin{lem}\label{lem:non-unital-Fact-is-coreflective}
	Let $\mu:x \actarrow z$ be a non-unital morphism in $\calO^{\act}$. Then the fully faithful inclusion,
	\[\Fact^{\nonu}_{\calO}(\mu) \mono \Fact_{\calO}(\mu),\]
	admits a right adjoint which on objects is given by:
	\[\begin{tikzcd}
	& y & {} & {} & {\im(\alpha)} \\
	x && z & x && z
	\arrow["\beta", squiggly, from=1-2, to=2-3]
	\arrow["\mu"', squiggly, from=2-1, to=2-3]
	\arrow["{\alpha^{\nonu}}", squiggly, from=2-4, to=1-5]
	\arrow["{\beta \circ \alpha^{\rm u}}", squiggly, from=1-5, to=2-6]
	\arrow["\mu"', squiggly, from=2-4, to=2-6]
	\arrow[""{name=0, anchor=center, inner sep=0}, draw=none, from=2-3, to=1-3]
	\arrow[""{name=1, anchor=center, inner sep=0}, draw=none, from=2-4, to=1-4]
	\arrow["\alpha", squiggly, from=2-1, to=1-2]
	\arrow["\xmapsto{\phantom{Hello}}"{description}, Rightarrow, draw=none, from=0, to=1]
	\end{tikzcd}\]
\end{lem}
\begin{proof}
	By the dual of \cref{prop:slice-category-factorization-adjoint} the fully faithful inclusion $\calO^{\act,\nonu}_{x/} \mono \calO^{\act}_{x/}$ admits a right adjoint. Since $\mu \in \calO^{\act,\nonu}_{x/}$ we conclude that the fully faithful inclusion
	\[\Fact^{\nonu}_{\calO}(\mu) \simeq (\calO^{\act,\nonu}_{x/})_{/\mu} \mono (\calO^{\act}_{x/})_{/\mu} \simeq \Fact_{\calO}(\mu),\] 
	admits a right adjoint. 
    The description on objects follows by unpacking the definitions.
\end{proof}

\begin{proof}[Proof of \cref{prop:reduction-to-essential-partitions}]\label{proof:reduction-to-essential}
	Consider the following square of fully faithful inclusions:
	\[\begin{tikzcd}
	{\Part_{\calO}(\mu)} & {\QPart_{\calO}(\mu)} \\
	{\Fact^{\nonu}_{\calO}(\mu)} & {\Fact_{\calO}(\mu)}
	\arrow[hook, from=1-1, to=1-2]
	\arrow[hook, from=1-2, to=2-2]
	\arrow[hook, from=1-1, to=2-1]
	\arrow[hook, from=2-1, to=2-2]
	\end{tikzcd}\]
	By \cref{lem:non-unital-Fact-is-coreflective} the bottom map admits a right adjoint. It thus suffices to show that this right adjoint sends quasi-partitions to partitions. Let $\pi = \braces{x \actmap{\alpha} y \actmap{\beta}z} \in \QPart_{\calO}(\mu)$ be a quasi-partition of $\mu$.	We claim that $\pi^{\nonu}\coloneq  \braces{x \actmap{\alpha^{\nonu}} \im(\alpha) \actmap{\beta \circ \alpha^{\mrm{u}}}z}$ is a partition of $\mu$. Since $\alpha^{\nonu}$ is non-unital by definition, it suffices to show that $\pi^{\nonu}$ is a quasi-partition. Equivalently we must show that $\alpha^{\nonu}$ is not maximally active and that $|\im(\alpha)|< |x|$. Since both of these properties are detected at the level of underlying pointed sets we may assume that $\calO = \Fin_{\ast}$ where the claim follows from a straightforward check.
\end{proof}

The remainder of this subsection is devoted to proving \cref{prop:reduction-to-quasi-partition-of-nonunital-morphism}.

\begin{lem}\label{lem:factorization-category-decomposition}
	Let $\mu : x \actarrow z$ be a morphism in $\calO^{\act}$. There's a canonical equivalence of $\infty$-categories:
	\[ \Fact_{\calO}(\mu) \simeq \underset{j \in \bracket{y}^{\circ}}{\prod} \Fact_{\calO}(\mu_j)\]
\end{lem}
\begin{proof}
    Follows from
	\begin{align*}
	\Fact_{\calO}(\mu) & = \calO^{\act}_{x//\mu} \\
	& = \calO^{\act}_{\mu//z} & \text{(\cref{lem:under-over-vs-over-under})}\\
	& \simeq \bigg( \underset{j \in \bracket{z}^{\circ}}{\prod} \calO^{\act}_{/z_j} \bigg)_{\mu/} & \text{(\cref{prop:decomposition-of-active-morphisms})}\\
	& \simeq \underset{j \in \bracket{z}^{\circ}}{\prod} \calO^{\act}_{\mu_j //z_j} \\
	& \simeq \underset{j \in \bracket{z}^{\circ}}{\prod} \calO^{\act}_{x_j //\mu_j} &  \text{(\cref{lem:under-over-vs-over-under})}\\
	& = \underset{j \in \bracket{z}^{\circ}}{\prod} \Fact_{\calO}(\mu_j).\\
	\end{align*}
\end{proof}

\begin{notation}
	Let $\mu : x \actarrow z$ be a maximally active morphism in $\calO$. 
	\begin{itemize}
		\item Let $\hl{i_{\mu}} \in \bracket{z}^{\circ}$ denote the unique element lying in the image of $\bracket{\mu}^{\circ}: \bracket{x}^{\circ} \to \bracket{z}^{\circ}$.
		\item Let $\hl{\lambda_{\mu}} :z \inert z_{i_{\mu}}$ denote the unique inert map corresponding to $i_{\mu}$ under the canonical equivalence $\calO^{\el}_{z/} \simeq \bracket{z}^{\circ}$. 
	\end{itemize}  
	Note that $\mu^{\nonu} =( \lambda_{\mu} \circ \mu)^{\act}$.
\end{notation}

\begin{defn}
	Let $\mu: x \actarrow z$ be a morphism in $\calO^{\act}$. Define $\hl{\Fact^{<|x|}_{\calO}(\mu)} \subseteq \Fact_{\calO}(\mu)$ as the full subcategory whose objects are the active factorizations,
	\[\begin{tikzcd}
	& y \\
	x && z
	\arrow["\mu"', squiggly, from=2-1, to=2-3]
	\arrow[squiggly, from=2-1, to=1-2]
	\arrow[squiggly, from=1-2, to=2-3]
	\end{tikzcd}\]
	in which $|y| < |x|$. 
\end{defn}

\begin{lem}\label{lem:reduction-map-for-factorizations}
	Let $\mu : x \actarrow z$ be a maximally active morphism in $\calO$. There's a natural functor,
	\[  \Fact^{<|x|}_{\calO}(\mu) \to \Fact^{<|x|}_{\calO}(\mu^{\nonu}) \]
	having the following property:
	\begin{itemize}
		\item $\pi \in \Fact^{<|x|}_{\calO}(\mu)$ is a quasi-partition of $\mu$ if and only if its image $\pi_{i_{\mu}} \in \Fact^{<|x|}_{\calO}(\mu^{\nonu})$ is quasi-partition of $\mu^{\nonu}$.
	\end{itemize} 
\end{lem}
\begin{proof}
	Composing the equivalence of \cref{lem:factorization-category-decomposition} with the projection onto the $i_{\mu}$-th coordinate gives a functor
	\[ \Fact_{\calO}(\mu) \simeq \underset{j \in \bracket{z}^{\circ}}{\prod} \Fact_{\calO}(\mu_j) \simeq \Fact_{\calO}(\mu^{\nonu}) \times  \underset{i_{\mu} \neq j \in \bracket{z}^{\circ} }{\mathlarger{\prod}}\Fact_{\calO}(\mu_j) \to \Fact_{\calO}(\mu^{\nonu})\]
	We claim that it restricts to a functor $\Fact^{<|x|}_{\calO}(\mu) \to \Fact^{<|x|}_{\calO}(\mu^{\nonu})$. To show this let $\pi$ be a factorization of $\mu$ corresponding to the following diagram
	\[\begin{tikzcd}
	& y \\
	x && z
	\arrow["\mu"', squiggly, from=2-1, to=2-3]
	\arrow["\alpha", squiggly, from=2-1, to=1-2]
	\arrow["\beta", squiggly, from=1-2, to=2-3]
	\end{tikzcd}\]
	Unpacking deifnitions we see that its image $\pi_{i_{\mu}}$ is the factorization of $\mu^{\nonu} \coloneq  \mu_{i_{\mu}}$ corresponding to the following diagram
	\[\begin{tikzcd}
	& y_{i_{\mu}} \\
	x && z_{i_{\mu}}
	\arrow["\mu^{\nonu}"', squiggly, from=2-1, to=2-3]
	\arrow["\alpha_{i_{\mu}}", squiggly, from=2-1, to=1-2]
	\arrow["\beta_{i_{\mu}}", squiggly, from=1-2, to=2-3]
	\end{tikzcd}\]
	where we have $y_{i_{\mu}} = \Lambda(\lambda_{\mu} \circ \beta)$, $\beta_{i_{\mu}} = (\lambda_{\mu} \circ \beta)^{\act}$ and $\alpha_{i_{\mu}} = (\lambda_{\mu} \circ \beta)^{\int} \circ \alpha$. In particular we see that 
	\[ |y_{i_{\mu}}| \le | \Lambda(\lambda_{\mu} \circ \beta)| \le |y| < |x|\]
	It remains to show that $\pi$ is a quasi-partition if and only if $\pi_{i_{\mu}}$ is a quasi-partition of $\mu^{\nonu}$. Since we always have $|y|<|x|$ and $|y_{i_{\mu}}|<|x|$ by assumption this is equivalent to the statement that $\alpha_{i_{\mu}}$ is maximally active if and only if $\alpha$ is maximally active which follows from \cref{lem:inert-cancellation-for-maximal}.
\end{proof}

\begin{cor}\label{cor:partition-category-as-pullback}
	Let $\mu: x \actarrow z$ be a maximally active morphism in $\calO$. The following commutative square,
	\[\begin{tikzcd}
	{\QPart_{\calO}(\mu)} & {\Fact^{<|x|}_{\calO}(\mu)} \\
	{\QPart_{\calO}(\mu^{\nonu})} & {\Fact^{<|x|}_{\calO}(\mu^{\nonu}),}
	\arrow[""{name=0, anchor=center, inner sep=0}, from=2-1, to=2-2]
	\arrow[from=1-2, to=2-2]
	\arrow[from=1-1, to=2-1]
	\arrow[from=1-1, to=1-2]
	\arrow["\lrcorner"{anchor=center, pos=0.125}, draw=none, from=1-1, to=0]
	\end{tikzcd}\]
	is cartesian.
\end{cor}
\begin{proof}
	Both horizontal maps are fully faithful by definition. It thus suffices to show that the natural map from $\QPart_{\calO}(\mu)$ to the pullback is essentially surjective which follows from \cref{lem:reduction-map-for-factorizations}.
\end{proof}

We recall the notion of smooth/proper functors of $\infty$-categories introduced in \cite[Definition 4.1.2.9]{HTT}. The following is essentially a reformulation of \cite[Proposition 4.1.2.11]{HTT}.

\begin{lem}[Lurie]\label{lem:smooth-characterization}
	Let $f : X \to Y$ be a functor of $\infty$-categories. The following are equivalent:
	\begin{enumerate}
		\item For every $y \in Y$ the natural map,
		\[X_y \coloneq  \{y\} \times_Y X \to Y_{y/} \times_Y X =: X_{y/},\]
		is initial. 
		\item For every pullback square,
		\[\begin{tikzcd}
		{X'} & X \\
		{Y'} & Y,
		\arrow["{g'}", from=1-1, to=1-2]
		\arrow["f", from=1-2, to=2-2]
		\arrow["{f'}"', from=1-1, to=2-1]
		\arrow["f"', from=2-1, to=2-2]
		\arrow["\lrcorner"{anchor=center, pos=0.125}, draw=none, from=1-1, to=2-2]
		\end{tikzcd}\]
		the commutative square,
		\[\begin{tikzcd}
		{\Fun(Y,\calS)} & {\Fun(X,\calS)} \\
		{\Fun(Y',\calS)} & {\Fun(X',\calS)}
		\arrow["{{f'}^{\ast}}"', from=2-1, to=2-2]
		\arrow["{g^{\ast}}"', from=1-1, to=2-1]
		\arrow["{f^{\ast}}", from=1-1, to=1-2]
		\arrow["{{g'}^{\ast}}"', from=1-2, to=2-2]
		\end{tikzcd}\]
		is right adjointable (in the sense of \cite[Definition 4.7.4.13]{HA}). That is, the Beck-Chevalley map $g^{\ast}f_{\ast} \to f'_{\ast} {g'}^{\ast}$ is an equivalence.
	\end{enumerate}
\end{lem}
\begin{proof}
	$(1) \implies (2)$ Let $F: X \to \calS$ be a functor. Evaluating the Beck-Chevalley map for $F$ at $y'\in Y'$ and then using the formula for right Kan extension gives
	\[  \underset{x \in X_{g(y')}}{\lim} F(x) \simeq \underset{x \in X_{g(y')/}}{\lim} F(x) \simeq (g^{\ast}f_{\ast}F)(y') \to (f'_{\ast} {g'}^{\ast}F)(y') \simeq  \underset{x' \in X'_{y'/}}{\lim} F(g'(x')) \simeq \underset{x' \in X'_{y'}}{\lim} F(g'(x')) \]
	where the first and last equivalences follow from assumption (1). The composite is induced from the map on fibers $X'_{y'} \to X_{g(y')}$ in the square from (2). But by assumption that square is cartesian and thus the map on fibers is an equivalence.

	$(2) \implies (1)$ 	Let $y \in Y$. Since $(2)$ holds by assumption and is manifestly stable under pullback the natural map $X_{y/}\coloneq X \times_{Y} Y_{y/} \to Y_{y/}$ also satisfies it. Consider now the pullback square:
	\[\begin{tikzcd}
	{X_y} & X_{y/} \\
	y & Y_{y/}
	\arrow[from=1-1, to=1-2]
	\arrow[from=1-2, to=2-2]
	\arrow[from=1-1, to=2-1]
	\arrow[from=2-1, to=2-2]
	\arrow["\lrcorner"{anchor=center, pos=0.125}, draw=none, from=1-1, to=2-2]
	\end{tikzcd}\]
	For every functor  $F: X_{y/} \to \calS $ we have that the Beck-Chevalley map,
	\[ \underset{x \in X_{y/}}{\lim} F(x) \to \underset{x' \in X_y}{\lim} F(x'),  \] 
	is an equivalence. We conclude that $X_y \to X_{y/}$ is initial as required.
\end{proof}

\begin{defn}
	A functor $f: X \to Y$ is called: 
	\begin{enumerate}
		\item \textit{Smooth} If it satisfies the equivalent conditions of \cref{lem:smooth-characterization}
		\item \textit{Proper} if $f^{\op}$ is smooth.
	\end{enumerate}
\end{defn}

\begin{cor}\label{lem:smooth-basechange}
	Smooth functors are stable under base change. That is, if $f: X \to Y$ is a smooth functor and the square:
	\[\begin{tikzcd}
	{X'} & X \\
	{Y'} & Y
	\arrow[from=1-1, to=1-2]
	\arrow["\phi", from=2-1, to=2-2]
	\arrow["f", from=1-2, to=2-2]
	\arrow["{f'}"', from=1-1, to=2-1]
	\arrow["\lrcorner"{anchor=center, pos=0.125}, draw=none, from=1-1, to=2-2]
	\end{tikzcd}\]
	is cartesian. Then $f'$ is a smooth functor.
\end{cor}
\begin{proof}
	Follows from \cref{lem:smooth-characterization} as $(2)$ is manifestly stable under base change.
\end{proof}

\begin{lem}\label{lem:smoth-weak-contractible-fibers}
	Let $f: X \to Y$ be a smooth (resp. proper) functor with weakly contractible fibers. Then:
	\begin{enumerate}
		\item $|f| : |X| \to |Y|$ is an equivalence.
			
		\item For every pullback square
		\[\begin{tikzcd}
		{X'} & X \\
		{Y'} & Y,
		\arrow[from=1-1, to=1-2]
		\arrow["\phi", from=2-1, to=2-2]
		\arrow["f", from=1-2, to=2-2]
		\arrow["{f'}"', from=1-1, to=2-1]
		\arrow["\lrcorner"{anchor=center, pos=0.125}, draw=none, from=1-1, to=2-2]
		\end{tikzcd}\]
		the map $f'$ is smooth (resp. proper) with weakly contractible fibers.
	\end{enumerate}
\end{lem}
\begin{proof}
	(1) The smooth case follows from the proper case by observing that $f$ is proper if and only if $f^{\rm op}$ is smooth, $|f^{\rm op}| \simeq |f|$ and $ |X^{\rm op}_y| \simeq |X_y|$. We may thus assume that $f$ is proper. The formula for left Kan extension then gives
	\[ |X| \simeq \underset{X}{\colim} \, \pt \simeq \underset{Y}{\colim} \, f_! (\pt) \simeq \underset{y \in Y}{\colim} \, |X_{/y}| \simeq \underset{y \in Y}{\colim} \, |X_y| \simeq \underset{Y}{\colim} \, \pt \simeq |Y|\]
	
	(2) Since smooth functors are stable under base change (\cref{lem:smooth-basechange}). It remains to check that all fibers of $f'$ are weakly contractible. For every $y \in Y'$, the induced map on fibers in the pullback square yields an equivalence of $\infty$-categories $X'_{y} \simeq X_{\phi(y)}$ and therefore $|X'_y| \simeq |X_{\phi(y)}|\simeq \pt$.
\end{proof}

\begin{lem}\label{lem:subfamily-of-product-is-smooth-contractible}
	Consider a square of $\infty$-categories
	\[\begin{tikzcd}
	E & {X'\times Y} \\
	X & {X'}
	\arrow[hook, from=2-1, to=2-2]
	\arrow[from=1-2, to=2-2]
	\arrow[hook, from=1-1, to=1-2]
	\arrow["\pi"', from=1-1, to=2-1]
	\end{tikzcd}\]
	such that 
	\begin{enumerate}
		\item The horizontal maps are fully faithful.
		\item $Y$ admits an initial object $y_{\emptyset} \in Y$.
		\item For every $x \in X$ we have $(x,y_{\emptyset}) \in E$.
	\end{enumerate}
	Then $\pi$ is smooth with weakly  contractible fibers.
\end{lem}
\begin{proof}
	By assumption, for every $x \in X$ we have  $(x,y_{\emptyset}) \in E$ and therefore $(x \overset{\id}{\simeq} x,y_{\emptyset}) \in E_x \subseteq E_{x/}$. To prove the claim it thus suffices to show that $(x \overset{\id}{\simeq} x,y_{\emptyset}) \in E_{x/}$ is an initial object for every $x \in X$. Indeed this would simultaneously show that the inclusion $E_x \mono E_{x/}$ is initial and that $|E_x|$ is contractible. Note that there's a natural fully faithful inclusion $E_{x/} \mono (X' \times Y)_{x/} \times_{X'} Y \simeq X'_{x/} \times Y$. But $(x \overset{\id}{\simeq} x, y_{\emptyset})$ is clearly initial in $ X'_{x/} \times Y$ so we're done.  
\end{proof}

\begin{prop}\label{lem:reduction-map-of-factorizations-is-smooth}
	The natural functor 
	$\Fact^{<|x|}_{\calO}(\mu) \to \Fact^{<|x|}_{\calO}(\mu^{\nonu})$
	is smooth with weakly contractible fibers.
\end{prop}
\begin{proof}
	To prove this we apply \cref{lem:subfamily-of-product-is-smooth-contractible} to the following square
	\[\begin{tikzcd}
	{\Fact^{<|x|}_{\calO}(\mu)} & {\Fact_{\calO}(\mu)} \\
	{\Fact^{<|x|}_{\calO}(\mu^{\nonu})} & {\Fact_{\calO}(\mu^{\nonu})}
	\arrow[hook, from=1-1, to=1-2]
	\arrow[from=1-2, to=2-2]
	\arrow[from=1-1, to=2-1]
	\arrow[hook, from=2-1, to=2-2]
	\end{tikzcd}\]
	where we recall that by \cref{lem:factorization-category-decomposition}, the right vertical functor identifies with the projection 
	\[\Fact_{\calO}(\mu) \simeq \Fact_{\calO}(\mu^{\nonu}) \times \underset{i_{\mu} \neq j \in \bracket{z}^{\circ} }{\mathlarger{\prod}}\Fact_{\calO}(\mu_j) \to \Fact_{\calO}(\mu^{\nonu})\]
	To finish we must verify the conditions of \cref{lem:subfamily-of-product-is-smooth-contractible}.
	\begin{enumerate}
		\item Follows from definitions. 
		
		\item For $i_{\mu} \neq j \in \bracket{z}^{\circ}$ denote by $\theta_j$ the trivial factorization of $\mu_j$ given by the following diagram
		\[\begin{tikzcd}
		& 0 \\
		0 && {z_j}
		\arrow["\id", from=2-1, to=1-2]
		\arrow["{\mu_j}", from=1-2, to=2-3]
		\arrow["{\mu_j}"', from=2-1, to=2-3]
		\end{tikzcd}\]
		Clearly $\theta_j \in \Fact_{\calO}(\mu_j)$ is initial and therefore $(\theta_j : i_{\mu} \neq j \in \bracket{z}^{\circ}) \in s \underset{i_{\mu} \neq j \in \bracket{z}^{\circ} }{\mathlarger{\prod}}\Fact_{\calO}(\mu_j) $ is initial.
		
		\item For a factorization $\sigma = \braces{\mu: x \actarrow y \actarrow z}$ denote $|\sigma|"=|y|$. We must show that for every $\pi \in \Fact^{<|x|}_{\calO}(\mu^{\nonu})$ if we let $\pi' \coloneq  \pi \oplus \underset{ i_{\mu} \neq j \in \bracket{z}^{\circ}}{\bigoplus} \theta_j $ then $\pi' \in \Fact^{<|x|}_{\calO}(\mu)$ or equivalently $|\pi'| <|x|$. But clearly $|\theta_j|=0$ and therefore 
		\[|\pi'| =|\pi|+\sum_{ i_{\mu} \neq j \in \bracket{z}^{\circ}}|\theta_j| = |\pi| <|x|\]
	\end{enumerate}
\end{proof}

We are now in a position to prove \cref{prop:reduction-to-quasi-partition-of-nonunital-morphism}

\begin{proof}[Proof of \cref{prop:reduction-to-quasi-partition-of-nonunital-morphism}]\label{proof:eduction-to-quasi}
	By \cref{cor:partition-category-as-pullback} we have a pullback square
	\[\begin{tikzcd}
	{\QPart_{\calO}(\mu)} & {\Fact^{<|x|}_{\calO}(\mu)} \\
	{\QPart_{\calO}(\mu^{\nonu})} & {\Fact^{<|x|}_{\calO}(\mu^{\nonu})}
	\arrow[""{name=0, anchor=center, inner sep=0}, from=2-1, to=2-2]
	\arrow[from=1-2, to=2-2]
	\arrow[from=1-1, to=2-1]
	\arrow[from=1-1, to=1-2]
	\arrow["\lrcorner"{anchor=center, pos=0.125}, draw=none, from=1-1, to=0]
	\end{tikzcd}\]
	By \cref{lem:reduction-map-of-factorizations-is-smooth} the right vertical map is smooth with weakly contractible fibers and thus by \cref{lem:smoth-weak-contractible-fibers} the left vertical map induces an equivalence $ |\QPart_{\calO}(\mu)| \simeq  |\QPart_{\calO}(\mu^{\nonu})|$.
\end{proof}

%% file: sections/1-reduced.tex
\section{Partition complexes of the little discs operad}\label{sect:5}

The main goal of this section is to prove
\cref{thm:partition-Ed-intro}.
We do this by showing that the operadic partition complexes admit an alternative description as the fiber of the "inclusion" of the operadic \textit{"decomposables"}.
We then identify this space of decomposables in the case of $\bbE_d$ with the boundary of the Fulton-Macpherson compactification.
We state our result in the language of dendroidal homotopy theory.

\subsubsection{Dendroidal description of the partition complex}
A finite rooted tree $T$ generates a non-unital operad $\Free_\Op(T)$ by thinking of its edges as colours and its vertices as operations.
We write $\Omega^\circ$ for the open dendroidal category, introduced in \cite{moerdijk2007dendroidal}.
Objects of $\Omega^\circ$ are trees and morphisms are maps of operads $\Free_\Op(T_1) \to \Free_\Op(T_2)$.
An \operad{} $\calO \in \Op_\infty$ defines a presheaf on $\Omega^\circ$ by:
\[\xN_\star^\dend(\calO) \colon \Omega^{\circ,\op} \too \calS, \qquad T \longmapsto \hl{\xN^\dend_T(\calO)} \coloneq \Map_{\Op_\infty}(\Free_\Op(T),\calO)\]
We make this precise in \cref{obs:lurie=barwick=dendroidal}.
Note that any morphism $\Free_\Op(T_1) \to \Free_\Op(T_2)$ can be factored uniquely as a composite
$\Free_\Op(T_1) \actarrow \Free_\Op(S) \inert \Free_\Op(T_2)$ where the first map induces a bijection on outer edges and the second map is induced by an inclusion of a subtree $S \subseteq T_2$.
The former type is called active and the latter inert.

Let $C_k \in \Omega^\circ$ denote the $k$-corolla, namely the rooted tree with a single vertex and $k$-edges excluding the root.
The space $\xN^\dend_{C_k}(\calO)$ may be identified with the space of multi-morphisms in $\calO$ of arity $k$:
\[\xN^\dend_{C_k}(\calO) \simeq \Map_{\Op_\infty}(\Free_\Op(C_k),\calO) \simeq \{\langle k \rangle^\circ\} \times_{\Fin^\simeq} \Fun([1],\calO^\act)^{\simeq} \times_{\Fin^\simeq}\{\langle 1 \rangle^\circ \}\]
We write $\left(\Omega^{\circ,\op,\act}_{/C_k}\right)^{<k} \subseteq \Omega^{\circ,\op,\act}_{/C_k}$ for the full subcategory spanned by active maps $C_k \actarrow T$ such that all vertices of $T$ are of arity $<k$.

\begin{thm}\label{thm:dendroidal-partition-complex}
	Let $\calO$ be a non-unital $\infty$-operad such that $\calO^{=1}$ is an $\infty$-groupoid. 
	Then for every multi-morphism 
	$ \mu \colon (x_1,\dots,x_k) \to y$ in $\calO$ of arity $k \ge 2$ we have a canonical pullback square: 
	\[\begin{tikzcd}
	{\spart{\calO}{\mu}} & {\colim_{C_k \actarrow T, T \neq C_k}\xN^\dend_T(\calO)} \\
	{\{\mu\}} & {\xN_{C_k}^\dend(\calO)}
	\arrow[from=1-2, to=2-2]
	\arrow[from=1-1, to=1-2]
	\arrow[from=1-1, to=2-1]
	\arrow[from=2-1, to=2-2]
	\arrow["\lrcorner"{anchor=center, pos=0.125}, draw=none, from=1-1, to=2-2]
    \end{tikzcd}\]
\end{thm}

We shall now explain how to use the Fulton-Macpherson operad to deduce \cref{thm:partition-Ed-intro} from \cref{thm:dendroidal-partition-complex}.
\begin{proof}[Proof of \cref{thm:partition-Ed-intro} given  \cref{thm:dendroidal-partition-complex}]
The space
$\FM_d(k)$ is a manifold with corners whose interior is the configuration space of points modulo translations and dilations $\mrm{int}(\FM_d(k)) = \mathrm{Conf}_k(\bbR^d)/\bbR^d\rtimes \bbR_{>0}$ \cite{getzler-jones}.
Getzler and Jones noticed that these spaces assemble to a topological operad.
Such a topological operad was constructed by Markl in \cite{markl} where it is denoted by $\widetilde{\chi}$.
In particular Markl constructs an isomorphism $\widetilde{\chi}(k)=\FM_d(k)$ for all $k\ge 1$ \cite[Theorem 3.4.]{markl}.
Let us henceforth write $\FM_d$ in place of $\widetilde{\chi}$.
Salvatore proved in \cite[Proposition 3.9.]{salvatore} that $\FM_d$ is weakly equivalent to the (non-unital) little $d$-disc operad $\bbE^\nonu_d$ hence we may work with $\FM_d$ to compute the partition complex.
Finally, since $\FM_d$ is cofibrant, the map $\colim_{C_k \actarrow T, T \neq C_k} \FM_d(T) \to \FM_d(k)$ can be identified with the inclusion of the boundary $\partial \FM_d(k) \hookrightarrow \FM_d(k)$ (see also \cite[Example 2.1.13.]{goppl}).
\end{proof}

To prove \cref{thm:dendroidal-partition-complex} we will need to pass information from Lurie's model of \operads{} to dendroidal Segal spaces. 
Unfortunately, the comparison between these models is quite tricky to write down explicitly. 
As far as the author is aware, every known proof of this comparison involves some auxiliary intermediate model (and often more than one).
Our proof of \cref{thm:dendroidal-partition-complex} will also rely on such intermediate models.

\subsubsection{Notation}
We fix the following notation in this section.
	\begin{itemize}
		\item 
		We let $\hl{\xN_{\bullet}}: \Cat_{\infty} \longrightarrow \Psh_{\seg}(\Delta)$
		denote the nerve functor.
		\item 
		For
		$\calC \in \Cat_{\infty}$ 
		we write 
		$\hl{\Delta^{\op}_{/\calC}}\coloneq  \Un(\Delta^{\op} \xrightarrow{\xN_{\bullet}(\calC)} \calS)$
		and 
		$\hl{\Delta^{\op,\act}_{/\calC}}\coloneq  \Un(\Delta^{\op,\act}  \xrightarrow{\xN_{\bullet}(\calC)}\calS)$.
		\item 
		We let $\hl{\Omega}$ denote the dendroidal category \cite{moerdijk2007dendroidal}.
		Note that $\Omega^\circ$ is a full subcategory of $\Omega$.
		\item
		We let $\hl{\Delta_{\Fin}}$ denote Bawrick's category of forest sets \cite{operator-cat} and let $\hl{\Delta_\Fin^1} \subseteq \Delta_\Fin$ denote the subcategory defined in \cite{Dend-Barwick}.
    	\item 
		We let $\hl{\xN^{\otimes}_{\bullet}}: \Cat^{\otimes}_{\infty} \longrightarrow \Fun(\Delta^{\op},\CMon) \subseteq \Psh(\Delta \times \Fin^{\op}_{\ast})$
		denote the symmetric monoidal nerve functor.
	\end{itemize}
\begin{thm}[{\cite{Dend-Barwick}}]\label{thm:barwick=dendroidal}
	The span
	$\Delta_{\Fin} \leftarrow \Delta^{1}_{\Fin} \to \Omega$
	induces equivalences on Segal presheaves:
	\[\Psh(\Delta_\Fin) \simeq \Psh(\Delta_\Fin^1) \simeq \Psh(\Omega).\]
\end{thm}

\begin{thm}[{\cites[Theorem 10.16]{operator-cat}[Corollary 1.1.5]{rune-joachim}}]\label{thm:Barwick=Lurie}
	There is a canonical fully faithful embedding of $\infty$-categories:
	\[  \hl{\xN^\xbar_\star(-)} \colon \Op_\infty \hookrightarrow \Psh(\Delta_\Fin).\]
\end{thm}

\begin{rem}
    For the sake of definiteness we shall henceforth let $\xN^\xbar_\star(-)$ denote the fully faithful embedding afforded by \cite[Corollary 1.1.5]{rune-joachim}.
\end{rem}

\begin{obs}\label{obs:lurie=barwick=dendroidal}
	Combining \cref{thm:Barwick=Lurie} and \cref{thm:barwick=dendroidal} gives commutative diagram
	\[\begin{tikzcd}
		& {\Op_{\infty}} \\
		{\Seg_{\Delta^{\op}_{\Fin}}(\calS)} && {\Seg_{\Omega^{\op}}(\calS)} \\
		& {\Seg_{\Delta^{1,\op}_{\Fin}}(\calS).}
		\arrow["\simeq"', from=2-1, to=3-2]
		\arrow["\simeq", from=2-3, to=3-2]
		\arrow["{\xN^\xbar_\star}"', hook', from=1-2, to=2-1]
		\arrow["{\xN^\dend_\star}", hook, from=1-2, to=2-3]
	\end{tikzcd}\]
	In particular, for any $\infty$-operad
	$\calO \to \Fin_{\ast}$
	we have a canonical equivalence $\xN^\xbar_\star(\calO)\big|_{\Delta^{1,\op}_{\Fin}} \simeq  \xN^\dend_\star(\calO)\big|_{\Delta^{1,\op}_{\Fin}}$
\end{obs}

\begin{rem}
    The equivalence between non-unital \operads{} in the sense of Lurie and complete Segal preshaves on $\Omega^\circ$ was first proved in \cite{heuts}.
\end{rem}

\begin{obs}\label{obs:barwick=Lurie}
    Let $\Env: \Op_{\infty} \longrightarrow \Cat^{\otimes}_{\infty}$
    denote Lurie's symmetric monoidal envelope \cite[Proposition 2.2.4.9.]{HA}.
	The Barwick nerve functor $\xN^\xbar_\star$ is essentially characterized
	by the following diagram:
	\[\begin{tikzcd}
		{\Un_{\Delta^{\op} \times \Fin_{\ast}}(\xN^{\otimes}_{\bullet}(\Env(\calO)))} & {\Un_{\Delta^{\op} \times \Fin_{\ast}}(\xN^{\otimes}_{\bullet}(\Env(\Fin_{\ast})))} & {\Delta^{\op} \times \Fin_{\ast}} \\
		{\Un_{\Delta^{\op}_{\Fin}}(\xN^\xbar_\star(\calO))} & {\Delta_{\Fin}^{\op}} & {\Delta^{\op}}
		\arrow[from=1-3, to=2-3]
		\arrow[from=1-2, to=1-3]
		\arrow[from=2-2, to=2-3]
		\arrow[from=2-1, to=2-2]
		\arrow[from=1-1, to=1-2]
		\arrow[from=1-2, to=2-2]
		\arrow[from=1-1, to=2-1]
	\end{tikzcd}\]
	where the left square is cartesian and 
	(less importantly for us) the middle vertical functor is a fiberwise localization over $\Delta^{\op}$.
\end{obs}

\begin{defn}
    For $\calO \in \Op_\infty$ and $A \in \Fin$ we use the following notation:
	\begin{itemize}
		\item 
		We write $\hl{\Fact_A(\calO)} \coloneq  \{A\} \times_{\Fin^{\act}} \Fun([2],\calO^{\act}) \times_{\Fin^{\act}} \{1\}$,
		\item 
		We write $\hl{\Fact_A} \coloneq  \Fact_A(\Fin_{\ast}) \simeq \Fin_{A/} \times_{\Fin} \Fin_{/1}$,
		\item 
		We let 
		$\hl{\Part_A} \subseteq \Fact_A$ 
		denote the full subcategory spanned by the non-trivial partitions, 
		i.e. factorizations $A \xrightarrow{f} B \xrightarrow{g} 1$ 
		such that $f$ is surjective and neither $f$ nor $g$ are isomorphisms.
		\item 
		We write $\hl{\Part_A(\calO)} \coloneq  
		\Part_A \times_{\Fact_A} \Fact_A(\calO) \subseteq \Fact_A(\calO)$.
	\end{itemize}
\end{defn}

Fix a finite set $A \in \Fin$ and an active morphism 
$\mu : x \actarrow y$ in $\calO$
such that
$\langle x \rangle^{\circ} \simeq A$.
Composition defines a functor 
$\Fact_A(\calO) \longrightarrow  \{A\} \times_{\Fin} \Fun([1], \calO^{\act}) \times_{\Fin} \{1\}$.
Unwinding definitions we see that the fiber at $\mu$ of this functor 
is precisely the factorization category $\Fact_{\mu}(\calO)$.
Similarly the fiber at $\mu$ of the map
$\Part_A(\calO) \longrightarrow  \{A\} \times_{\Fin^{\simeq}}  \Fun([1], \calO^{\act}) \times_{\Fin^{\simeq}} \{1\}$
is the partition category $\Part_{\mu}(\calO)$.

\begin{obs}\label{lem:groupoid-lemma}
	Let $\calO \to \Fin_{\ast}$ be an $\infty$-operad such that $\calO^{=1}$ is an $\infty$-groupoid. 
	Then using the nonunital-unitary factorization factorization system
	one verifies that
	$\calO^{\act} \to \Fin_\ast^\act \simeq \Fin$ 
	is conservative.
	In particular for any $A \in \Fin$
	the natural inclusion induces an equivalence: 
	\[\{A\} \times_{\Fin^{\simeq}} \xN_1(\calO^{\act}) \times_{\Fin^{\simeq}} \{1\}
	\iso 
	\{A\} \times_{\Fin} \Fun([1],\calO^{\act}) \times_{\Fin} \{1\}\]
\end{obs}

Recall that every functor into an $\infty$-gropupoid is a cocartesian fibration \cite[3.3.1.8]{HTT}.
In particular for $\calO$ as in \cref{lem:groupoid-lemma} we have:
\begin{enumerate}
	\item 
	The map
	$\Fact_A(\calO) \to 	\{A\} \times_{\Fin} \Fun([1],\calO^{\act}) \times_{\Fin} \{1\}$
	is a cocartesian fibration.
	Its fibers are given by factorization categories which are all weakly contractible.
	It therefore induces an equivalence:
	\[|\Fact_A(\calO)| \iso  \{A\} \times_{\Fin} \Fun([1],\calO^{\act}) \times_{\Fin} \{1\} \simeq \{A\} \times_{\Fin^{\simeq}} \xN_1(\calO^{\act}) \times_{\Fin^{\simeq}} \{1\}\]
	\item 
	The map $\Part_A(\calO) \to 	\{A\} \times_{\Fin} \Fun([1],\calO^{\act}) \times_{\Fin} \{1\}$ 
	is a cocartesian fibration whose fibers are the partition categories. 
	We conclude that the fiber at $\mu \in \{A\} \times_{\Fin^{\simeq}} \xN_1(\calO^{\act}) \times_{\Fin^{\simeq}} \{1\}$ 
	of the composite
	\[|\Part_A(\calO)| \too \{A\} \times_{\Fin} \Fun([1],\calO^{\act}) \times_{\Fin} \{1\}  \simeq |\Fact_A(\calO)|\]
	is equivalent to the partition complex
	$\spart{\calO}{\mu}$.
\end{enumerate}

We record our findings in a corollary for future use.

\begin{cor}\label{cor:unstraightning-of-part}
	Let $\calO$ be an \operad{} such that $\calO^{=1}$ is an $\infty$-groupoid.
	Then for any $\mu \in \Fact_A(\calO)$ there is a canonical pullback square:
	\[\begin{tikzcd}
		{\spart{\calO}{\mu}} & {|\Part_A(\calO)|} \\
		{\{\mu\}} & {|\Fact_A(\calO)|}
		\arrow[from=1-2, to=2-2]
		\arrow[from=1-1, to=2-1]
		\arrow[from=1-1, to=1-2]
		\arrow[""{name=0, anchor=center, inner sep=0}, from=2-1, to=2-2]
		\arrow["\lrcorner"{anchor=center, pos=0.125}, draw=none, from=1-1, to=0]
	\end{tikzcd}\]
\end{cor}

\begin{const}
	Consider the functor
	$\Delta^{\op,\act}_{/\Fact_A(\calO)} \to \Delta^{\op,\act}_{/\calO^{\act}} $
	given by unstraightning over $\Delta^{\op,\act} \subseteq \Delta^{\op}$ the composite natural transformation
	\[\xN_{\bullet}(\Fact_A(\calO)) \longrightarrow \xN_{\bullet}(\Fun([2],\calO^{\act})) \longrightarrow \xN_{\bullet+2}(\calO^{\act})\]
	where the first map is induced by the inclusion 
	$\Fact_A(\calO) \to \Fun([2],\calO)$ 
	and the second map is the restriction along the active map 
	$[k+2] \to [k] \times [2]$
	defined by:
	\[\begin{tikzcd}
		0 & 1 & 2 & \cdots & {k+1} & {k+2} \\
		{(0,0)} & {(0,1)} & {(1,1)} & \cdots & {(k,1)} & {(k,2)}
		\arrow[from=2-3, to=2-4]
		\arrow[from=2-4, to=2-5]
		\arrow[from=2-5, to=2-6]
		\arrow[from=2-1, to=2-2]
		\arrow[from=2-2, to=2-3]
		\arrow[from=1-1, to=1-2]
		\arrow[from=1-2, to=1-3]
		\arrow[from=1-3, to=1-4]
		\arrow[from=1-4, to=1-5]
		\arrow[from=1-5, to=1-6]
		\arrow[maps to, from=1-6, to=2-6]
		\arrow[maps to, from=1-1, to=2-1]
		\arrow[maps to, from=1-2, to=2-2]
		\arrow[maps to, from=1-3, to=2-3]
		\arrow[maps to, from=1-5, to=2-5]
	\end{tikzcd}\]
	The same construction applied to the subcategory
	$\Part_A(\calO) \subseteq \Fact_A(\calO)$ 
	yields a functor $\Delta^{\op,\act}_{/\Part_A(\calO)} \to \Delta^{\op,\act}_{/\calO^{\act}}$.
\end{const}

\begin{lem}\label{cor:cat-of-simplices-desc-of-fact/part}
	In the following natural diagram, all squares are cartesian:
	\[\begin{tikzcd}
		{\Delta^{\op,\act}_{/\Part_A(\calO)}} & {\Delta^{\op,\act}_{/\Fact_A(\calO)}} & {\Delta^{\op,\act}_{/\calO^{\act}}} \\
		{\Delta^{\op,\act}_{/\Part_A}} & {\Delta^{\op,\act}_{/\Fact_A}} & {\Delta^{\op,\act}_{/\Fin}}
		\arrow[from=1-2, to=1-3]
		\arrow[from=1-2, to=2-2]
		\arrow[from=1-3, to=2-3]
		\arrow[""{name=0, anchor=center, inner sep=0}, from=2-2, to=2-3]
		\arrow[""{name=1, anchor=center, inner sep=0}, from=2-1, to=2-2]
		\arrow[from=1-1, to=2-1]
		\arrow[from=1-1, to=1-2]
		\arrow["\lrcorner"{anchor=center, pos=0.125}, draw=none, from=1-1, to=1]
		\arrow["\lrcorner"{anchor=center, pos=0.125}, draw=none, from=1-2, to=0]
	\end{tikzcd}\]
\end{lem}
\begin{proof}
	The diagram is given by unstraightning (over $\Delta^{\op,\act} \subseteq \Delta^{\op}$) 
	the following diagram of simplicial spaces:
	\[\begin{tikzcd}
		{\xN_{\bullet}(\Part_A(\calO))} & {\xN_{\bullet}(\Fact_A(\calO))} & {\xN_{\bullet+2}(\calO^{\act})} \\
		{\xN_{\bullet}(\Part_A)} & {\xN_{\bullet}(\Fact_A)} & {\xN_{\bullet+2}(\Fin)}
		\arrow[from=1-2, to=1-3]
		\arrow[from=1-2, to=2-2]
		\arrow[from=1-3, to=2-3]
		\arrow[from=2-2, to=2-3]
		\arrow[hook, from=2-1, to=2-2]
		\arrow[from=1-1, to=2-1]
		\arrow[hook, from=1-1, to=1-2]
	\end{tikzcd}\]
	It is straightforward to verify that 
	the induced map on vertical fibers in the right square is an equivalence.
	Consequently the right square is cartesian and since $\xN_{\bullet}$ preserves limits so is the left square.
\end{proof}

Note that $\Delta^{\op,\act}_{/\Fact_A}$ has a terminal object given by the trivial factorization hence \cref{cor:cat-of-simplices-desc-of-fact/part} in particular gives:
\[|\Fact_A(\calO)| \simeq |\Delta^{\op,\act}_{/\Fact_A(\calO)}| \simeq \colim\left( \Delta^{\op,\act}_{/\Fact_A} \xrightarrow{\xN^\xbar_\star(\calO)} \calS \right) \simeq  \xN^\xbar_{(A\to \pt)}(\calO) \]
Combining \cref{cor:unstraightning-of-part} and \cref{cor:cat-of-simplices-desc-of-fact/part} therefore gives the following.

\begin{cor}\label{cor:partition-complex-barwick}
	Let $\calO$ be an \operad{} such that $\calO^{=1}$ is an $\infty$-groupoid.
	Then for every multi-morphism 
	$\mu \colon (x_a \colon a\in A) \actarrow y $ in $\calO$ there is a canonical pullback square:
	\[\begin{tikzcd}
		{\spart{\calO}{\mu}} & {\colim \left(\Delta^{\op,\act}_{/\Part_A} \to \Delta^\op_\Fin \xrightarrow{\xN^\xbar_\star(\calO)} \calS \right)} \\
		{\{\mu\}} & {\xN^\xbar_{(A \to \pt)}(\calO)}
		\arrow[from=1-2, to=2-2]
		\arrow[from=1-1, to=2-1]
		\arrow[from=1-1, to=1-2]
		\arrow[""{name=0, anchor=center, inner sep=0}, from=2-1, to=2-2]
		\arrow["\lrcorner"{anchor=center, pos=0.125}, draw=none, from=1-1, to=0]
	\end{tikzcd}\]
\end{cor}

\begin{lem}\label{lem:rest-barwick-nerve}
	There is a canonical equivalence
	$\Delta^{\op,\act}_{/\Fin} \simeq \Delta^{1,\op,\act}_{\Fin}$ 
	under which the composite functor
	\[\Delta^{\op,\act}_{/\Fin}  \simeq \Delta^{1,\op,\act}_{\Fin} \longrightarrow \Delta^{\op}_{\Fin} \xrightarrow{\xN^\xbar_\star(\calO)} \calS \]
	is equivalent to the straightning of the left fibration
	$\Delta^{\op,\act}_{/\calO^{\act}} \to \Delta^{\op,\act}_{/\Fin}$.
\end{lem}
\begin{proof}
    Consider the following natural diagram
	\[\begin{tikzcd}
		{\Delta^{\op,\act}_{/\calO^{\act}}} & {\Delta^{\op}_{/\calO^{\act}}} &
		{\Un_{\Delta^{\op}\times\Fin_{\ast}}(\xN^{\otimes}_{\bullet}(\Env(\calO)))} & {\Un_{\Delta^{\op}_{\Fin}}(\xN^\xbar_\star(\calO))} \\
		{\Delta^{\op,\act}_{/\Fin}} & {\Delta^{\op}_{/\Fin}} & {\Un_{\Delta^{\op}\times\Fin_{\ast}}(\xN^{\otimes}_{\bullet}(\Env(\Fin_{\ast})))} & {\Delta^{\op}_{\Fin}} \\
		{\Delta^{\op,\act}\times\{1\}} & {\Delta^{\op}\times\{1\}} & {\Delta^{\op}\times\Fin_{\ast}} 
		\arrow[from=1-2, to=2-2]
		\arrow[from=1-4, to=2-4]
		\arrow[from=2-2, to=3-2]
		\arrow[""{name=1, anchor=center, inner sep=0}, from=3-2, to=3-3]
		\arrow[from=1-2, to=1-3]
		\arrow[from=1-3, to=1-4]
		\arrow[""{name=2, anchor=center, inner sep=0}, from=2-2, to=2-3]
		\arrow[from=2-3, to=3-3]
		\arrow[from=1-3, to=2-3]
		\arrow[from=2-3, to=2-4]
		\arrow[""{name=3, anchor=center, inner sep=0}, from=3-1, to=3-2]
		\arrow[from=2-1, to=3-1]
		\arrow[from=1-1, to=1-2]
		\arrow[from=1-1, to=2-1]
		\arrow[from=2-1, to=2-2]
	\end{tikzcd}\]
	the bottom squares as well as the left and middle outer vertical rectangles are cartesian by construction.
	It follows that the top left and top middle square are also cartesian. 
	The top right square is cartesian by \cref{obs:barwick=Lurie}.
	Pasting the top three squares we conclude that the top total composite  rectangle is also cartesian. 
	This rectangle also sits as the top total composite rectangle in the following diagram:
	\[\begin{tikzcd}
		{\Delta^{\op,\act}_{/\calO^{\act}}} & {\Un_{\Delta^{1,\op,\act}_{\Fin}}(\xN^\xbar_\star(\calO))} & {\Un_{\Delta^{1,\op}_{\Fin}}(\xN^\xbar_\star(\calO))} & {\Un_{\Delta^{\op}_{\Fin}}(\xN^\xbar_\star(\calO))} \\
		{\Delta^{\op,\act}_{/\Fin}} & {\Delta^{1,\op,\act}_{\Fin}} & {\Delta^{1,\op}_{\Fin}} & {\Delta^{\op}_{\Fin}} \\
		& {\Delta^{\op,\act}} & {\Delta^{\op}}
		\arrow[from=1-1, to=2-1]
		\arrow[from=1-4, to=2-4]
		\arrow[from=1-3, to=2-3]
		\arrow[from=1-3, to=1-4]
		\arrow[from=2-3, to=2-4]
		\arrow["\simeq", from=2-1, to=2-2]
		\arrow[from=2-2, to=2-3]
		\arrow[from=2-3, to=3-3]
		\arrow[from=3-2, to=3-3]
		\arrow[from=2-2, to=3-2]
		\arrow[from=1-2, to=1-3]
		\arrow[from=1-2, to=2-2]
		\arrow[from=1-1, to=1-2]
		\arrow[from=2-1, to=3-2]
	\end{tikzcd}\]
	where the bottom horizontal map in the top left square is an equivalence by unravelling definitions. 
	The bottom square and the top right square are pullbacks by construction. 
	Similarly the middle vertical rectangle is also cartesian and then by cancellation the middle top square is as well.
	We verified that all squares in the diagram are cartesian except possibly the top left square.
	Since we showed in the previous paragraph  that the top total composite rectangle is cartesian, 
	cancellation implies that the top left square is also cartesian.
	Finally since the bottom horizontal map in the top left square is an equivalence the same holds for the top horizontal map in that square.
\end{proof}

We are are finally in a position to prove \cref{thm:dendroidal-partition-complex}.

\begin{proof}[Proof of \cref{thm:dendroidal-partition-complex}]	
	We begin with the following cube:
	\[\begin{tikzcd}
		{\Omega^{\op}_{/\calO}} && {\Un_{\Delta^{1,\op}}(\xN^\dend_\star(\calO)) \overset{(\ref{obs:lurie=barwick=dendroidal})}{\simeq} \Un_{\Delta_{\Fin}^{1,\op}}(\xN^\xbar_\star(\calO))} \\
		& {\Omega^{\op,\act}_{/\calO}} && {\Delta^{\op,\act}_{/\calO^{\act}}} \\
		{\Omega^{\op}} && {\Delta^{1,\op}_{\Fin}} \\
		& {\Omega^{\op,\act}} && {\Delta^{1,\op,\act}_{\Fin}}
		\arrow[from=2-2, to=1-1]
		\arrow[from=2-4, to=1-3]
		\arrow[from=1-1, to=3-1]
		\arrow[from=4-2, to=3-1]
		\arrow[from=2-4, to=4-4]
		\arrow[from=1-3, to=3-3]
		\arrow[from=2-2, to=4-2]
		\arrow[from=3-3, to=3-1]
		\arrow[from=4-4, to=4-2]
		\arrow[from=2-4, to=2-2]
		\arrow[from=1-3, to=1-1]
		\arrow[from=4-4, to=3-3]
	\end{tikzcd}\]
	The back square and the left face of the cube are cartesian by construction. 
	The right face of the cube is cartesian by \cref{lem:rest-barwick-nerve} and thus the front face is also cartesian. 
	Let us now fix a set $A$ of size $k$ and consider the following diagram, 
	in which the front face of the previous diagram is embedded as the right most face:
	\[\begin{tikzcd}
	{\Un_{(\Omega^{\op,\act}_{/C_A})^{< |A|}}(\xN^\dend_\star(\calO))} && {\Un_{\Omega^{\op,\act}_{/C_A}}(\xN^\dend_\star(\calO))} && {\Omega^{\op,\act}_{/\calO}} \\
		& {\Delta^{\op,\act}_{/\Part_A(\calO)}} && {\Delta^{\op,\act}_{/\Fact_A(\calO)}} && {\Delta^{\op,\act}_{/\calO^{\act}}} \\
		{(\Omega^{\op,\act}_{/C_A})^{< |A|}} && {\Omega^{\op,\act}_{/C_A}} && {\Omega^{\op,\act}} \\
		& {\Delta^{\op,\act}_{/\Part_A}} && {\Delta^{\op,\act}_{/\Fact_A}} && {\Delta^{1,\op,\act}_{\Fin}}
		\arrow[from=2-2, to=1-1]
		\arrow[from=2-2, to=2-4]
		\arrow[from=1-1, to=1-3]
		\arrow[from=2-4, to=1-3]
		\arrow[from=1-1, to=3-1]
		\arrow[from=4-2, to=3-1]
		\arrow[from=2-2, to=4-2]
		\arrow[from=4-2, to=4-4]
		\arrow[from=2-4, to=4-4]
		\arrow[from=3-1, to=3-3]
		\arrow[from=4-4, to=3-3]
		\arrow[from=1-3, to=3-3]
		\arrow[from=3-3, to=3-5]
		\arrow[from=1-3, to=1-5]
		\arrow[from=1-5, to=3-5]
		\arrow[from=2-4, to=2-6]
		\arrow[from=2-6, to=1-5]
		\arrow[from=4-4, to=4-6]
		\arrow[from=2-6, to=4-6]
		\arrow[from=4-6, to=3-5]
	\end{tikzcd}\]
	The back faces of both cubes are cartesian by construction and the front faces are cartesian by \cref{cor:cat-of-simplices-desc-of-fact/part}.
	We showed in the previous paragraph that the right face of the right cube is cartesian
	and hence so are the left and right faces of the left cube.
	The bottom middle map $\Delta^{\op,\act}_{/\Fact_A} \to \Omega^{\op,\act}_{/C_A}$ preserves the terminal object hence the top middle map induces an equivalence on classifying spaces:
	\[\xN^\xbar_{(A \to \pt)}(\calO) \simeq \colim\left(\Delta^{\op,\act}_{/\Fact_A} \xrightarrow{\xN^\xbar_\star(\calO)} \calS \right) \simeq \colim\left(\Omega^{\op,\act}_{/C_A} \xrightarrow{\xN^\dend_\star(\calO)} \calS \right) \simeq \xN^\dend_{C_A}(\calO) \]
	Examaning the top left square we see that via this equivalence the latching map of Barwick's model, i.e. the right vertical map in the square from \cref{cor:partition-complex-barwick},
	\[\colim \left(\Delta^{\op,\act}_{/\Part_A} \to \Delta^\op_\Fin \xrightarrow{\xN^\xbar_\star(\calO)} \calS \right) \too \xN^\xbar_{(A \to \pt)}(\calO)\]
    gets identified with the composite
	\[\colim \left(\Delta^{\op,\act}_{/\Part_A} \to (\Omega^{\op,\act}_{/C_A})^{<|A|} \xrightarrow{\xN^\dend_\star(\calO)} \calS \right) \too \colim \left((\Omega^{\op,\act}_{/C_A})^{<|A|} \xrightarrow{\xN^\dend_\star(\calO)} \calS \right)\too \xN^\dend_{C_A}(\calO).\]
    To conclude it thus suffices to show that the first map is an equivalence.
    Since $\calO$ is non-unital we have $\xN^\dend_T(\calO) = \emptyset$ for any tree $T$ whose set of leaves is non-empty.
    Hence we may replace $(\Omega^{\op,\act}_{C_A/})^{<|A|}$ 
    with its open counterpart $({\Omega}^{\circ,\op,\act}_{/C_A})^{<|A|}$.
    Finally, in \cite{heuts2021partition} it is shown that $\Delta^{\circ,\op,\act}_{/\Part_A} \to (\Omega_{/C_A}^{\op,\act})^{<|A|}$
    is cofinal.
\end{proof}

%% file: sections/appendix.tex
\appendix
\section{Appendix}
\begin{lem}\label{lem:under-over-vs-over-under}
	Let $\calC$ be an $\infty$-category and let $f: c_0 \to c_1$ be a morphism in $\calC$. Then there's a canonical equivalence of $\infty$-categories
    $\calC_{c_0//f} \simeq \calC_{f//c_1}$.
\end{lem}
\begin{proof}
	We show that there is a canonical isomorphism of simplicial sets $\calC_{c_0//f} = \calC_{f//c_1}$. An $n$-simplex in $\calC_{c_0//f}$ is equivalently an $(n+1)$-simplex of $\calC_{c_0/}$ whose last vertex is $f$. The latter is equivalent to an $(n+2)$-simplex $\sigma : \Delta^{n+2} \to \calC$ with $\sigma|_{\Delta^{\{0,n+2\}}} = f$. On the other hand an $n$-simplex in $\calC_{f//c_1}$ is equivalent to an $(n+1)$-simplex of $\calC_{/c_1}$ whose first vertex is $f$. The latter is equivalent to an $(n+2)$-simplex $\tau : \Delta^{n+2} \to \calC$ with $\tau|_{\Delta^{\{0,n+2\}}} = f$.
\end{proof}

\begin{lem}\label{lem: double slice}
	Let $\calC \to \calD$ be a functor and let $d \in \calD$. Then the natural projection 
    $\calC_{d/} \to \calC$
    induces for every $\big(c,d \overset{\alpha}{\to} f(c)\big) \in \calC_{d/}$ 
    an equivalence of categories 
    $(\calC_{d/ })_{\alpha/} \simeq \calC_{c/}$.
\end{lem}
\begin{proof}
	Apply \cite[Proposition~2.1.2.5.]{HTT} with $A\coloneq \emptyset$, $B\coloneq \Delta^0$, $S\coloneq  \calC_{d/}$, $T\coloneq \calC$, $\pi: \calC_{d/} \to \calC$ and $p: \Delta^0 \to \calC_{d/}$ the map corresponding to the edge $\Delta^0 \ast \Delta^0 \simeq \Delta^1 \overset{\alpha}{\to}  \calC$.
\end{proof}

\begin{lem}\label{lem:fully-faithful-on-clice-from-faithful}
    Let $\calC \subseteq \calD$ be a faithful subcategory. For every $x \in \calC$ the natural functor 
    $\calC_{x/} \mono \calD_{x/}$ is fully faithful.
\end{lem}
\begin{proof}
	Let $f: x \to y \in \calC_{x/}$ and consider the natural square
	\[\begin{tikzcd}
	{\Map_{\calC}(y,z)} & { \Map_{\calD}(y,z)} \\
	{\Map_{\calC}(x,z)} & {\Map_{\calD}(x,z)}
	\arrow[hook, from=2-1, to=2-2]
	\arrow["{\circ f}", from=1-2, to=2-2]
	\arrow[hook, from=1-1, to=1-2]
	\arrow["{\circ f}", from=1-1, to=2-1]
	\end{tikzcd}\]
	The horizontal maps are monomorphisms and thus the induced map on vertical fibers at every $g:x \to z \in \calC_{x/}$
	\[\Map_{\calC_{x/}}(f,g) \simeq \Map_{\calC}(y,z) \times_{\Map_{\calC}(x,z)} \{g\}  \to \Map_{\calD}(y,z) \times_{\Map_{\calD}(x,z)} \{g\}  \simeq \Map_{\calD}(f,g)\]
	is an equivalence.
\end{proof}

\begin{lem}\label{lem:lifting-property-equivalent-to-local-in-slice}
	Let $\calC$ be an $\infty$-categry and let $S$ be a collection of arrows in $\calC$. Then $f \in S^{\perp}$ if and only if the object $f: x \to y \in \calC_{/y}$ is $S$-local.
\end{lem}
\begin{proof}
	By definition $f \in S^{\perp}$ if and only if for every $s: u \to v \in S$ the following square
	\[\begin{tikzcd}
	{\Map_{\calC}(v,x)} & {\Map_{\calC}(v,y)} \\
	{\Map_{\calC}(u,x)} & {\Map_{\calC}(u,y)}
	\arrow["{\circ s}", from=1-2, to=2-2]
	\arrow["{\circ s}", from=1-1, to=2-1]
	\arrow["{f \circ}", from=1-1, to=1-2]
	\arrow["{f \circ}", from=2-1, to=2-2]
	\end{tikzcd}\]
	is cartesian. The latter holds if and only if for every $g : v \to y$ and every $s : u \to v \in S$ the induced map on horizontal fibers 
	\[ \Map_{\calC_{/y}}(g,f) \simeq  \Map_{\calC}(v,x) \times_{\Map_{\calC}(v,y)} \{g\} \longrightarrow \Map_{\calC}(u,x) \times_{\Map_{\calC}(u,y)} \{ g \circ s \} \simeq \Map_{\calC_{/y}}(g \circ s,f)\]
	is an equivalence. By definition this holds if and only if $f:x \to y \in \calC_{/y}$ is $S$-local.
\end{proof}

\begin{lem}[{\cite[Lemma 5.2.8.6]{HTT}}]\label{lem:cancellation}
	Let $(\calL,\calR)$ be an factorization system on $\calC$ and let $x \overset{f}{\to} y \overset{g}{\to} z$ be a pair of composable morphisms in $\calC$. Then the following holds
	\begin{enumerate}
		\item If $f \circ g \in \calL$ and $g \in \calL$ then $f \in \calL$
		\item If $f \circ g \in  \calR$ and $f \in \calR$ then $g \in \calR$.
	\end{enumerate}
\end{lem}

\begin{prop}\label{prop:slice-category-factorization-adjoint}
	Let $\calC$ be an $\infty$-categry equipped with an factorization system $(\calL,\calR)$ and let $x \in \calC$ be an object. We have
	\begin{enumerate}
		\item The inclusion $\calC^{\calR}_{/ x} \coloneq   \calR  \times_{\calC} \{x\}  \to \calC_{/x}$ admits a left adjoint. 
		
		\item The inclusion $\calC^{\calL}_{x /} \coloneq  \{x\} \times_{\calC} \calL  \to \calC_{x/}$ admits a right adjoint.
	\end{enumerate}
\end{prop}
\begin{proof}
	(1) By \cite[Lemma 5.2.8.19]{HTT} the inclusion $\calR \mono \calC^{\Delta^1}$ admits a left adjoint which fixes the target of all arrows. Taking the fiber at $x$ of the evaluation at the target functor $t: \calC^{\Delta^1} \to \calC$ produces a left adjoint to the inclusion 
	$$\calC^{\calR}_{x/} \simeq  \calR  \times_{\calC} \{x\} \mono  \calC^{\Delta^1} \times_{\calC}  \{x\} \simeq \calC_{x/}$$
	(2) follows from (1) by taking opposites.
\end{proof}

\begin{prop}\label{prop:factorization-system-on-slice}
	Let $\calC$ be an $\infty$-categry and let $(\calL,\calR)$ be an factorization system on $\calC$. Then, for every $x \in \calC$ there's a (unique) induced factrization system $(\calL^{\prime},\calR^{\prime})$ on $\calC_{x/}$ defined by the following property 
	\begin{itemize}
		\item A morphism $\alpha: f \to f^{\prime}$ represented by a diagram
		\[\begin{tikzcd}
		& x \\
		y && {y^{\prime}}
		\arrow["f"', from=1-2, to=2-1]
		\arrow["{f^{\prime}}", from=1-2, to=2-3]
		\arrow["\alpha"', from=2-1, to=2-3]
		\end{tikzcd}\]
		corresponds to a morphism in $\calL^{\prime}$ (resp. $\calR^{\prime}$) if and only if $\alpha$ is in $\calL$ (resp. $\calR$).
	\end{itemize}
\end{prop}
\begin{proof}
	By \cite[Corollary~5.2.8.18.]{HTT} there's a factorization system on $\calC^{\Delta^1}$ characterized by the property that a morphism $\alpha: g \to g'$ is in $\calL'$ if and only if $s(\alpha)$ and $t(\alpha)$ are both in $\calL$. Here $s$ and $t$ denote respectively the source and target maps $\calC^{\Delta^1} \to \calC$. Note that we have a pullback square
	\[\begin{tikzcd}
	{\calC_{x/}} & {\calC^{\Delta^1}} \\
	{\{x\}} & \calC
	\arrow[from=2-1, to=2-2]
	\arrow["s", from=1-2, to=2-2]
	\arrow[from=1-1, to=1-2]
	\arrow[from=1-1, to=2-1]
	\arrow["\lrcorner"{anchor=center, pos=0.125}, draw=none, from=1-1, to=2-2]
	\end{tikzcd}\]
	It follows from \cite[Proposition~5.2.]{RH-algebraic} that the factorization system restricts to $\calC_{x/}$. 
\end{proof}

\begin{prop}\label{prop:factorization-adjoint-on-comma-category}
	Let $\calD$ be an $\infty$-categry equipped with a factorization system $(\calD^{L},\calD^{R})$. Let $\calC \subseteq \calD$ be a replete subcategory and let $d_0 \in \calD$. Denote $\calC^L \coloneq  \calD^L \cap \calC$ and $\calC^R \coloneq  \calD^R \cap \calC$. Suppose the following conditions holds:
	\begin{enumerate}[(1)]
		\item If $\alpha : c^{\prime} \to c$ is a morphism in $\calC$ then $\alpha^L: c^{\prime} \to \Lambda(\alpha)$ is also a morphism in $\calC$. 
		\item For every $c \in \calC$ the natural functor $\calC^R_{/c} \to \calD^{R}_{/c}$ is essentially surjective. 
	\end{enumerate}
 Then the natural fully faithful inclusion
	\[ \calD^L_{d_0/} \times_{\calD^L} \calC^L \mono  \calD_{d_0/} \times_{\calD} \calC \]
	admits a right adjoint.
\end{prop}
\begin{proof}
	It suffices to show that every object $(d_0 \overset{\alpha}{\to} c,c) \in \calD_{d_0/} \times_{\calD} \calC$ admits a coreflection. The $(\calD^L,\calD^R)$-factorization of $\alpha$ yields the following diagram
	\[\begin{tikzcd}
	& d_0 \\
	{\Lambda(\alpha)} && c
	\arrow["\alpha", from=1-2, to=2-3]
	\arrow["{\alpha^L}"', tail, from=1-2, to=2-1]
	\arrow["{\alpha^R}"', squiggly, from=2-1, to=2-3]
	\end{tikzcd}\]
	By (2) we know that $\alpha^R : \Lambda(\alpha) \rightsquigarrow c \in \calC^R_{/c}$. It follows that the above diagram defines a morphism in $\calD_{d_0/} \times_{\calD} \calC$ of the form
	\[\tag{$\star$} \overrightarrow{\alpha}^R : \big(d_0 \overset{\alpha^L}{\to} \Lambda(\alpha),\Lambda(\alpha)\big) \to (d_0 \overset{\alpha}{\to} c,c\big)\]
	We claim that $\overrightarrow{\alpha}^R $ exhibits $(d_0 \overset{\alpha^L}{\to} \Lambda(\alpha),\Lambda(\alpha))$ as a $\calD^L_{d_0/} \times_{\calD^L} \calC^L$-coreflection of $(d_0 \overset{\alpha}{\to} c,c)$. Equivalently this is the claim that the natural map given by composing with $\alpha^R$ 
	\begin{equation*}\label{diag:comparison-map-reflection-L-R}
	    \alpha^R \circ (-) : \Map_{\calD^L_{d_0/} \times_{\calD^L} \calC^L}(\varphi,\alpha^L) \to \Map_{\calD_{d_0/} \times_{\calD} \calC}(\varphi,\alpha)
	\end{equation*} 
 	
	is an equivalence for every $\varphi  \in \calD^L_{d_0/} \times_{\calD^L} \calC^L$. To see thise first note that for every $\varphi : d_0 \rightarrowtail c^{\prime}$ in $\calD^L$ with $c^{\prime} \in \calC$ there are canonical equivalences
	\begin{align*}
	\Map_{\calD_{d_0/} \times_{\calD} \calC}(\varphi,\alpha) & \simeq \Map_{\calD_{d_0/}}(\varphi,\alpha) \times_{\Map_{\calD}(c^{\prime},c)} \Map_{\calC}(c^{\prime},c)\\
	& \simeq   \{\alpha\} \times_{\Map_{\calD}(d_0,c)} \Map_{\calD}(c^{\prime},c) \times_{\Map_{\calD}(c^{\prime},c)} \Map_{\calC}(c^{\prime},c)\\
	& \simeq \{\alpha\} \times_{\Map_{\calD}(d_0,c)}  \Map_{\calC}(c^{\prime},c)
	\end{align*}
	Similarly there's a canonical equivalence
	\[ \Map_{\calD^L_{d_0/} \times_{\calD^L} \calC^L}(\varphi,\alpha)  \simeq \{\alpha^L\} \times_{\Map^L_{\calD}(d_0,\Lambda(\alpha))}  \Map^L_{\calC}(c^{\prime},\Lambda(\alpha)) \]
	Tracing through these equivalences identifies \eqref{diag:comparison-map-reflection-L-R}{($\star$)} with the induced map on fibers in the following square
	\[\begin{tikzcd}
	{\Map^L_{\calC}(c^{\prime},\Lambda(\alpha))} && {\Map_{\calC}(c^{\prime},c)} \\
	{\Map^L_{\calD}(d_0,\Lambda(\alpha))} && {\Map_{\calD}(d_0,c)}
	\arrow["{\alpha^R \circ}", from=1-1, to=1-3]
	\arrow["{\circ \varphi}", from=1-3, to=2-3]
	\arrow["{\alpha^R \circ}"', from=2-1, to=2-3]
	\arrow["{\circ \varphi}"', from=1-1, to=2-1]
	\end{tikzcd}\]
	where the fibers are taken over $\alpha^{L} \in \Map^L_{\calD}(d_0,\Lambda(\alpha))$ and its image $\alpha \in  \Map_{\calD}(d_0,\Lambda(\alpha))$ respectively. It follows that the comparison map
	\[\{\alpha^L\} \times_{\Map^L_{\calD}(d_0,\Lambda(\alpha))}  \Map^L_{\calC}(c^{\prime},\Lambda(\alpha))  \to \{\alpha\} \times_{\Map_{\calD}(d_0,c)}  \Map_{\calC}(c^{\prime},c)\] 
	is an equivalence if and only if for every morphism $\beta : c^{\prime} \to c$ in $\calC$, that sits in a square of the form
	\[\begin{tikzcd}
	d_0 & {\Lambda(\alpha)} \\
	c^{\prime} & c
	\arrow["{\alpha^L}", tail, from=1-1, to=1-2]
	\arrow["\varphi"', tail, from=1-1, to=2-1]
	\arrow["{\alpha^R}", squiggly, from=1-2, to=2-2]
	\arrow["\beta"', from=2-1, to=2-2]
	\arrow[dashed,from=2-1, to=1-2]
	\end{tikzcd}\]
	the space of dashed arrows in $\calC$ filling the diagram is contractible. Note however that by cancellation (see \cref{lem:cancellation}) any lift in the above diagram (whether or not it exists in $\calC$) must necessarily be in $\calD^L$ and must therefore be equivalent to $\beta^L$ by uniqueness of factorizations. Finally by (1) we have that $\beta^L$ is a morphism in $\calC$ so we're done.
\end{proof}

%% file: literature.bib
@misc{RH-cartesian,
      title={Free algebras through Day convolution}, 
      author={Hongyi Chu and Rune Haugseng},
      year={2021},
      eprint={2006.08269},
      archivePrefix={arXiv},
      primaryClass={math.AT}
}

@misc{RH-algebraic,
	title={Homotopy-coherent algebra via Segal conditions}, 
	author={Hongyi Chu and Rune Haugseng},
	year={2021},
	eprint={1907.03977},
	archivePrefix={arXiv},
	primaryClass={math.AT}
}

@article{goppl,
	title={A spectral sequence for spaces of maps between operads},
	author={G{\"o}ppl, Florian},
	journal={arXiv preprint arXiv:1810.05589},
	year={2018}
}

@article{getzler-jones,
	title={Operads, homotopy algebra and iterated integrals for double loop spaces},
	author={Getzler, Ezra and Jones, John DS},
	journal={arXiv preprint hep-th/9403055},
	year={1994}
}

@article{Dend-Barwick,
	title={Two models for the homotopy theory of infinity operads},
	author={Chu, Hongyi and Haugseng, Rune and Heuts, Gijs},
	journal={Journal of Topology},
	volume={11},
	number={4},
	pages={857--873},
	year={2018},
	publisher={Wiley Online Library}
}

@article{heuts2021partition,
  title={Partition complexes and trees},
  author={Heuts, Gijs and Moerdijk, Ieke},
  journal={arXiv preprint arXiv:2112.08043},
  year={2021}
}

@article{heuts,
  title={On the equivalence between Lurie's model and the dendroidal model for infinity-operads},
  author={Heuts, Gijs and Hinich, Vladimir and Moerdijk, Ieke},
  journal={Advances in Mathematics},
  volume={302},
  pages={869--1043},
  year={2016},
  publisher={Elsevier}
}

@article{operator-cat,
	title={From operator categories to higher operads},
	author={Barwick, Clark},
	journal={Geometry \& Topology},
	volume={22},
	number={4},
	pages={1893--1959},
	year={2018},
	publisher={Mathematical Sciences Publishers}
}

@article{getzler,
  title={Operads, homotopy algebra and iterated integrals for double loop spaces},
  author={Getzler, Ezra and Jones, John DS},
  journal={arXiv preprint hep-th/9403055},
  year={1994}
}

@article{markl,
  title={A compactification of the real configuration space as an operadic completion},
  author={Markl, Martin},
  journal={arXiv preprint hep-th/9608067},
  year={1996}
}

@article{salvatore,
  title={Configuration spaces with summable labels},
  author={Salvatore, Paolo},
  journal={arXiv preprint math/9907073},
  year={1999}
}

@article{moerdijk2007dendroidal,
  title={Dendroidal sets},
  author={Moerdijk, Ieke and Weiss, Ittay},
  journal={Algebraic \& Geometric Topology},
  volume={7},
  number={3},
  pages={1441--1470},
  year={2007},
  publisher={Mathematical Sciences Publishers}
}

@article{rune-joachim,
	title={Infinity operads as symmetric monoidal infinity categories},
	author={Haugseng, Rune and Kock, Joachim},
	journal={arXiv preprint arXiv:2106.12975},
	year={2021}
}

@InProceedings{Quillen-A,
	author="Quillen, Daniel",
	editor="Bass, H.",
	title="Higher algebraic K-theory: I",
	booktitle="Higher K-Theories",
	year="1973",
	publisher="Springer Berlin Heidelberg",
	address="Berlin, Heidelberg",
	pages="85--147",
	isbn="978-3-540-37767-2"
}

@book{HTT,
	author = {Jacob Lurie},
	doi = {doi:10.1515/9781400830558},
	url = {https://doi.org/10.1515/9781400830558},
	title = {Higher Topos Theory (AM-170)},
	year = {2009},
	publisher = {Princeton University Press},
	ISBN = {9781400830558}
}

@unpublished{HA,
	author = {Lurie, Jacob},
	title = {Higher algebra},
	NOTE = {{http://www.math.harvard.edu/~lurie/}},
}

@misc{Pries,
      title={The Classification of Two-Dimensional Extended Topological Field Theories}, 
      author={Christopher J. Schommer-Pries},
      year={2014},
      eprint={1112.1000},
      archivePrefix={arXiv},
      primaryClass={math.AT}
}

@article{Day,
title = {Monoidal Bicategories and Hopf Algebroids},
journal = {Advances in Mathematics},
volume = {129},
number = {1},
pages = {99-157},
year = {1997},
issn = {0001-8708},
doi = {https://doi.org/10.1006/aima.1997.1649},
url = {https://www.sciencedirect.com/science/article/pii/S0001870897916492},
author = {Brian Day and Ross Street}
}

@article{Part1,
author = {Arone, Gregory and Dwyer, W.},
year = {2001},
month = {01},
pages = {},
title = {Partition Complexes, Tits Buildings and Symmetric Products},
volume = {82},
journal = {Proceedings of The London Mathematical Society - PROC LONDON MATH SOC},
doi = {10.1112/S0024611500012715}
}

@book{Hyperplanes,
	author = {Peter Orlik and Hiroaki Terao},
	doi = {https://doi.org/10.1007/978-3-662-02772-1},
	title = {Arrangements of Hyperplanes},
	year = {1992},
	publisher = {Springer, Berlin, Heidelberg},
	ISBN = {978-3-540-55259-8}
}

@article{Ching,
	author = {Ching, Michael},
	journal = "{Geometry and Topology}",
	language = {eng},
	pages = {833-933},
	publisher = {University of Warwick, Mathematics Institute, Coventry; Mathematical Sciences Publishers, Berkeley},
	title = {Bar constructions for topological operads and the Goodwillie derivatives of the identity.},
	url = {http://eudml.org/doc/129552},
	volume = {9},
	year = {2005},
}

@article{Fresse,
	author = {Fresse, Benoit},
	year = {2003},
	month = {03},
	pages = {},
	title = {Koszul duality of operads and homology of partition posets},
	volume = {346},
	isbn = {9780821832851},
	doi = {10.1090/conm/346/06287}
}

@article{Joyal,
	author = {Andre, Joyal},
	year = {2002},
	month = {11},
	pages = {207-222},
	title = {Quasi-categories and Kan complexes},
	volume = {175},
	journal = {Journal of Pure and Applied Algebra},
	doi = {10.1016/S0022-4049(02)00135-4}
}
